\documentclass[12pt,a4paper,hidelinks]{report}
\let\openright=\clearpage

\usepackage[a-2u]{pdfx}

\usepackage[utf8x]{inputenc}

\usepackage{lmodern}

\usepackage{amsmath}        
\usepackage{amsfonts}       
\usepackage{amssymb} 
\usepackage{amsthm}         
\usepackage{bbding}         
\usepackage{bm}             
\usepackage{graphicx}       
\usepackage{fancyvrb}       
\usepackage[nottoc]{tocbibind} 
\usepackage{dcolumn}        
\usepackage{booktabs}       
\usepackage{paralist}       
\usepackage[usenames]{xcolor}  

\usepackage{tikz-cd}		
\usepackage{enumitem}
\setlist[1]{itemsep=0em} 	

\usepackage{array}

\def\ThesisTitle{Semigroup-valued metric spaces}

\def\ThesisAuthor{Matěj Konečný}

\def\Department{Department of Applied Mathematics}

\def\DeptType{Department}

\def\Supervisor{Mgr. Jan Hubi\v cka, Ph.D.}

\def\Consultant{prof. RNDr. Jaroslav Nešetřil, DrSc.}

\def\SupervisorsDepartment{Department of Applied Mathematics}
\def\ConsultantsDepartment{Computer Science Institute}

\newcommand{\shrug}[1][]{%
\begin{tikzpicture}[baseline,x=0.8\ht\strutbox,y=0.8\ht\strutbox,line width=0.125ex,#1]
\def\arm{(-2.5,0.95) to (-2,0.95) (-1.9,1) to (-1.5,0) (-1.35,0) to (-0.8,0)};
\draw \arm;
\draw[xscale=-1] \arm;
\def\headpart{(0.6,0) arc[start angle=-40, end angle=40,x radius=0.6,y radius=0.8]};
\draw \headpart;
\draw[xscale=-1] \headpart;
\def\eye{(-0.075,0.15) .. controls (0.02,0) .. (0.075,-0.15)};
\draw[shift={(-0.3,0.8)}] \eye;
\draw[shift={(0,0.85)}] \eye;
\draw (-0.1,0.2) to [out=15,in=-100] (0.4,0.95); 
\end{tikzpicture}}

\def\Dedication{%
\begin{center}
This thesis is dedicated to everyone you hate.
\vskip 1cm
\Huge \shrug
\end{center}

}

\def\Abstract{%
The structural Ramsey theory is a field on the boundary of combinatorics and model theory with deep connections to topological dynamics. Most of the known Ramsey classes in finite binary symmetric relational language can be shown to be Ramsey by utilizing a variant of the shortest path completion (e.g. Sauer's $S$-metric spaces, Conant's generalised metric spaces, Braunfeld's $\Lambda$-ultrametric spaces or Cherlin's metrically homogeneous graphs). In this thesis we explore the limits of the shortest path completion. We offer a unifying framework --- semigroup-valued metric spaces --- for all the aforementioned Ramsey classes and study their Ramsey expansions and EPPA (the extension property for partial automorphisms). Our results can be seen as evidence for the importance of studying the completion problem for amalgamation classes and have some further applications (such as the stationary independence relation).

As a corollary of our general theorems, we reprove results of Hubi\v cka and Ne\v set\v ril on Sauer's $S$-metric spaces, results of Hubi\v cka, Ne\v set\v ril and the author on Conant's generalised metric spaces, Braunfeld's results on $\Lambda$-ultrametric spaces and the results of Aranda et al. on Cherlin's primitive 3-constrained metrically homogeneous graphs. We also solve several open problems such as EPPA for $\Lambda$-ultrametric spaces, $S$-metric spaces or Conant's generalised metric spaces.

Our framework seems to be universal enough that we conjecture that every pri\-mitive strong amalgamation class of complete edge-labelled graphs with finitely many labels is in fact a class of semigroup-valued metric spaces.
}

\def\Keywords{%
{metric space}, {semigroup}, {structural Ramsey theory}, {homogeneous structure}, {Ramsey expansion}, {EPPA}
}

\hypersetup{unicode}
\hypersetup{breaklinks=true}
\hypersetup{hidelinks}

\makeatletter
\def\@makechapterhead#1{
  {\parindent \z@ \raggedright \normalfont
   \Huge\bfseries \thechapter. #1
   \par\nobreak
   \vskip 20\p@
}}
\def\@makeschapterhead#1{
  {\parindent \z@ \raggedright \normalfont
   \Huge\bfseries #1
   \par\nobreak
   \vskip 20\p@
}}
\makeatother

\theoremstyle{plain}
\newtheorem{theorem}{Theorem}
\numberwithin{theorem}{chapter}

\newtheorem{lemma}[theorem]{Lemma}
\newtheorem{claim}[theorem]{Claim}
\newtheorem{prop}[theorem]{Proposition}
\newtheorem{corollary}[theorem]{Corollary}
\newtheorem{observation}[theorem]{Observation}

\theoremstyle{definition}
\newtheorem{definition}[theorem]{Definition}
\newtheorem{example}[theorem]{Example}
\newtheorem{question}{Question}
\newtheorem{conjecture}{Conjecture}

\theoremstyle{remark}

\newtheorem{remark}[theorem]{Remark}

\DefineVerbatimEnvironment{code}{Verbatim}{fontsize=\small, frame=single}

\renewcommand{\restriction}{\mathord{\upharpoonright}}

\newcommand{\Forb}{\mathop{\mathrm{Forb}}\nolimits}
\newcommand{\str}[1]{\mathbf{#1}}
\def\Fraisse{Fra\"{\i}ss\' e}

\def\rel#1#2{R_{\mathbf{#1}}^{#2}}
\def\func#1#2{F_{\mathbf{#1}}^{#2}}
\newcommand{\arity}[1]{a(#1)}
\def\dom{\mathop{\mathrm{Dom}}\nolimits}
\def\range{\mathop{\mathrm{Range}}\nolimits}
\def\Age{\mathop{\mathrm{Age}}\nolimits}

\newcommand{\overbar}[1]{\mkern 1.5mu\overline{\mkern-1.5mu#1\mkern-1.5mu}\mkern 1.5mu}

\def\Aut{\mathop{\mathrm{Aut}}\nolimits}
\def\acts{\curvearrowright}

\def\Aclass{\mathcal A^\delta_{K_1,K_2,C}}
\def\Fclass{\mathcal F^\delta_{K_1,K_2,C}}

\def\Semig{{\mathfrak M}}
\def\ESemig{\Semig=(M,\oplus,\mleq)}
\def\Block{{\mathcal B}}
\def\mleq{\preceq}
\def\nmleq{\not\mleq}
\def\mgeq{\succeq}
\def\nmgeq{\not\succeq}
\def\mgt{\succ}
\def\mlt{\prec}

\def\meet{\wedge}

\def\mus{\mathop{\mathrm{mus}}\nolimits}

\def\bigfoplus{\mathop{\mathrm{\bigoplus}}\limits}

\def\func#1#2{F_{\mathbf{#1}}^{#2}}
\def\nbfunc#1#2{F_{#1}^{#2}}

\def\pocs{partially ordered commutative semigroup}

\def\MF{\mathcal M_\Semig \cap \Forb\left(\mathcal F\right)}
\def\Mstarleq{\mathcal M^{\star,\leq}_\Semig}
\def\Mstar{\mathcal M^{\star}_\Semig}
\def\Mstarf{\mathcal M^{\star}_{\Semig}\cap \Forb(\mathcal F)}
\def\Hstarf{\mathcal H^\star_{\Semig}\cap \Forb(\mathcal F)}
\def\Hstarleq{\mathcal H^{\star,\leq}_{\Semig}}
\def\Hstarleqf{\mathcal H^{\star,\leq}_{\Semig}\cap \Forb(\mathcal F)}

\def\classstarf{\Mstar\cap \Forb(\mathcal F)}
\def\classstarleqf{\Mstarleq\cap \Forb(\mathcal F)}
\def\GstarS{\mathcal G^\star_{\Semig}\cap \Forb(\mathcal F)}

\def\Ind#1#2{#1\setbox0=\hbox{$#1x$}\kern\wd0\hbox to 0pt{\hss$#1\mid$\hss}
\lower.9\ht0\hbox to 0pt{\hss$#1\smile$\hss}\kern\wd0}
\def\Notind#1#2{#1\setbox0=\hbox{$#1x$}\kern\wd0\hbox to
0pt{\mathchardef\nn="0236\hss$#1\nn$\kern1.4\wd0\hss}\hbox 
to 0pt{\hss$#1\mid$\hss}\lower.9\ht0
\hbox to 0pt{\hss$#1\smile$\hss}\kern\wd0}
\def\ind{\mathop{\mathpalette\Ind{}}}

\def\tp{\mathop{\mathrm{tp}}\nolimits}

\def\magicsemig{\Semig_{M,C}^\delta}

\def\ExSemig{{\mathfrak D}}

\begin{document}


\pagestyle{empty}
\hypersetup{pageanchor=false}
\begin{center}

\vspace{-8mm}
\vfill

{\bf\Large THESIS (Unofficial version)}

\vfill

{\LARGE\ThesisAuthor}

\vspace{15mm}

{\LARGE\bfseries\ThesisTitle}

\vfill

\Department \\ Faculty of Mathematics and Physics \\ Charles University

\vfill

\begin{tabular}{rl}

Supervisor of the thesis: & \Supervisor \\
\noalign{\vspace{2mm}}
Consultant: & \Consultant
\end{tabular}

\vfill

\end{center}

\openright

\noindent
\Dedication

\newpage


\openright

\vbox to 0.5\vsize{
\setlength\parindent{0mm}
\setlength\parskip{5mm}

Title:
\ThesisTitle

Author:
\ThesisAuthor

\DeptType:
\Department

Supervisor:
\Supervisor, \SupervisorsDepartment

Consultant:
\Consultant, \ConsultantsDepartment

Abstract:
\Abstract

Keywords:
\Keywords

\vss}

\newpage

\openright
\pagestyle{plain}
\pagenumbering{arabic}
\setcounter{page}{1}


\renewcommand{\baselinestretch}{0.995}\normalsize
\tableofcontents
\renewcommand{\baselinestretch}{1.0}\normalsize
\chapter{Introduction}\label{ch:preface}
In~2007 Ne\v set\v ril~\cite{Nevsetvril2007} proved that the class of all linearly ordered finite metric spaces is Ramsey (see Section~\ref{subsec:category}). For graphs (and isometric embeddings) a similar result was obtained by Dellamonica and R\"odl~\cite{Dellamonica2012} in 2012; it also follows from a more general result of Hubi\v cka and Ne\v set\v ril~\cite{Hubicka2016}. Ma\v su\-lo\-vi\'c~\cite{masulovic2016pre} gave a simpler proof by a reduction to partially ordered sets which are known to be Ramsey~\cite{Nevsetvril1984,Trotter1985}. Solecki~\cite{solecki2005} and Vershik~\cite{vershik2008} independently proved that the class of all finite metric spaces has EPPA (the extension property for partial automorphisms, see Definition~\ref{defn:eppa}).

Sauer~\cite{Sauer2013b} in 2013 classified the sets $S\subseteq \mathbb R^{\geq 0}$ for which there is a universal homogeneous complete separable metric space with distances from $S$. Ramsey expansions of Sauer's $S$-metric spaces were then fully determined by Hubi\v cka and Ne\v set\v ril~\cite{Hubicka2016} (extending partial results by Nguyen Van Th\'e~\cite{The2010}).

Conant~\cite{Conant2015} studied EPPA in the context of generalised metric spaces where the distances come from a linearly ordered monoid and Hubi\v cka, Ne\v set\v ril and the author~\cite{Hubicka2017sauerconnant} later found Ramsey expansions for all such spaces.

Braunfeld~\cite{Sam}, motivated by his classification of generalised permutation structures~\cite{Sam2} found Ramsey expansions of $\Lambda$-ultrametric spaces which are ``metric spaces'' where the distances come from a finite distributive lattice.

Aranda, Bradley-Williams, Hubi{\v c}ka, Karamanlis, Kompatscher, Pawliuk and the author~\cite{Aranda2017} studied Ramsey expansions of metric spaces from Cherlin's list of metrically homogeneous graphs~\cite{Cherlin2013,Cherlin2011b}. These are metric spaces with distances from $\{0, 1, \ldots, \delta\}$ with some other families of triangles also forbidden.

The Ramsey property of all of these (and also other) classes follows, sometimes directly, sometimes less so, from the fact that they admit some form of the \emph{shortest path completion} (see Chapter~\ref{ch:shortestpath}). In this thesis we explore the boundaries of this method and introduce the concept of semigroup-valued metric spaces, a unifying framework for the aforementioned results. We prove that under certain assumptions, such classes have the strong amalgamation property, the Ramsey property and EPPA.

\section{Our results}
A \emph{commutative semigroup} is a tuple $(M, \oplus)$, where $\oplus\colon M^2\rightarrow M$ is a commutative and associative operation.

\begin{definition}[Partially ordered commutative semigroup]\label{def:POS}
We say that a tuple $\ESemig$ is a \emph{\pocs} if
\begin{enumerate}
\item $(M,\oplus)$ is a commutative semigroup;
\item $(M, \mleq)$ is a partial order which is reflexive ($a\mleq a$ for every $a\in M$);
\item For every $a, b\in M$ it holds that $a\mleq a\oplus b$;
\item \label{def:POS:monot} For every $a,b,c\in M$ it holds that if $b\mleq c$ then $a\oplus b \mleq a\oplus c$ ($\oplus$ is monotone with respect to $\mleq$).
\end{enumerate}
Note that point~\ref{def:POS:monot} implies that if $a\mleq b$ and $c\mleq d$, then $a\oplus c\mleq b\oplus d$.
\end{definition}

The following definition is motivated by earlier work of Conant and Braunfeld~\cite{Conant2015,Sam}, see also~\cite{Kabil2018}. It is one of the key definitions of this thesis.
\begin{definition}\label{def:Mmetric}
Let $\ESemig$ be a \pocs. Given a set $A$ and a function $d\colon {A\choose 2}\to M$, we call $(A,d)$ an \emph{$\Semig$-valued metric space} (or just \emph{$\Semig$-metric space}) if for every $x,y,z\in A$ it holds that $d(\{x,z\})\mleq d(\{x,y\})\oplus d(\{y,z\})$.

Given a \pocs{} $\Semig$, we let $\mathcal M_\Semig$ denote the class of all finite $\Semig$-metric spaces.
\end{definition}
The function $d$ being defined on the set of all pairs instead of on $A^2$ is simply due to the fact that there is no identity in $\Semig$ to set $d(x,x)$ to. In the rest of the paper we will treat $d$ as a two-variable symmetric function not defined on any pair $x,x$.

\begin{example}\label{example1}
The following are some examples of partially ordered commutative semigroups:
\begin{enumerate}
\item\label{ex:Smetric} Let $S$ be a set of positive reals and for $a,b\in S$ define $a \oplus_S b = \sup\{c\in S; c\leq a+b\}$. Sauer~\cite{Sauer2013b} classified the sets $S$ such that $S\cup\{0\}$ is closed and $\oplus_S$ is an associative operation and hence $(S, \oplus, \leq)$ is a \pocs.
\item\label{infini} Consider the set of non-negative real numbers extended by infinitesimal ele\-ments, i.e. $R^* = \left\{a + b\cdot \mathrm{dx}\mid a,b\in\mathbb R^+_0\right\}$ with piece-wise addition $+$ and order $\mleq$ given by the standard order of reals and $\mathrm{dx}\mlt a$ for every positive real number $a$. Then $(R^*, +, \mleq)$ is also a \pocs{}.
\item The ultrametric $(\{1, \ldots, n\}, \max, \leq)$, where $\leq$ is the linear order of integers is a \pocs{}.
\item A distributive lattice $\Lambda = (L, \wedge, \vee, 0)$ with minimum $0$ can be viewed as a partially ordered commutative semigroup, where the operation is $\vee$ and the order is the standard partial order of $\Lambda$. The $\Lambda$-valued metric spaces are essentially Braunfeld's $\Lambda$-ultrametric spaces~\cite{Sam} (see Remark~\ref{rem:sam}).
\item The multiplicative monoid $(\mathbb Z^{\geq 1}, \cdot, |)$, where the order is given by the ``is a divisor of'' relation is also a \pocs{}. The divisibility metric spaces behave very differently from the standard real-valued ones and one of the original contributions of this thesis is finding their Ramsey expansions.
\end{enumerate}
\end{example}

\begin{remark}\label{rem:sam}
For Definition~\ref{def:Mmetric} it would be more convenient to work with monoids instead of semigroups (that is, semigroups which contain a neutral element $0$) and demand that if $d(x,y)=0$, then $x=y$. It would also simplify Sections~\ref{sec:lstar} and~\ref{sec:convord}. On the other hand, it would make some other parts of this thesis more notationally complicated (if $0$ represents the identity, then it necessarily needs to be treated differently than all the other distances).

In the context of this thesis, the advantages of not having any special distances outweighed the disadvantages. However, all the mentioned results work with monoids. In particular, Braunfeld's definition of $\Lambda$-ultrametric spaces is stronger than Definition~\ref{def:Mmetric} for $\Semig$ being a distributive lattice, because it allows for the neutral element (representing the identity in the $\Lambda$-ultrametric spaces) to be \emph{meet-reducible} (that is, to be the meet of two non-neutral elements, see Definition~\ref{defn:meetreducible}). Such classes do not have the \emph{strong amalgamation property} (see Definition~\ref{defn:amalgamation}) and are out of this thesis' scope.
\end{remark}

Now we state the main theorem of this thesis. There are several undefined notions which will be defined in the subsequent chapters.

\begin{theorem}\label{thm:main}
Let $\ESemig$ be a \pocs{} and $\mathcal F$ be a family of $\Semig$-edge-labelled cycles (Definition~\ref{defn:graphs}). Suppose that the following conditions hold:
\begin{enumerate}
\item $\mathcal F$ is $\Semig$-omissible (Definition~\ref{defn:omissible});
\item $\mathcal F$ contains all $\Semig$-disobedient cycles (Definition~\ref{defn:disobedient});
\item $\mathcal F$ synchronizes meets (Definition~\ref{defn:meetsync}); and
\item $\mathcal F$ is confined (Definition~\ref{defn:slocallyfinite}).
\end{enumerate}

Then the class $\MF$ of all finite $\Semig$-valued metric spaces $\str A$ such that there is no $\str F \in \mathcal F$ with a homomorphism $\str F\to\str A$ has the strong amalgamation property and the class $\overrightarrow{\mathcal M}_\Semig\cap \Forb(\mathcal F)$ of all convexly ordered (Definition~\ref{defn:Js}) finite $\Semig$-valued metric spaces $\str A$ such that there is no $\str F \in \mathcal F$ with a homomorphism $\str F\to\str A$ has the Ramsey property.
\end{theorem}

This theorem looks quite technical, but often some of the conditions are tri\-vial to satisfy. In particular, when the order is linear, any family $\mathcal F$ including the empty one contains all disobedient cycles and synchronizes meets and hence we obtain a strengthening of the results of Hubi\v cka, Ne\v set\v ril and the author~\cite{Hubicka2017sauerconnant} on Conant's generalised metric spaces and the results of Hubi\v cka and Ne\v set\v ril~\cite{Hubicka2016} on Sauer's $S$-metric spaces. On the other hand, when $\Semig$ is a distributive lattice (i.e. $\vee$ is the operation) and $\mathcal F$ is empty, one obtains the results of Braunfeld~\cite{Sam} on $\Lambda$-ultrametric spaces (only ones with the strong amalgamation property, see Remark~\ref{rem:sam}). In Section~\ref{sec:methom} we shall see that for a suitable \pocs{} $\Semig$ (the \emph{magic semigroup}) and a suitable family $\mathcal F$, Theorem~\ref{thm:main} also extends the results of~\cite{Aranda2017} on Cherlin's primitive 3-constrained metrically homogeneous graphs.

\begin{remark}
In our proofs we make use of all the conditions on $\mathcal F$. For example, $\Semig$-omissibility is necessary for the family $\mathcal F$ to be consistent with the shortest path completion. $\mathcal F$ containing all $\Semig$-disobedient cycles is a weaker version of requiring $\mleq$ to be a lattice and $\oplus$ to distribute with the infima. We need this weaker variant in order to be able to represent the metrically homogeneous graphs. Because we want to prove local finiteness (see Definition~\ref{def:locallyfinite}) of structures omitting $\mathcal F$, we have to put some conditions on $\mathcal F$, namely confinedness, to ensure that it does not break the local finiteness.
\end{remark}



To illustrate the flexibility of our techniques we also prove EPPA.
\begin{theorem}\label{thm:eppa}
Let $\ESemig$ be a \pocs{} and $\mathcal F$ be a family of $\Semig$-edge-labelled cycles (Definition~\ref{defn:graphs}). Suppose that the following conditions hold:
\begin{enumerate}
\item $\mathcal F$ is $\Semig$-omissible (Definition~\ref{defn:omissible});
\item $\mathcal F$ contains all $\Semig$-disobedient cycles (Definition~\ref{defn:disobedient});
\item $\mathcal F$ synchronizes meets (Definition~\ref{defn:meetsync}); and
\item $\mathcal F$ is confined (Definition~\ref{defn:slocallyfinite}).
\end{enumerate}
Then the class $\mathcal M_\Semig \cap \Forb\left(\mathcal F\right)$ has EPPA (Definition~\ref{defn:eppa}).
\end{theorem}

As a corollary of Theorems~\ref{thm:main} and~\ref{thm:eppa}, we re-prove several results and also solve some open problems. The most important examples are summarised in the following table:

\begin{table}[h!]
\centering
\begin{tabular}{ p{6cm}  l  l }
	\toprule
    \textbf{Class} & \textbf{Ramsey} & \textbf{EPPA} \\
    \midrule
    $S$-metric spaces & \cite{Hubicka2016} & Open (part~\cite{Conant2015}) \\
    Monoid-valued spaces & \cite{Hubicka2017sauerconnant} &  Open (part~\cite{Conant2015}) \\
    $\Lambda$-ultrametric spaces & \cite{Sam} & Open \\
    Metrically homogeneous graphs & \cite{Aranda2017} & \cite{Aranda2017} \\
    $4$-edge-labelled graphs & Open (using~\cite{Li2018}) & Open (using~\cite{Li2018})\\
    \bottomrule
\end{tabular}
\caption{Selected corollaries of Theorems~\ref{thm:main} and~\ref{thm:eppa}}
\end{table}
All these corollaries are discussed in Chapter~\ref{ch:applications}.

\medskip

In the appendix of~\cite{Cherlin1998}, Cherlin gave a classification of classes of complete $\{A,B,C,D\}$-edge-labelled graphs determined by forbidden triangles which form a non-free strong amalgamation class whose \Fraisse{} limit is \emph{primitive} (no nontrivial definable equivalences). There are 26 of them. Li~\cite{Li2018} studied simplicity of the automorphism groups of \Fraisse{} limits of these classes and in the process found partially ordered commutative semigroups that, as it turns out, can be plugged into our machinery. In other words, every triangle constrained primitive strong amalgamation class of $\{A,B,C,D\}$-edge-labelled complete graphs fits into our framework. A similar classification was obtained by the author for $\{A,B,C,D,E\}$-edge-labelled graphs (using a computer program, there are more than 1400 such primitive strong amalgamation classes) and all these classes can also be understood as semigroup-valued metric spaces which fit into Theorems~\ref{thm:main} and~\ref{thm:eppa}.\footnote{This is still work-in-progress, the computer program needs to be verified.} This motivates the following conjecture:
\begin{conjecture}\label{conj:universal}
Let $L$ be a finite set and let $\mathcal C$ be a strong amalgamation class of complete $L$-edge-labelled graphs whose \Fraisse{} limit is primitive. Then there is an archimedean (see Section~\ref{sec:blocks}) \pocs{} $\Semig$ on the set $L$, a family $\mathcal F$ of $\Semig$-edge-labelled graphs and a family $\mathcal H$ of Henson constraints (see Section~\ref{sec:henson}) such that $\Semig$ and $\mathcal F$ satisfy the conditions of Theorems~\ref{thm:main} and~\ref{thm:eppa} and $\mathcal C = \MF\cap\Forb(\mathcal H)$.
\end{conjecture}
This conjecture implies that such $\mathcal C$ admits a shortest path completion (see Section~\ref{sec:henson}). Another example where this conjecture holds very non-trivially are metrically homogeneous graphs (see Section~\ref{sec:methom}). For more details see Chapter~\ref{ch:conclusion}.

\section{Organization of the thesis}
In Chapter~\ref{ch:background} we review the history of the Ramsey theory and the study of homogeneous structures with emphasis on different metric-like structures. We also present all the necessary notions and definitions including the results of~\cite{Hubicka2016} on multiamalgamation classes, which are going to be a key ingredient for proving the Ramsey property.

Chapter~\ref{ch:shortestpath} is dedicated to introducing the setting which we are going to work in, stating all the necessary definitions and proving some basic results on amalgamation and completion. In Chapter~\ref{ch:orderless} we study \emph{blocks} (maximal archimedean subsemigroups) and the corresponding definable equivalences in $\Semig$-valued metric spaces. We define an expansion which explicitly represents these equivalences and prove that these expanded classes have the strong amalgamation property while being locally finite. These results suffice to prove EPPA (Section~\ref{sec:eppaproof}).

Because for the Ramsey property one needs an order, we add such an order in Chapter~\ref{ch:withorder} (in a way very similar to Braunfeld~\cite{Sam}) and again prove the strong amalgamation property and local finiteness. This enables us to prove Theorem~\ref{thm:main} and also the expansion property. In Chapter~\ref{ch:applications} we sketch how Cherlin's metrically homogeneous graphs relate to $\Semig$-valued metric spaces and discuss other extensions and applications of our results.

Throughout the thesis we introduce many new notions and several different classes derived from the $\Semig$-valued metric spaces. Appendix~\ref{ap:vocabulary} contains the important notions together with short (often informal) definitions and links to the proper definitions, Appendix~\ref{ap:classes} contains a list of the classes which we introduced and work with. Again, their descriptions are informal; these appendices should be used as cheat sheets while reading the thesis, not as replacements for reading and understanding the definitions.

\section{Outline of the methods}\label{sec:outline}
To carry out the ideas of this thesis correctly takes some effort and perhaps obscures the key insights a bit. In this section we try to give a very intuitive outline of what is going to happen on the upcoming pages. We omit many details and technical complications and for simplicity only talk about edge-labelled graphs (that is, binary symmetric structures).

Recently (cf. Theorems~\ref{thm:hn} and~\ref{thm:herwiglascar}), the following meta-question became of interest in the combinatorial model theory: ``Given a finite graph $\str G$ with edges labelled by symbols from some set $L$, is it possible to add the remaining edges and their symbols so that the resulting complete edge-labelled graph belongs to a given class $\mathcal C$?'' The answer is often found constructively, that is, by finding an explicit \emph{completion procedure} which, given a graph $\str G$, produces a complete edge-labelled graph $\str G'$ such that $\str G'$ is from $\mathcal C$ unless $\str G$ has no \emph{completion} in $\mathcal C$.

A prime example is the \emph{shortest path completion} which works for $\mathcal C$ being a class of, say, \emph{integer-valued metric spaces}, that is, complete graphs with edges labelled by positive integers which contain no \emph{non-metric triangles} (triples of vertices such that the label of one of the edges is larger than the sum of the labels of the other two edges). The shortest path completion sets the \emph{distance} (label of the edge) of each pair of vertices to be the \emph{length} (sum of the labels) of the shortest path connecting them (shortest in terms of the length, not the number of edges).

The shortest path completion has been adapted for many different classes of edge-labelled complete graphs. In this thesis, we study the boundaries of this method. Instead of integers, we allow for the labels to come from an arbitrary commutative semigroup $\Semig$ with operation $\oplus$. And instead of the linear order of the integers, we work with a partial order $\mleq$ of $\Semig$. Our classes thus consist of finite $\Semig$-edge-labelled complete graphs such that in every triangle every edge is smaller (in $\mleq$) than the $\oplus$-sum of the other two; we call such graphs \emph{$\Semig$-valued metric spaces} or simply \emph{$\Semig$-metric spaces}. In the shortest path completion, instead of setting the distance of two vertices to the length of the shortest path connecting them, we set it to the infimum of the $\Semig$-lengths of all the paths connecting the two vertices.

The last sentence was deliberately slightly imprecise. The partial order on the semigroup surely cannot be arbitrary. One should at least require $\oplus$ to be monotone. And in order for the previous paragraph to describe a well-defined procedure, the partial order would actually have to be a semi-lattice. This, however, turns out to be too strong a condition (cf. Section~\ref{sec:methom}). Instead, we refine the relevant classes and allow to further forbid (homomorphic images of) a family of $\Semig$-edge-labelled cycles $\mathcal F$. In other words, the classes which we are interested in are the intersections of all finite $\Semig$-valued metric spaces and $\Forb(\mathcal F)$ (the class of all finite $\Semig$-edge-labelled graphs containing no homomorphic image of a member of $\mathcal F$) for some \pocs{} $\Semig$ and a family $\mathcal F$. 

Such a family $\mathcal F$ then ensures that all infima of paths which the shortest path completion procedure encounters are well-defined. Of course, not every family of cycles can be forbidden, the least we have to demand is that it is \emph{compatible} with the shortest path completion, that is, if the input graph $\str G$ is from $\Forb(\mathcal F)$, then its shortest path completion $\str G'$ must also be from $\Forb(\mathcal F)$.

In Chapter~\ref{ch:shortestpath} we describe precisely the conditions on $\mathcal F$ and study the basic properties of the shortest path completion. In particular, we observe that a graph from $\Forb(\mathcal F)$ has a completion to a $\Semig$-metric space if and only if it contains no non-$\Semig$-metric cycle (a cycle which has no completion).

\medskip

While the completion procedure is an important ingredient, for some applications --- for example for finding Ramsey expansions --- it is not enough. In order to use the Hubi\v cka--Ne\v set\v ril theorem (Theorem~\ref{thm:hn}), one needs for every finite set of distances $S\subseteq \Semig$ to find a finite family of \emph{obstacles} $\mathcal O_S$ such that an $S$-edge-labelled graph has a completion in the given class if and only if it is from $\Forb(\mathcal O_S)$.

Such families, however, do not exist for every \pocs{} $\Semig$ and family $\mathcal F$. For example, for $\Semig = (\{1,2\}, \max, \leq)$ and $\mathcal F = \emptyset$, every cycle with one edge labelled 2 and all the other edges labelled 1 has no completion. The reason for this is that ``being in distance 1'' is an equivalence relation. In Chapter~\ref{ch:orderless} we deal with this phenomenon in general and study the \emph{block structure} of \pocs{}s, or in other words, we study the equivalence relations given by ``being in a distance from a given subset of $\Semig$''.

We prove (in Section~\ref{sec:nonimportant}) that for a fixed finite $S\subseteq \Semig$, one can select a bounded number of \emph{important} edges from each non-$\Semig$-metric cycle such that the other edges effectively only say that some vertices are block-equivalent for some block. The next thing to do then is to remove the unimportant edges and replace them by something which can concisely represent the ``these vertices are block-equivalent'' relation.

If the reader is familiar with model theory, they know that we want to find an expansion with \emph{elimination of imaginaries}. Specifically, we will add new vertices (\emph{ball vertices}) to represent the equivalence classes of all the block equivalences and link the \emph{original vertices} to them by unary functions. Two vertices are then equivalent if and only if they ``point'' (by the unary functions) to the same vertex corresponding to the given equivalence. This means that one can indeed forget about the unimportant edges and only keep the ball vertices certifying that some original vertices are equivalent, and hence making the family $\mathcal O_S$ of obstacles finite.

The previous paragraphs summarised the ideas from Chapter~\ref{ch:orderless}, but to carry it out correctly takes some effort, we ignored several important details here.

\medskip

While being interesting and useful by itself, Chapter~\ref{ch:orderless} is not enough to get Ramsey expansions of the $\Semig$-valued metric spaces, because a Ramsey class has to have an order, which is the topic of Chapter~\ref{ch:withorder}. In particular, one also has to order the ball vertices. And as is standard in the structural Ramsey theory, the order of the ball vertices has to depend on the order of the original vertices. Braunfeld~\cite{Sam} has found such an ordering for his $\Lambda$-ultrametric spaces and we use the same ideas.
\chapter{Preliminaries and background}\label{ch:background}
We first review some standard model-theoretic notions regarding structures with relations and functions (see e.g.~\cite{Hodges1993}) with a small variation that our functions will be partial.

A \emph{language} $L$ is a collection $L=L_{\mathcal{R}}\cup L_{\mathcal{F}}$ of relational symbols $\rel{}{}\in L_{\mathcal{R}}$ and function symbols $F\in L_{\mathcal{F}}$ each having associated arities. For relations the arity is denoted by $\arity{\rel{}{}}>0$ for relations, for functions $\arity{\func{}{}}$ is the arity of the domain, in this thesis range will always have arity one.

An \emph{$L$-structure} $\str A$ is then a tuple $(A, \{\rel{A}{}\}_{\rel{}{}\in L_{\mathcal R}}, \{\func{A}{}\}_{\func{}{}\in L_{\mathcal F}})$, where $A$ is the \emph{vertex set}, $\rel{A}{}\subseteq A^{\arity{\rel{}{}}}$ is an interpretation of $\rel{}{}$ for each $\rel{}{}\in L_{\mathcal R}$ and $\func{A}{}\colon A^{\arity{\func{}{}}}\rightarrow A$ is a partial function for each $\func{}{}\in L_{\mathcal F}$. We denote by $\dom(\func{A}{})$ the domain of $\func{}{}$ (i.e. the set of tuples of vertices of $\str A$ for which $\func{}{}$ is defined).  An $L$-structure is \emph{finite} (or has \emph{finite support}) if its vertex set is finite.

An $L$-structure $\str A$ is \emph{connected} if for every two vertices $u,v\in A$ there exists a sequence of vertices $u=v_0, v_1, \ldots, v_k=v$ such that for every $i$ there is a tuple $\bar x\in A^k$ and a relation $\rel{}{}\in L$ or a function $\func{}{}\in L$ such that $\bar x\in \rel{}{}$ or $\bar x\in \func{}{}$ respectively (where we understand a function as a relation with some outdegree condition).

Notationally, we shall distinguish structures from their underlying sets by typesetting structures in bold font. When the language $L$ is clear from the context, we will use it implicitly.

Let $\str{A}$ and $\str B$ be $L$-structures. A \emph{homomorphism} $f\colon\str{A}\to \str{B}$ is a mapping $f\colon A\rightarrow B$ satisfying for every $\rel{}{}\in L_{\mathcal R}$ and for every $\func{}{}\in L_\mathcal F$ the following three statements:
\begin{enumerate}
\item[(a)] $(x_1,\ldots, x_{\arity{\rel{}{}}})\in \rel{A}{}\Rightarrow (f(x_1),\ldots,f(x_{\arity{\rel{}{}}}))\in \rel{B}{}$;
\item[(b)] $f(\dom(\func{A}{}))\subseteq \dom(\func{B}{})$; and
\item[(c)] $f(\func{A}{}(x_1,\allowbreak \ldots, x_{\arity{\func{}{}}}))=\func{B}{}(f(x_1),\allowbreak \ldots,\allowbreak f(x_{\arity{\func{}{}}}))$ for every $(x_1,\allowbreak \ldots, x_{\arity{\func{}{}}})\in \dom(\func{A}{})$.
\end{enumerate} 
For a subset $A'\subseteq A$ we denote by $f(A')$ the set $\{f(x);x\in A'\}$ and by $f(\str{A})$ the homomorphic image of a structure.

If $f$ is and injective homomorphism, it is a \emph{monomorphism}. A monomorphism is an \emph{embedding} if
for every $\rel{}{}\in L_{\mathcal R}$ and $\func{}{}\in L_\mathcal F$ the following holds:
\begin{enumerate}
\item[(a)] $(x_1,\ldots, x_{\arity{\rel{}{}}})\in \rel{A}{}\iff (f(x_1),\ldots,f(x_{\arity{\rel{}{}}}))\in \rel{B}{}$, and,
\item[(b)]  
$(x_1,\ldots, x_{\arity{\func{}{}}})\in\dom(\func{A}{}) \iff (f(x_1),\ldots,f(x_{\arity{\func{}{}}}))\in \dom(\func{B}{}).$
\end{enumerate}

If $f$ is a bijective embedding then it is an \emph{isomorphism} and we say that $\str A$ and $\str B$ are \emph{isomorphic}. An isomorphism $\str A\to\str A$ is called \emph{automorphism}. If the inclusion $\str A\subseteq \str B$ is an embedding then $\str{A}$ is a \emph{substructure} of $\str{B}$. For $\str A$ and $\str B$ structures, we denote by $\str B\choose \str A$ the set of all embeddings of $\str A$ to $\str B$. Note that while for purely relational languages every set $A\subseteq B$ gives a substructure of $\str B$, it does not hold in general for languages with functions (we need $A$ to be \emph{closed} on functions).

For a family $\mathcal F$ of $L$-structures, we denote by $\Forb(\mathcal F)$ the class of all finite $L$-structures (or $L^+$-structures with $L^+\supseteq L$ if this is clear from the context) $\str A$ such that there is no $\str F\in \mathcal F$ with a homomorphism $\str F\rightarrow \str A$. If $f\colon X\rightarrow Z$ is a function and $Y\subseteq X$, we denote by $f\restriction_Y$ the restriction of $f$ on $Y$.

Let $\str A, \str B$ be $L$-structures and let $f\colon A\rightarrow B$ be an injective map. We say that $f$ is \emph{automorphism-preserving} if there is a group homomorphism $h\colon\Aut(\str A)\to \Aut(\str B)$ such that $f\circ\alpha\circ f^{-1} \subseteq h(\alpha)$ for every $\alpha\in\Aut(\str A)$. In particular, for $f$ being the identity and $A=B$ this means that $\Aut(\str A)\subseteq \Aut(\str B)$.

\section{Homogeneous structures}\label{sec:homog}
The main reference for this section is the survey on homogeneous structures by~Macpherson~\cite{Macpherson2011}.

We say that an $L$-structure $\str M$ is \emph{homogeneous} (sometimes called \emph{ultrahomogeneous}) if for every finite $\str A,\str B\subseteq \str M$ and every isomorphism $g\colon  \str A\rightarrow \str B$ there is an automorphism $f$ of $\str M$ with $g\subseteq f$.

Let $\mathcal C$ be a class of (not necessarily all) $L$-structures and let $\str M$ be an $L$-structure. We say that $\str M$ is \emph{universal} for $\mathcal C$ if for every $\str C\in \mathcal C$ there exists an embedding $\str C\rightarrow \str M$.

\begin{example}
The structure $(\mathbb Q, <)$, where $<$ is the standard linear order of the rationals, is homogeneous and universal for countable linear orders.
\end{example}
\begin{proof}
First we observe that $(\mathbb Q, <)$ satisfies the \emph{extension property}, which is a very direct generalisation of density and means that for every finite $S\subset \mathbb Q$ and $B\subset\mathbb Q$ such that for every $s\in S$ and $b\in B$ it holds that $s < b$ there is $x\in\mathbb Q$ such that for every $s\in S$ it holds that $s < x$ and for every $b\in B$ it holds that $x < b$. For both $S$ and $B$ nonempty this follows by using density of $(\mathbb Q, <)$ for $\max S$ and $\min B$, for $S$ or $B$ empty is follows by $(\mathbb Q, <)$ having no endpoints.

Let $(X, \prec)$ be a countable (i.e. finite or countably infinite) linear order. We will construct an embedding $f\colon  (X, \prec) \rightarrow (\mathbb Q, <)$. Enumerate vertices of $X$ arbitrarily as $x_1, x_2, \ldots$. Define $f(x_1)$ to be an arbitrary vertex of $\mathbb Q$. Now assume that $f$ is already defined on $x_1, \ldots, x_{i-1}$. By the extension property there is a vertex $y\in \mathbb Q$ different from all already defined $f(x_j)$ such that for all $j<i$ it holds that $y < f(x_j)$ if and only if $x_i < x_j$. Thus we can set $f(x_i) = y$. Formally we are creating a chain of partial embeddings and the embedding $(X, \prec) \rightarrow (\mathbb Q, <)$ will then be their union.

To prove that $(\mathbb Q, <)$ is homogeneous, we will again construct an embedding of $(\mathbb Q, <)$ into $(\mathbb Q, <)$ step-by-step, but now we have to make sure that it is a surjection and that it extends the given partial automorphism. In order to do this, we use the so-called \emph{back-and-forth} argument.

Let $g$ be an isomorphism of finite substructures of $(\mathbb Q, <)$. We will define a chain $g=f_0\subset f_1\subset \cdots$ of partial automorphisms of $(\mathbb Q, <)$ as follows: Enumerate vertices of $\mathbb Q$ arbitrarily as $q_1, q_2, \ldots$. Assume that $g=f_0, \ldots, f_{i-1}$ are already defined and we want to extend $f_{i-1}$ by one point to construct $f_i$.

If $i$ is odd, let $j$ be the smallest integer such that $q_j\notin \dom(f_{i-1})$ where $\dom(f_{i-1})$ is the domain of $f_{i-1}$. By the extension property there is $y\in \mathbb Q\setminus \range(f_{i-1})$ such that for all $x\in \dom(f_{i-1})$ it holds that $q_j < x$ if and only if $y < f_{i-1}(x)$. Then we can set $f_i(q_j) = y$.

If $i$ is even, let $j$ be the smallest integer such that $q_j\notin \range(f_{i-1})$. Again, by the extension property there is $x\in \mathbb Q\setminus \dom(f_{i-1})$ such that for all $x'\in \dom(f_{i-1})$ it holds that $x < x'$ if and only if $q_j < f_{i-1}(x')$. Then we can set $f_i(x) = q_j$.

From the construction it follows that $f=\bigcup_{i=0}^\infty f_i$ is an automorphism of $(\mathbb Q, <)$ such that $g\subset f$.
\end{proof}

In the early 1950s \Fraisse{}~\cite{Fraisse1953, Fraisse1986} noticed that while $(\mathbb Q, <)$ is homogeneous and universal for countable linear orders, $(\mathbb N, <)$ is neither of those (though it is universal for finite linear orders) and as a (very successful) attempt to extract the necessary properties and generalise this phenomenon he proved Theorem~\ref{thm:fraisse}. The extracted properties are summarised in the following definition.

\begin{definition}[Amalgamation property~\cite{Fraisse1953}]\label{defn:amalgamation}
Fix a language $L$ and let $\mathcal C$ be a non-empty class of finite $L$-structures. We say that $\mathcal C$ is an \emph{amalgamation class} if it has the following properties:
\begin{enumerate}
\item $\mathcal C$ is closed under isomorphisms and substructures;
\item $\mathcal C$ has the \emph{joint embedding property (JEP)}: For all $\str B_1, \str B_2\in \mathcal C$ there is $\str C\in \mathcal C$ and embeddings $\beta_1\colon  \str B_1\rightarrow \str C$ and $\beta_2\colon \str B_2\rightarrow \str C$; and
\item $\mathcal C$ has the \emph{amalgamation property (AP)}: For all $\str A, \str B_1, \str B_2\in \mathcal C$ and embeddings $\alpha_1\colon \str A\rightarrow \str B_1$ and $\alpha_2\colon \str A\rightarrow \str B_2$, there is $\str C\in \mathcal C$ and embeddings $\beta_1\colon  \str B_1\rightarrow \str C$ and $\beta_2\colon \str B_2\rightarrow \str C$ such that $\beta_1\circ\alpha_1 = \beta_2\circ\alpha_2$. We call such $\str C$ the \emph{amalgamation} (or \emph{amalgam}) of $\str B_1$ and $\str B_2$ over $\str A$ (with respect to $\alpha_1$ and $\alpha_2$, but these embeddings are often assumed implicitly).
\end{enumerate}
It is common in the area to identify members of $\mathcal C$ with their isomorphism types.

Let $\str C$ be an amalgamation of $\str B_1$ and $\str B_2$ over $\str A$. We say that $\str C$ is a \emph{strong amalgamation} (or \emph{strong amalgam}) of $\str B_1$ and $\str B_2$ over $\str A$ if $\beta_1(x_1) = \beta_2(x_2)$ if and only if $x_1\in \alpha_1(\str A)$ and $x_2\in \alpha_2(\str A)$, which means that no vertices outside of the copies of $\str A$ are identified.

A strong amalgamation is a \emph{free amalgamation} (\emph{free amalgam}) if $C=\beta_1(x_1)\cup \beta_2(x_2)$ and furthermore for every relation $\rel{}{}$ and every tuple $\bar{x}=(x_1, \ldots, x_{\arity{\rel{}{}}})\in \rel{C}{}$ it holds that $\bar{x}\in\beta_i(B_i)$ for some $i\in\{1,2\}$ and similarly for every function $\func{}{}$, every tuple $\bar{x}=(x_1, \ldots, x_{\arity{\func{}{}}})\in \dom(\func{C}{})$ and $y\in C$ such that $\func{C}{}(\bar{x})=y$ it holds that $\bar{x}, y\in \beta_i(\str B_i)$ for the same $i\in\{1,2\}$.
\end{definition}

The joint embedding property says that for every two structures in $\mathcal C$ there is a structure in $\mathcal C$ containing both of them. If present in $\mathcal C$, their disjoint union could play such a role. But at the other extreme, if $\str B_1 = \str B_2$ then one can also put $\str C = \str B_1$.

The amalgamation property says that if one glues two structures $\str B_1$ and $\str B_2$ over a common substructure $\str A$, there is a structure $\str C$ in $\mathcal C$ which contains this patchwork, see Figure~\ref{fig:amalgamation}. By definition it is possible that the embeddings of $\str B_1$ and $\str B_2$ in $\str C$ overlap by more vertices than just the vertices of $\str A$ and also that $\str C$ has more vertices than just $\beta_1(B_1)\cup \beta_2(B_2)$.

If $\mathcal C$ contains the empty structure, then AP for $\str A$ being the empty structure is precisely JEP.

\begin{example}
Let $\mathcal C$ be the class of all finite metric spaces (with distances from, say, $\mathbb R^{\geq 0}$) and suppose that $\str A = (A, d_{\str A})$, $\str B_1 = (B_1, d_{\str B_1})$ and $\str B_2 = (B_2, d_{\str B_2})$ are structures from $\mathcal C$ such that $\str A\subseteq \str B_1$ and $\str A \subseteq \str B_2$. Define the metric space $\str C$ such that its vertex set is $C = B_1\cup B_2$ and the metric is defined as
$$d_{\str C}(u,v) = 
 \begin{cases} 
   d_{\str B_1}(u,v) & \text{if } u,v\in B_1 \\
   d_{\str B_2}(u,v) & \text{if } u,v\in B_2 \\
   \min_{w\in A} d_{\str B_1}(u,w) + d_{\str B_2}(v,w) & \text{if } u\in B_1,v\in B_2 \\
   \min_{w\in A} d_{\str B_2}(u,w) + d_{\str B_1}(v,w) & \text{if } u\in B_2,v\in B_1.
  \end{cases}$$
Then $\str C$ is the strong amalgam of $\str B_1$ and $\str B_2$ over $\str A$ where all the embeddings are inclusions. Furthermore, out of all strong amalgams of $\str B_1$ and $\str B_2$ over $\str A$, each distance in $\str C$ is as large as possible.
\end{example}

Let $\str M$ be an $L$-structure. Then the \emph{age of $\str M$} is defined as $$\Age(\str M) = \left\{\str A : \str A\text{ is a \textbf{finite} $L$-structure with an embedding $\alpha\colon \str A\rightarrow \str M$}\right\}.$$
Again, it is common to identify the age with the class of all isomorphism types of structures in the age.

If $\str M$ is a countable homogeneous $L$-structure such that for every finite $X\subseteq M$ there is a finite substructure $\str Y\subseteq \str M$ with $X\subseteq Y$ then we call $\str M$ a \emph{\Fraisse{} structure}. A \emph{\Fraisse{} class} is an amalgamation class with only countably many members up to isomorphism.
The main condition on a homogeneous structure to be \Fraisse{} says that finite sets of vertices generate only finite substructures. This is sometimes called local finiteness (for example in group theory), but unfortunately in the structural Ramsey theory this name is reserved for something else (see Definition~\ref{def:locallyfinite}).

\begin{figure}
\centering
\includegraphics{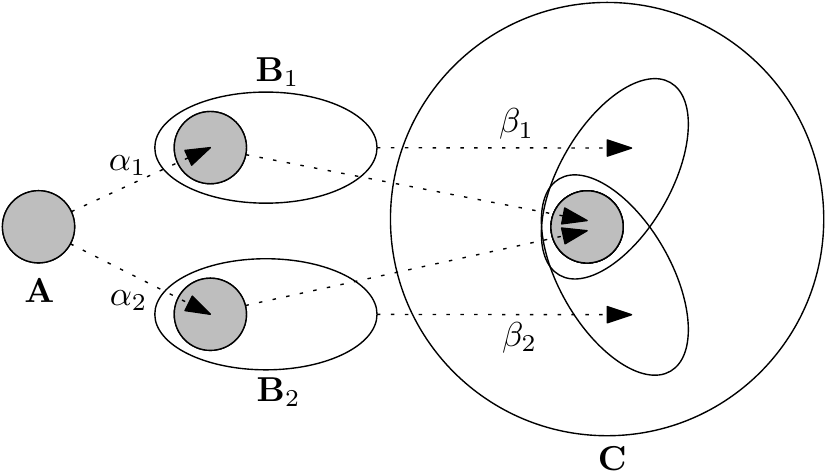}
\caption{An amalgamation of $\str{B}_1$ and $\str{B}_2$ over $\str{A}$.}
\label{fig:amalgamation}
\end{figure}

\begin{theorem}[\Fraisse{}~\cite{Fraisse1953}]\label{thm:fraisse}\leavevmode 
\begin{enumerate}
\item Let $\str M$ be a \Fraisse{} structure. Then $\Age\left(\str M\right)$ is a \Fraisse{} class.
\item For every \Fraisse{} class $\mathcal C$ there is a \Fraisse{} structure $\str M$ such that $\Age(\str M) = \mathcal C$. Furthermore, if $\str N$ is a countable homogeneous $L$-structure such that $\Age(\str N) = \mathcal C$, then $\str M$ and $\str N$ are isomorphic.
\end{enumerate}
\end{theorem}
We call the structure $\str M$ from the second point the \emph{\Fraisse{} limit of $\mathcal C$}. 

\medskip

\Fraisse{}'s theorem gives a correspondence between amalgamation classes and homogeneous structures. Besides the age of $(\mathbb Q, <)$, which is the class of all finite linear orders, there are many more known homogeneous structures and amalgamation classes. A prominent example is the random graph (often called the Rado graph), which we understand as a structure with one binary relation $E$ whose age is the class of all finite graphs and which is actually universal for all countable graphs. Or the generic triangle-free graph, which is the \Fraisse{} limit of all finite triangle-free graphs and again is universal for all countable triangle-free graphs. An example from a very different area is Hall's universal group~\cite{Hall1959} (a nice exposition is in Siniora's PhD thesis~\cite{Siniora2}). Hall's universal group is universal for all countable locally finite groups and it is a countable homogeneous group, which means that every isomorphism between finite subgroups can be extended to an automorphism of the whole group.

Take the random graph $\str R$\footnote{For a proper review of the history of the random graph see Peter Cameron's blog post History of the Random Graph~\cite{CameronBlog}.}. From homogeneity and universality it follows that it satisfies the \emph{extension property} which says that for every two disjoint finite sets of vertices $U, V\subset \str R$ there is a vertex $x\in \str R$ such that $x$ is connected by an edge to all vertices in $U$ and no vertex in $V$: There clearly exists a finite graph $\str G$ whose vertex set can be partitioned into disjoint sets $U'\cup V'\cup \{x'\}$ such that $\str G[U'\cup V']\cong \str R[U\cup V]$ ($\str G[X]$ means the subgraph of $\str G$ induced on $X$) and $x'$ is connected to all members of $U'$ and no member of $V'$. Hence, by universality, we can assume that $\str G\subseteq \str R$. But by homogeneity one can extend the natural partial isomorphism sending $U\cup V$ to $U'\cup V'$ to an automorphism $g$ of $\str R$ and then just let $x = g(x')$.

This extension property for $\str R$ is an analogue of the extension property for $(\mathbb Q, <)$. And it turns out that it defines $\str R$ --- in model theory $\str R$ is usually axiomatized by having the extension property. By a back-and-forth argument one can show that every two countable graphs having this property are isomorphic. The extension property thus implies homogeneity and universality in the same way as it does for $(\mathbb Q, <)$.

A suitable variant of the extension property can be defined for every homogeneous countable $L$-structure $\str M$. And when $L$ is for example a finite relational language, a back-and-forth argument can be utilized to prove that every countable structure in the same language with the extension property is isomorphic to $\str M$.\footnote{Theories with the property that they have only one countable model up to isomorphism are called $\omega$-categorical and are very important and interesting in the model theory context. This observation means that the theory of every homogeneous structure in a finite relational language is $\omega$-categorical.}

Another example of a homogeneous structure is the Urysohn space, which is a homogeneous separable metric space universal for all countable separable metric spaces. It was constructed by Urysohn in~1924~\cite{Urysohn1927}.

The Urysohn space $\mathbb U$ is constructed as the completion (in the metric space sense) of the rational Urysohn space $\mathbb U_{\mathbb Q}$, which is a homogeneous countable metric space with rational distances which is universal for all finite metric spaces with rational distances. $\mathbb U_{\mathbb Q}$ is constructed by a procedure in principle not very different from what \Fraisse{} used to prove Theorem~\ref{thm:fraisse}. Hence Urysohn was ahead of \Fraisse{} by roughly 30 years and the whole theory should perhaps be called Urysohn--\Fraisse{} theory instead of just \Fraisse{} theory.

So far we have seen a couple of examples of homogeneous structures, which are homogeneous for ``good reasons''. On the other hand, $(\mathbb N, <, (U^i)_{i\in \mathbb N})$, the structure consisting of natural numbers with the standard order plus infinitely many unary relations such that $U^i_{\mathbb N} = \{i\}$ (each vertex gets its own unary relation), is also homogeneous, but for ``stupid reasons'', namely the complete lack of isomorphisms between substructures.

\subsection{Classification results}
Homogeneous structures are being studied from several different perspectives. In this thesis we promote the combinatorial one. Another possible perspective is one of group theory: The automorphism groups of homogeneous structures are very rich (unless, of course, the structure is for example $(\mathbb N, <, (U^i)_{i\in \mathbb N})$). There are many results and notions connected to automorphism groups of homogeneous structures and a survey by Cameron~\cite{Cameron1999} serves as a very good starting reference. In model theory, homogeneous structures are studied for example from the stability point of view, see~\cite{Macpherson2011} for details. And last but not least, the automorphism group can be equipped with a natural topology --- the \emph{pointwise-convergence topology} --- and then studied from the point of view of topological dynamics. This will be touched upon a little more in Section~\ref{sec:kpt}.

But the initial direction after the \Fraisse{} theorem was on classification. In this section, we briefly and partially overview some classification results. We start with a theorem of Lachlan and Woodrow on the classification of all countably infinite homogeneous (undirected) graphs.

\begin{theorem}[Classification of countably infinite homogeneous graphs~\cite{Lachlan1980}]\label{thm:lachlanwoodrow}
Let $\str G$ be a countably infinite homogeneous undirected graph. Then $\str G$ or $\overbar{\str G}$ (the complement of $\str G$) is one of the following:
\begin{enumerate}
\item The random graph $\str R$ (i.e. the \Fraisse{} limit of the class of all finite graphs);
\item the \emph{generic} (that is, universal and homogeneous) $K_n$-free graph for some finite clique $K_n$, which is the \Fraisse{} limit of the class of all finite $K_n$-free graphs; or
\item the disjoint union of complete graphs of the same size (either an infinite union of $K_n$'s for some $n<\infty$, or a finite or infinite union of $K_\omega$'s).
\end{enumerate}
\end{theorem}

This theorem implies that there are only countably many countable homogeneous graphs, which is in contrast with an earlier result of Henson~\cite{Henson1972} who found $2^{\aleph_0}$ non-isomorphic countable homogeneous directed graphs. They are analogues of the $K_n$-free graphs, but while forbidding $K_n$ and $K_m$ is the same as forbidding $K_{\min(m, n)}$, Henson is forbidding tournaments and one can construct infinite sets of pairwise ``incomparable'' tournaments. Cherlin~\cite{Cherlin1998}, more than 20 years later, gave a full classification of countably infinite homogeneous directed graphs. The proof takes more than 170 pages and even the list itself is too long and complicated for our small historical overview; it contains for example the much older classification of homogeneous partial orders~\cite{Schmerl1979}. And recently, Cherlin~\cite{Cherlin2011b,Cherlin2013} offered a list of metrically homogeneous graphs, that is, countable graphs which become homogeneous metric spaces if one considers the path metric. They were the topic of the author's Bachelor thesis~\cite{Konecny2018bc} and we will briefly touch upon them in Section~\ref{sec:methom}.

\section{Ramsey theory}
Surveying the rich history of the Ramsey theory could easily be a topic for more than one thesis, but not a very good topic as there already are several good references. For this chapter's brief sketch of some of the most important results of Ramsey theory, Ne\v set\v ril's chapter~\cite[Ch. 25]{Graham1995} and Pr\"omel's book~\cite{PromelBook} were the main references. Some of the more recent development in the structural Ramsey theory was surveyed by, for example, Bodirsky~\cite{Bodirsky2015}, Nguyen Van Th\'e~\cite{NVT14} and Solecki~\cite{Solecki2013}.

In~1930, Ramsey published a paper where he proves the following theorem (which we state in today's language, by $[n]$ we mean the set $\{0, 1, \ldots, n-1\}$, and, for a set $A$, by ${A\choose p}$ we mean the set of all $p$-element subsets of $A$):
\begin{theorem}[Ramsey's theorem~\cite{Ramsey1930}]
For every triple of natural numbers $n,p,k$ with $n\geq 0$, $p > 0$ and $k\geq 1$ there is $N$ such that the following holds:

For every colouring $c\colon  {[N]\choose p} \rightarrow [k]$ there is an $n$-element subset $H\in {[N]\choose n}$ such that $c\restriction_{H\choose p}$ is constant.

\end{theorem}

\subsection{Category theory}\label{subsec:category}
Historically, the structural Ramsey theory was based on category theory~\cite{Leeb}. Although it is no longer the most common way, it will later be useful to define the Ramsey property for categories. In order to do it, recall that a \emph{category} $\mathfrak C$ is a tuple $(\mathrm{ob}(\mathfrak C), \mathrm{hom}(\mathfrak C))$, where $\mathrm{ob}(\mathfrak C)$ is a class of $\emph{objects}$ and $\mathrm{hom}(\mathfrak C)$ is a class of \emph{morphisms} (\emph{arrows}, \emph{maps}) between the objects. If $f$ is a morphism from object $A$ to object $B$ we write $f\colon A\rightarrow B$ and we write $\mathrm{hom}(A,B)$ for the class of all morphisms $A\rightarrow B$. Finally, for every triple of objects $A,B,C$ there is an associative binary operation $\circ\colon \mathrm{hom}(B,C)\times\mathrm{hom}(A,B) \rightarrow \mathrm{hom}(A,C)$ (the \emph{composition}), and we require that for every object there is an identity morphism (with respect to composition).

The name $\mathrm{hom}$ is rather unfortunate, note that is has nothing to do with homomorphisms as defined at the beginning of this chapter, or more precisely, one can consider categories where $\mathrm{hom}(A,B)$ is the set of all homomorphisms $A\to B$, but it is just a particular example of a category.

Let $\mathfrak C$ be a category. We say that an object $C$ is a \emph{Ramsey witness} for objects $A$ and $B$ and $k\in \mathbb N$ colours if for every colouring $c\colon \mathrm{hom}(A, C)\rightarrow \{0,\ldots,k-1\}$ there is a morphism $f\colon  B\rightarrow C$ and a colour $0\leq i < k$ such that for every $g\in\mathrm{hom}(A, B)$ it holds that $c(f\circ g) = i$.

Denoted schematically, for every $k$-colouring of $\mathrm{hom}(A, C)$ as in the following diagram (where we for simplicity let $k=2$ and depict one colour by dashed arrows and the other by full arrows)
\begin{equation*}
\begin{tikzcd}
A \arrow[rr, shift left=.5ex] \arrow[rr, shift right=.5ex]
\arrow[ddrr, dashed, shift right=.7ex] \arrow[ddrr] \arrow[ddrr, dashed, shift left=.7ex]
& & B \arrow[dd, shift left=.5ex] \arrow[dd, shift right=.5ex]\\ \\
& & C
\end{tikzcd}
\end{equation*}
there is an arrow $f\colon B\rightarrow C$, such that the composition $f\circ \mathrm{hom}(A,B)$ is mono\-chromatic:
\begin{equation*}
\begin{tikzcd}
A \arrow[rr, shift left=.5ex] \arrow[rr, shift right=.5ex]
\arrow[ddrr, dashed, shift right=.7ex] \arrow[ddrr, dashed, shift left=.7ex]
& & B \arrow[dd, shift left=0.5ex]\\ \\
& & C
\end{tikzcd}
\end{equation*}

If $C$ is a \emph{Ramsey witness} for $A$, $B$ and $k$, we denote it as $$C\longrightarrow(B)^A_k,$$ which is called the Erd\"os--Rado partition arrow. If a category has the property that for every $A$, $B$ and $k$ there is $C$ such that $C\longrightarrow(B)^A_k$, we say that the category has the \emph{Ramsey property}.

Notice now that the Ramsey theorem says that the category $\mathfrak{LO}$ of linear orders where maps are monotone injections has the Ramsey property.

\begin{remark}
The amalgamation property is also a categorical property: A category has the amalgamation property if for every triple of objects $A$, $B_1$ and $B_2$ with morphisms $\alpha_1\colon  A\rightarrow B_1$ and $\alpha_2\colon  A\rightarrow B_2$ there is an object $C$ and morphisms $\beta_1\colon  B_1\rightarrow C$ and $\beta_2\colon  B_2\rightarrow C$ such that the following diagram commutes:
\begin{equation*}
\begin{tikzcd}
A \arrow[rr, "\alpha_1"] \arrow[dd, "\alpha_2"]
& & B_1 \arrow[dd, "\beta_1"] \\ \\
B_2 \arrow[rr, "\beta_2"]
& & C
\end{tikzcd}
\end{equation*}
Notice that it is a weak version of a pushout as $C$ need not have any universal property.
\end{remark}

This categorical point of view will be useful later as it is an easy way to transfer the amalgamation and Ramsey properties between isomorphic categories.

\subsection{Some more Ramsey-type results}
Ramsey's paper was concerned with a problem of deciding whether satisfaction of certain formulas (the $\exists\forall$-formulas) is decidable. In 1935, Erd\"os and Szekeres initiated the combinatorial applications of the Ramsey theorem by proving the following result:
\begin{theorem}[Erd\"os--Szekeres~\cite{ErdosSzekeres}]
For every positive integer $m\geq 3$ there is a positive integer $N$ such that in any set of $N$ points in the Euclidean plane, no three of which are collinear, there are $m$ points which form the vertex set of a convex $m$-gon.
\end{theorem}
This theorem can be proved using the Ramsey theorem for quadruples of points coloured based on whether they all lie on the boundary of their convex hull.

Long before Ramsey, Schur, using the following Ramsey-type result about the natural numbers, proved that Fermat's Last Theorem is false modulo large primes:
\begin{lemma}[Schur~\cite{Schur1917}]
For every positive integer $k$ there exists a positive integer $N$ such that for every colouring $c\colon \{1, \ldots,2N\} \rightarrow [k]$ there are $1\leq x < y \leq N$ such that $c(x) = c(y) = c(x+y)$.
\end{lemma}
It follows easily from Ramsey's theorem by colouring pairs $x,y$ by $c(|x-y|)$ and looking for a monochromatic triple. Schur also conjectured the following statement which was in 1927 proved by van der Waerden.
\begin{theorem}[van der Waerden~\cite{vanderWaerden1927}]
For every pair of positive integers $k$ and $r$ there exists a positive integer $N$ such that for every colouring $c\colon \{1,\ldots, N\} \rightarrow [k]$ there exist positive integers $a$ and $d$ such that $A = \{a+id\mid 0\leq i<r\} \subseteq \{1,\ldots, N\}$ and $c\restriction_A$ is constant.
\end{theorem}
The original proof of van der Waerden's theorem goes by double induction and does not give a primitive recursive upper bound. This was an open problem for a long time until Shelah~\cite{ShelahvdW} found such a bound.

There are many more important and interesting results (for example the celebrated Szemer\'edi theorem~\cite{Szemeredi1969}, the Hales--Jewett~\cite{Hales1963} or the Graham--Rothschild theorems~\cite{Graham1971}), but we now concentrate on the structural Ramsey theory which this thesis is a contribution to.

\subsection{Structural Ramsey theory}
Having defined what the Ramsey property for a category is, a natural question for a combinatorialist would be: ``What about graphs?''

The answer is that it is an ambiguous question. What are the morphisms? If one looks at graphs with the non-induced-subgraph morphisms (that is, equivalence classes of monomorphisms modulo permutation of vertices), then this category has the Ramsey property directly by the Ramsey theorem as it is enough to be able to find arbitrarily large monochromatic complete graphs. If one takes the category $\mathrm{Gra}$ of finite graphs with embeddings, then the answer is negative.

\begin{prop}\label{prop:rigidity}
The category $\mathrm{Gra}$ of finite graphs with embeddings does not have the Ramsey property.
\end{prop}
\begin{proof}
It is enough to take $\str{A} = \str{B}$ the graph consisting of two vertices connected by an edge. It has a non-trivial automorphism (one that switches the vertices), hence there are two embeddings of $\str{A}$ into $\str{B}$. Suppose that there is $\str{C}$ such that $\str{C}\longrightarrow (\str{B})^{\str{A}}_2$ and let $c\colon  {\str C\choose \str A} \rightarrow [2]$ be such that for every edge of $\str C$ it will colour one of the embeddings by colour~$0$ and the other by colour $1$. Then, clearly, there is no embedding of $\str B$ into $\str C$ which is $c$-monochromatic.
\end{proof}
Note that one can repeat the same argument for any category which contains an object with a non-trivial automorphism. Thus a category needs to be \emph{rigid} (no non-trivial automorphisms of any object) in order to have the Ramsey property.

\begin{remark}
If instead of embeddings one considers the category of graphs with the ``induced subgraph'' morphisms (i.e. equivalence classes of embeddings modulo automorphisms), then the category still doesn't have the Ramsey property. It has the $\str A$-Ramsey property (for every $k$ and $\str B$ there is a $\str C$ with $\str C\longrightarrow (\str B)^\str{A}_k$) precisely for complete graphs and independent sets, see~\cite[Ch. 25, Sec. 5]{Graham1995}.
\end{remark}

In the structural Ramsey theory, one is always working with categories of (finite) structures equipped with embeddings. Hence it is appropriate to say that a class $\mathcal C$ of finite $L$-structures \emph{has the Ramsey property}, or \emph{is Ramsey}, if the category with objects from $\mathcal C$ and embeddings as morphisms has the Ramsey property.

In~1977 Ne\v set\v ril and R\"odl and independently in~1978 Abramson and Harrington proved the following:
\begin{theorem}[Ne\v set\v ril--R\"odl~\cite{Nevsetvril1977,Nevsetvril1977b}, Abramson--Harrington~\cite{Abramson1978}]\label{thm:nrgraph}
The class of all linearly ordered finite graphs is Ramsey.
\end{theorem}
Here the language is $L=\{E, \leq\}$ and the embeddings are order-preserving.

The techniques of Ne\v set\v ril and R\"odl actually prove much more. Let $L$ be a language. An $L$-structure $\str C$ is \emph{reducible} if there are $L$-structures $\str A, \str B_1, \str B_2$ different from $\str C$ such that $\str C$ is the free amalgam of $\str B_1$ and $\str B_2$ over $\str A$. Otherwise $\str C$ is \emph{irreducible}.

In the class of graphs the irreducible structures are precisely cliques. If $L$ is a relational language then irreducibility means that for every pair of vertices there is a relation and a tuple in that relation containing both vertices (i.e. the Gaifman graph is a clique).

Now we state what is known as the Ne\v set\v ril--R\"odl theorem:
\begin{theorem}[Ne\v set\v ril--R\"odl~\cite{Nevsetvril1977,Nevsetvril1977b}]\label{thm:nr}
Let $L$ be a relational language and let $\mathcal F$ be a collection of irreducible finite $L$-structures. Define $\mathcal F^\leq$ to be the collection of all linear orderings of structures from $\mathcal F$. Then the class of all finite linearly ordered $L$-structures $\str M$ such that there is no $\str F\in \mathcal F^\leq$ with an embedding $\str F\rightarrow \str M$ is Ramsey.
\end{theorem}
Theorem~\ref{thm:nrgraph} is a direct consequence of Theorem~\ref{thm:nr}. Theorem~\ref{thm:nr} also implies that the class of ordered $K_n$-free graphs is a Ramsey class or in general that every relational free amalgamation class is Ramsey when linearly ordered. It is known (cf. Proposition~\ref{prop:ordernecessary}) that every Ramsey class \emph{fixes} an order. Thus Theorem~\ref{thm:nr} is tight in the sense that adding the order is necessary.

The following observation of Ne\v set\v ril from~1989 gives (under a mild assumption) a strong necessary condition for Ramsey classes and connects the Ramsey theory with the theory of homogeneous structures:
\begin{theorem}[Ne\v set\v ril~\cite{Nevsetvril1989a,Nevsetril2005}] \label{thm:ramseyamalg}
Let $\mathcal C$ be a Ramsey class of finite structures with the joint embedding property. Then $\mathcal C$ has the amalgamation property.
\end{theorem}
\begin{proof}
We need to show that for every $\str A,\str B_1,\str B_2\in \mathcal C$ and embeddings $\alpha_1\colon  \str A\rightarrow \str B_1$ and $\alpha_2\colon  \str A\rightarrow \str B_2$ there is $\str C\in \mathcal C$ and embeddings $\beta_1\colon  \str B_1\rightarrow \str C$ and $\beta_2\colon  \str B_2\rightarrow \str C$ such that $\beta_1\circ \alpha_1 = \beta_2\circ \alpha_2$.

Let $\str B$ be a joint embedding of $\str B_1$ and $\str B_2$ and take $\str C\in \mathcal C$ such that $\str C \longrightarrow (\str B)^\str A_2$. We will prove that $\str C$ is the amalgam we are looking for.

Assume the contrary which means that there is no embedding $\alpha\colon  \str A\rightarrow \str C$ with the property that there are embeddings $\beta_1\colon  \str B_1\rightarrow \str C$ and $\beta_2\colon  \str B_2\rightarrow \str C$ such that $\beta_i\circ\alpha_i = \alpha$ for $i\in\{1,2\}$. Hence, for every $\alpha\colon  \str A\rightarrow \str C$ there is at most one such embedding $\beta_i\colon  \str B_i\rightarrow \str C$. Define the colouring $c\colon  {\str C\choose \str A} \rightarrow \{0, 1\}$ by letting
$$c(\alpha) = 
 \begin{cases} 
   0 & \text{if there is } \beta_1 \colon  \str B_1\rightarrow \str C\text{ such that }\alpha = \beta_1\circ\alpha_1 \\
   1       & \text{otherwise }.
  \end{cases}$$

For an illustration, see Figure~\ref{fig:ramseyamalg}.

\begin{figure}
\centering
\includegraphics{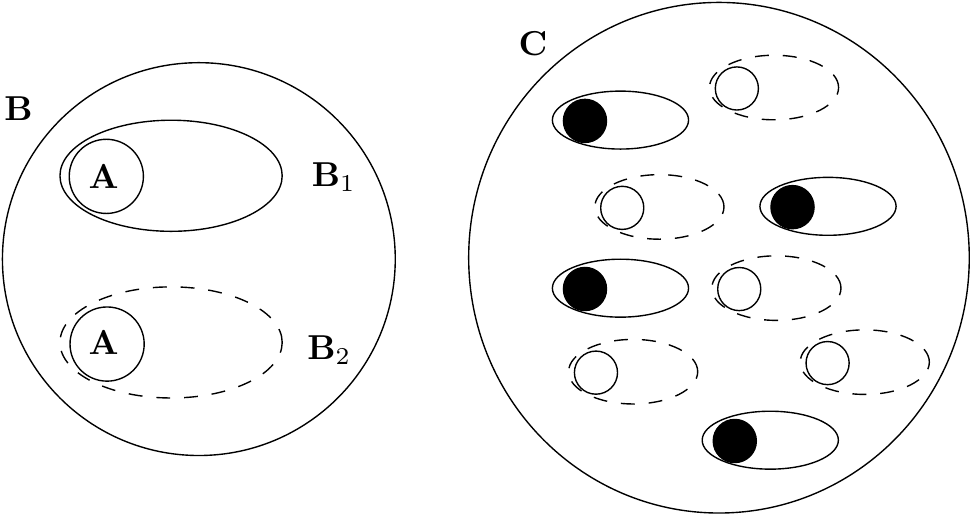}
\caption{Illustration of the proof of Theorem~\ref{thm:ramseyamalg}. Copies of $\str A$ from $\str B_1$ are coloured black, copies of $\str A$ from $\str B_2$ are coloured white.}
\label{fig:ramseyamalg}
\end{figure}

But then, as $\str C \longrightarrow (\str B)^\str A_2$, there is an embedding $\beta \colon  \str B \rightarrow \str C$ such that $c\restriction_{\beta(\str B)}$ is constant. But there are at least two embeddings of $\str A$ into $\beta$ --- one is given by $\alpha_1$ and the other is given by $\alpha_2$. And $\alpha_1$ can be extended to an embedding of $\str B_1$, while $\alpha_2$ can be extended to an embedding of $\str B_2$, hence they got different colours, which is a contradiction.
\end{proof}

Theorem~\ref{thm:ramseyamalg} gives rise to the question whether every amalgamation class is a Ramsey class, to which the answer is negative, as for example the (unordered) graphs form an amalgamation class which is not Ramsey. A follow-up question might be whether every amalgamation class, when enriched by all possible linear orders (for each structure $\str M$ in the original class there will be $n!$ structures in the new, ordered class), is Ramsey. And the answer is, again, no:
\begin{prop}
Let $\str M$ be the disjoint union of two infinite cliques $K_\omega$. $\str M$ is on the Lachlan--Woodrow list (Theorem~\ref{thm:lachlanwoodrow}) hence is homogeneous. Let $\mathcal C = \Age(\str M)$ be the class of all finite graphs such that they are either a clique or the disjoint union of two cliques. And let $\mathcal C^\leq$ be the class of all possible orderings of members of $\mathcal C$. Then $\mathcal C^\leq$ has the amalgamation property, but does not have the Ramsey property.
\end{prop}
\begin{proof}
The amalgamation property is easy: since the order and the graph structure are independent, we can use the amalgamation procedure for the order and the graph structure independently using the fact that both structures have the strong amalgamation property.

Now let $\str A$ be a vertex and $\str B$ be a pair of vertices not connected by an edge. Any $\str C \in \mathcal C^\leq$ which contains a copy of $\str B$ must consist of two cliques. Then one can colour the vertices of one of the cliques red and the vertices of the other clique blue and there will be no monochromatic non-edge in this colouring.
\end{proof}
Such a situation happens often and the following section introduces a way to deal with it.

\subsection{Expansions}
Notice that our colouring was based on the fact that the edge relation in the structure $\str M$ is actually an equivalence relation with two equivalence classes. If one could, for example, expand the language by a unary relation and distinguish the equivalence classes by putting the unary relation on all vertices from one of them, then such a class equipped with an order would be Ramsey (the edge relation would then be redundant and could be formally removed and the Ramseyness of such a class would follow from the Ne\v set\v ril--R\"odl theorem).

In order to formalize this observation and state another variant of the question whether all classes are Ramsey, we need to give a model-theoretic definition.

\begin{definition}[Expansion and reduct]\label{defn:expansion}
Let $L$ be a language and let $L^+$ be another language such that $L\subseteq L^+$ (i.e. $L^+$ contains all symbols that $L$ contains and they have the same arities). Then we call $L^+$ an \emph{expansion} of $L$ and we call $L$ a \emph{reduct} of $L^+$.

Let $\str M$ be an $L$-structure and let $\str M^+$ be an $L^+$-structure such that $\str M^+\restriction_L = \str M$ (by this we mean that $\str M$ and $\str M^+$ have the same sets of vertices and the interpretations of symbols from $L$ are exactly the same in both structures). Then we call $\str M^+$ an \emph{expansion} of $\str M$ and we call $\str M$ a \emph{reduct} of $\str M^+$.

If $\mathcal C$ is a class of finite $L$-structures, we say that $\mathcal C^+$, a class of finite $L^+$-structures, is its \emph{expansion} if for every $\str A\in \mathcal C$ there is $\str A^+\in \mathcal C^+$ which is its expansion and for every $\str A^+\in \mathcal C^+$ there is $\str A\in \mathcal C$ which is its reduct.
\end{definition}
\begin{remark}
In model theory, reduct and expansion often mean something more general, but for our purposes this definition is sufficient.
\end{remark}

Historically, expansions are often called \emph{lifts} in the Ramsey-theoretic context~\cite{Kun2007,Hubicka2009b} and reducts are called \emph{shadows}. We say that a class has a \emph{Ramsey expansion} if it has an expansion which is Ramsey.

So far we have only been adding all linear orders, which is clearly a special expansion (and corresponds to adding to the \Fraisse{} limit the dense linear order with no endpoints which is independent from the rest of the relations). But we also have seen that sometimes it isn't enough. 

To sum up, we know that every Ramsey class (with the joint embedding property) has the amalgamation property. Amalgamation classes (of finite structures in a countable language) correspond to homogeneous structures, their \Fraisse{} limits. And as we have seen, by adding some more structure on top of a homogeneous structure and looking at the age (or in general expanding the class), one can get a Ramsey class.

\medskip

In 2005, Ne\v set\v ril~\cite{Nevsetril2005} started the \emph{classification programme of Ramsey classes} --- the counterpart of the Lachlan--Cherlin classification programme of homogeneous structures. Its goal is to classify all possible Ramsey classes, a goal quite ambitious, but in some cases achievable; the classification programme of homogeneous structures offers lists of possible Ramsey classes, or rather base classes for expansions.

In this thesis we contribute to Ne\v set\v ril's programme by extracting the similarities of several known Ramsey classes and introducing a more abstract version which is their common generalisation.

Having read this far, the reader has probably already asked themselves: ``Does every amalgamation class have a Ramsey expansion?''

The answer to this question is positive, but by cheating: One can add infinitely many unary predicates and let each vertex have its own predicate. Then every structure has at most one embedding to any other and the Ramsey question becomes trivial. Two different ways of avoiding this cheat have been offered:
\begin{question}[Bodirsky--Pinsker--Tsankov~\cite{Bodirsky2011a}]
Does every amalgamation class in a finite language have a Ramsey expansion in a finite language?
\end{question}
This question still remains open. The other possible fix is motivated by topological dynamics (which we touch very briefly in Section~\ref{sec:kpt}). An amalgamation class $\mathcal C$ of finite $L$-structures is said to be \emph{$\omega$-categorical} if for every $n$ there are only finitely many non-isomorphic structures in $\mathcal C$ on $n$ vertices.\footnote{We have already mentioned what it means to be $\omega$-categorical for countable structures. One can prove that a countable structure is $\omega$-categorical if and only if its age is.}

Let $\mathcal C$ be a class of $L$-structures and let $\mathcal C^+$ be its expansion. We say that $\mathcal C^+$ is a \emph{precompact} expansion of $\mathcal C$ if for every $\str A\in \mathcal C$ there are only finitely many non-isomorphic $\str A^+\in\mathcal C^+$ which are expansions of $\str A$. Thus, precompactness is a relative version of $\omega$-categoricity.

\begin{question}[Melleray--Nguyen Van Th\'e--Tsankov~\cite{Melleray2015}]
Does every $\omega$-categorical amalgamation class have a \textit{precompact} Ramsey expansion?
\end{question}
This question has recently been answered negatively by Evans, Hubi\v cka and Ne\v set\v ril~\cite{Evans2}.

\section{The KPT correspondence}\label{sec:kpt}
In 2005, Kechris, Pestov and Todor\v cevi\'c published their famous paper which connected the field of structural Ramsey theory with the seemingly unrelated field of topological dynamics.

While topological dynamics and Ramsey theory are now interconnected, the necessary backgrounds needed to understand the fields are still very different. This section is here mostly for the historical context and hence the reader is sometimes assumed to know some concepts from topology and group theory. However, with the exception of Definition~\ref{defn:expproperty}, no contents of this section are necessary for understanding the rest of the thesis. For a solid overview of the KPT correspondence the survey by Nguyen Van Th\' e~\cite{NVT14} is a good reference.

\medskip

Let $\str M$ be a structure. By $\Aut(\str M)$ we denote the automorphism group of $\str M$. This group can be viewed as a \emph{topological group} when endowed with the \emph{pointwise convergence topology} (the natural choice in this setting), where by a group being topological we mean that both the operation and the inverse are continuous with respect to the given topology.

Recall that we say that $\str M$ is \emph{rigid} if $\Aut(\str M)$ is trivial, i.e. consist only of the identity.

For a topological group $G$ a \emph{$G$-flow} is a continuous action of $G$ on a topological space $X$, often denoted as $G\acts X$. We say that a $G$-flow is \emph{compact} if the space $X$ is compact.

\begin{definition}[Extremely amenable group]\label{defn:extremelyamenable}
Left $G$ be a topological group. We say that $G$ is \emph{extremely amenable} if every compact $G$-flow has a fixed point (i.e. $x\in X$ such that $g\cdot x = x$ for every $g\in G$).
\end{definition}
\begin{remark}
One might ask what an amenable group is. A (locally compact) topological group is amenable if it admits a finitely additive left-invariant probability measure on its subsets. For extremely amenable groups this measure can be Dirac.\footnote{Several paragraphs ago we said that Evans, Hubi\v cka and Ne\v set\v ril~\cite{Evans2} found an $\omega$-categorical class with no precompact Ramsey expansion, one of the so-called \emph{Hrushovski constructions}. It turns out that the difficulty for this class is already in finding an amenable expansion (an expansion whose \Fraisse{} limit's automorphism group is amenable). They suggested yet another question, namely whether every $\omega$-categorical structure with an amenable automorphism group has a precompact Ramsey expansion. This question has also been asked by Ivanov~\cite{Ivanov2015}.}

Yet another equivalent definition of extreme amenability is that the \emph{universal minimal flow} (which we did not define here) is a singleton. Computing universal minimal flows is of interest to people in topological dynamics and it turns out that the results of the structural Ramsey theory are an essential ingredient.
\end{remark}

Now we can state the theorem of Kechris, Pestov and Todor\v cevi\'c (preprint was published in 2003):
\begin{theorem}[KPT correspondence~\cite{Kechris2005}]
Let $\str M$ be a countable homogeneous structure. Then the following are equivalent:
\begin{enumerate}
\item $\Aut(\str M)$ is extremely amenable; and
\item $\Age(\str M)$ has the Ramsey property.
\end{enumerate}
\end{theorem}
We have defined Ramseyness as a property of an amalgamation class, but thanks to the KPT correspondence the Ramsey property is now witnessed directly by its \Fraisse{} limit (if it exists). This theorem of Kechris, Pestov and Todor\v cevi\'c ignited a new wave of interest in the structural Ramsey theory.

Nguyen Van Th\'e later~\cite{The2013} built on the ideas used in the proof of the KPT correspondence and introduced a way of computing the universal minimal flows using Ramsey expansions which are in certain sense \textit{minimal}:

\begin{definition}[Expansion property~\cite{The2013}]\label{defn:expproperty}
Let $\mathcal C$ be a class of finite structures and let $\mathcal C^+$ be its expansion. We say that $\mathcal C^+$ has the \emph{expansion property} (with respect to $\mathcal C$) if for every $\str B\in \mathcal C$ there is $\str C\in \mathcal C$ such that for every $\str B^+\in \mathcal C^+$ and for every $\str C^+\in \mathcal C^+$ such that $\str B^+$ is an expansion of $\str B$ and $\str C^+$ is an expansion of $\str C$ it holds that there is an embedding $\str B^+ \to \str C^+$.
\end{definition}
An expansion has the expansion property if for every small structure $\str B$ in the non-expanded class there is a large structure $\str C$ in the non-expanded class such that every expansion of $\str C$ contains every expansion of $\str B$. The expansion property is a generalization of the ordering property studied by Ne\v set\v ril and R\"odl in the 70's and 80's~\cite{Nesetril1975} and it turns out that it expresses well what an intuitively ``good'' expansion is.

Nguyen Van Th\'e also proves that under certain assumptions (for example finite relational languages satisfy those) Ramsey expansions with the expansion property correspond to the universal minimal flow. This means that, up to bi-definability, there is only one Ramsey expansion with the expansion property. And this is in some sense the best one. It is worth noting that Kechris, Pestov and Todor\v cevi\' c~\cite{Kechris2005} proved this for the special case when the expansion is all linear orders (i.e. the expansion property is the ordering property).

We conclude this section and the whole chapter with a sketch of an application of the KPT correspondence. For finite relational languages this proposition can be proved combinatorially (see~\cite[Proposition 2.22]{Bodirsky2015}) and in a stronger setting where the order will be definable:
\begin{prop}[Kechris--Pestov--Todor\v cevi\' c~\cite{Kechris2005}]\label{prop:ordernecessary}
Let $\str M$ be a Ramsey structure (that is, its age has the Ramsey property). Then $\Aut(\str M)$ fixes a linear order, which means that there exists a linear order $\preceq$ on the vertices of $\str M$ such that for every $g\in \Aut(\str M)$ and every $x,y\in M$ it holds that $x\preceq y$ if and only if $g(x)\preceq g(y)$.
\end{prop}
\begin{proof}[Sketch of proof]
Let $LO(M)$ be the space of all linear orders on $M$ viewed as a subspace of $\{0,1\}^{M\times M}$. Clearly $LO(M)$ is compact and $\Aut(\str M)$ acts continuously on it by its standard action: For $L\in LO(M)$ and $g\in \Aut(\str M)$ we define $g\cdot L$ by $(x,y)\in g\cdot L$ if and only if $(g^{-1}(x), g^{-1}(y))\in L$. Therefore, as $\str M$ is Ramsey, $\Aut(\str M)$ is extremely amenable and thus this action has a fixed point, which is an order $L$ such that $g\cdot L = L$ for every automorphism $g$.
\end{proof}

\section{The Hubi\v cka--Ne\v set\v ril theorem}
A homomorphism $f\colon \str A\rightarrow \str B$ is a \emph{homomorphism-embedding} if for every irreducible $\str C\subseteq \str A$ it holds that $f\restriction_C$ is an embedding $\str C\rightarrow \str B$.

\begin{definition}[Completion]\label{defn:completion}
Let $L$ be a language, let $\str C,\str C'$ be $L$-structures. We say that $\str C'$ is a \emph{(strong) completion} of $\str C$ if there is an injective homo\-morphism-embedding $f\colon \str C\rightarrow \str C'$.
\end{definition}
All completions in this thesis will be strong, therefore we will sometimes omit the adjective.
 
\begin{definition}[Locally finite subclass~\cite{Hubicka2016}]\label{def:locallyfinite}
Let $L$ be a language, let $\mathcal R$ be a class of finite structures and let $\mathcal K\subseteq \mathcal R$ be a subclass of $\mathcal R$. We say that $\mathcal K$ is a \emph{locally finite subclass of $\mathcal R$} if for every $\str C_0 \in \mathcal R$ there exists an integer $n=n(\str C_0)$ such that for every $L$-structure $\str C$ there exists $\str C' \in \mathcal K$ which is a completion of $\str C$ provided that:
\begin{enumerate}
\item there exists a homomorphism-embedding from $\str C$ to $\str C_0$;
\item every irreducible substructure of $\str C$ is in $\mathcal K$; and
\item for every substructure $\str S\subseteq \str C$ such that $\str S$ has at most $n$ vertices there exists $\str S'\in \mathcal K$ which is a completion of $\str S$.
\end{enumerate}
\end{definition}

Now we can state the main result of~\cite{Hubicka2016}.
\begin{theorem}[Hubi\v cka--Ne\v set\v ril~\cite{Hubicka2016}, Theorem~2.2]\label{thm:hn}
Let $L$ be a language and let $\mathcal R$ be a class of finite irreducible $L$-structures which has the Ramsey property. Let $\mathcal K\subseteq \mathcal R$ be a locally finite subclass of $\mathcal R$ which has the strong amalgamation property and is hereditary (if $\str B\in \mathcal K$ and $\str A\subseteq \str B$, then $\str A\in \mathcal K$). Then $\mathcal K$ is Ramsey.
\end{theorem}

When one works with a relational language $L$, the class $\mathcal R$ is usually the class of all ordered finite $L$-structures, which is Ramsey by the Ne\v set\v ril--R\"odl theorem. For languages involving both relations and functions, we need an analogue of the Ne\v set\v ril--R\"odl theorem proved recently by Evans, Hubi\v cka and Ne\v set\v ril.
\begin{theorem}[Evans--Hubi\v cka--Ne\v set\v ril~\cite{Evans3}]\label{thm:NRclosures}
Let $L$ be a language and let $\mathcal K$ be a free amalgamation class of $L$-structures. Then $\overrightarrow{\mathcal K}$, the class of all linearly ordered structures from $\mathcal K$, is a Ramsey class.
\end{theorem}

\section{EPPA}\label{sec:eppa}
There is another combinatorial property with a direct analogue of Theorem~\ref{thm:hn} called EPPA (the extension property for partial automorphisms). We only give the minimum necessary background, for a broader overview of EPPA consult Siniora's PhD thesis~\cite{Siniora2}.

\begin{definition}[EPPA]\label{defn:eppa}
Let $L$ be a language and let $\mathcal C$ be a class of finite $L$-structures. We say that $\mathcal C$ has the \emph{extension property for partial automorphisms} (\emph{EPPA}) if for every $\str A\in \mathcal C$ there exists $\str B\in \mathcal C$ such that $\str A\subseteq \str B$ and for every \emph{partial automorphism} $f\colon  \str A\rightarrow \str A$ (that is, an isomorphism of substructures of $\str A$) there exists an automorphism $g$ of $\str B$ such that $f\subseteq g$. We call such $\str B$ an \emph{EPPA-witness for $\str A$}.
\end{definition}
EPPA is sometimes also called the \emph{Hrushovski property}, because Hrushovski was the first to prove that the class of all finite graphs has EPPA~\cite{hrushovski1992}. Since then the quest of finding classes with EPPA continued with a series of papers including~\cite{Herwig1995,herwig2000,hodkinson2003,solecki2005,vershik2008,Conant2015,Aranda2017,Hubicka2018metricEPPA,eppatwographs,Hubicka2018EPPA,HubickaSemigeneric,Konecny2019a}.

A distant analogue of the Hubi\v cka--Ne\v set\v ril theorem for EPPA is the Herwig--Lascar theorem which has recently been strengthened by Hubi\v cka, Ne\v set\v ril and the author~\cite{Hubicka2018EPPA}. Conveniently, we also gave a formulation very similar to Theorem~\ref{thm:hn} which makes it easier to apply both of them in the same paper.

Let $\str C$ and $\str C'$ be structures such that $\str C'$ is a completion of $\str C$. We say that $\str C'$ is an \emph{automorphism-preserving} completion of $\str C$ if the corresponding map is automorphism-preserving. We say that a locally finite subclass $\mathcal K\subseteq \mathcal R$ (see Definition~\ref{def:locallyfinite}) is \emph{automorphism-preserving} if the completions $\str C'$ can be chosen to be automorphism-preserving.

We can now state (a weaker form) of the main theorem from~\cite{Hubicka2018EPPA}.
\begin{theorem}[\cite{Hubicka2018EPPA}]\label{thm:herwiglascar}
Let $L$ be a finite language consisting of relations and unary functions and let $\mathcal E$ be a class of finite $L$-structures which has EPPA. Let $\mathcal K$ be a hereditary locally finite automorphism-preserving subclass of $\mathcal E$ which consists of irreducible structures and has the strong amalgamation property. Then $\mathcal K$ has EPPA.
\end{theorem}

Again, as for the Ramsey property, Theorem~\ref{thm:herwiglascar} has a form of implication for which we need a base class:

\begin{theorem}[\cite{Hubicka2018EPPA}]\label{thm:eppanr}
Let $L$ be a finite language consisting of relations and unary functions. Then the class of all finite $L$-structures has EPPA.
\end{theorem}

\chapter{The shortest path completion}\label{ch:shortestpath}
In this section we study the basic properties of the shortest path completion. They are then used implicitly in the rest of the thesis. The ideas used in this chapter are often generalisations of~\cite{Hubicka2017sauerconnant}; \cite{Hubicka2016}, \cite{Nevsetvril2007}, \cite{solecki2005} or~\cite{vershik2008} also proceed similarly.

\begin{definition}\label{defn:graphs}
Fix a \pocs{} $\ESemig$. An \emph{$\Semig$-edge-labelled graph} $\str{G}$ is a pair $(G,d)$ where $G$ is the \emph{vertex
set} and $d$ is a partial function from $G\choose 2$ to $M$. For simplicity, we shall write $d(u,v)$ instead of $d(\{u,v\})$ and keep in mind that the function $d$ is symmetric and undefined for $u=v$. A pair of vertices $u,v$ on which $d(u,v)$ is defined is called an \emph{edge} of $\str{G}$. We also call $d(u,v)$ the \emph{length of the edge} $u,v$.

If $\Semig$ is clear from context, we may write simply \emph{edge-labelled graph}. We denote by $\mathcal G^\Semig$ the class of all finite $\Semig$-edge-labelled graphs.
\end{definition}
An edge-labelled graph can also be seen as a relational structure. Hence the standard notions of homomorphism, embedding, and isomorphism extend naturally to edge-labelled graphs. (This is very important for this thesis!) We find it convenient
to use notation that resembles the standard notation of metric spaces.

An ($\Semig$-edge-labelled) graph $\str{G}$ is \emph{complete} if every pair of vertices
forms an edge; a complete $\Semig$-edge-labelled graph $\str{G}$ is called an \emph{$\Semig$-metric space} if it is an $\Semig$-metric space as defined in Definition~\ref{def:Mmetric}.
An $\Semig$-edge-labelled graph $\str{G}=(G,d)$ is \emph{$\Semig$-metric} (or \emph{metric} if $\Semig$ is clear from the context) if there exists an $\Semig$-metric
space $\str{M}=(G,d')$ such that $d(u,v)=d'(u,v)$ for every edge $u,v$ of $\str{G}$. Such a metric space $\str{M}$ is also called a \emph{(strong) $\Semig$-metric completion} of
$\str{G}$. An $\Semig$-edge-labelled graph which is not $\Semig$-metric is called \emph{non-$\Semig$-metric}.

We shall adopt the standard graph-theoretic notions, such as the notion of a cycle: An $\Semig$-edge-labelled graph $\str{G}$ is a ($\Semig$-edge-labelled) cycle if its edges form a cycle. 

Given an $\Semig$-edge-labelled graph $\str{G}$, the \emph{walk} from $u$ to $v$ is any sequence of vertices $u=v_1,v_2,\ldots, v_n=v\in \str{G}$
such that $v_i,v_{i+1}$ form an edge for every $1\leq i<n$. If the sequence contains no repeated vertices, it is a \emph{path}.
The $\Semig$-length (or simply a length) of a walk $\str{W}$ is defined and denoted as $$\|\str{W}\| = d(v_1, v_2)\oplus d(v_2,v_3)\oplus \cdots\oplus d(v_{n-1},v_n).$$

We say that $\str{G}$ is \emph{connected} if there exists a path from $u$ to $v$ for every choice of $u\neq v\in G$.

To avoid unnecessary notational complications, we shall sometimes treat a walk or a cycle as a (cyclic) sequence of elements of $\Semig$ which represent the lengths of the edges of the path/cycle $\str K$. In this case, we will say that $\str{K} = (c_1, c_2, \ldots, c_k)$, where $c_1, \ldots, c_k\in \Semig$ and the edges with lengths $c_i,c_{i+1}$ are adjacent (if $\str K$ is a cycle, we further want $c_1,c_k$ adjacent).

In~2007 Ne\v set\v ril proved that the class of all linearly ordered metric spaces has the Ramsey property~\cite{Nevsetvril2007}, while Solecki~\cite{solecki2005} and Vershik~\cite{vershik2008} independently proved that the class of all metric spaces has EPPA. All these results were based on the fact that a finite $\mathbb R^{>0}$-edge-labelled graph is metric if and only if it contains no non-metric cycle as a (non-induced) subgraph and that one can complete a metric edge-labelled graph to a metric space by letting the distance between every two vertices be the length of the shortest path connecting them.

As mentioned in Example~\ref{example1}, Sauer~\cite{Sauer2013b} studied the sets $S$ for which $\oplus_S$ is defined. He proved that $\oplus_S$ is associative (and hence $(S, \oplus, \leq)$ is a \pocs) if and only if the class of all finite metric spaces with distances from $S$ is an amalgamation class. Hubi\v cka and Ne\v set\v ril~\cite{Hubicka2016} gave Ramsey expansions of all such $S$-metric spaces (some of their results had been obtained before by Nguyen Van Th\'e~\cite{The2010}), again using the shortest path completion (where now the lengths of paths are measured using the $\oplus_S$ operation). Their techniques also directly imply the same results for classes of $S$-metric spaces where one further forbids ``short odd cycles'', that is, triangles with distances $a,b,c$ such that $a+b+c$ is odd and ``small enough''.

Conant~\cite{Conant2015} studied EPPA in the context of generalised metric spaces which are $\Semig$-metric spaces where $\Semig$ is a linearly ordered commutative monoid. Hubi\v cka, Ne\v set\v ril and the author~\cite{Hubicka2017sauerconnant} later found Ramsey expansions for all such spaces.

Braunfeld~\cite{Sam}, motivated by his classification of generalised permutation structures~\cite{Sam2}, found Ramsey expansions of $\Lambda$-ultrametric spaces for a distributive lattice $\Lambda = (\Lambda, \vee, \wedge, 0)$, which are essentially (see Remark~\ref{rem:sam}) $\Semig$-valued metric spaces, where $\Semig = (\Lambda\setminus\{0\}, \vee, \mleq)$ for $\mleq$ the standard lattice order.

Finally, Aranda, Bradley-Williams, Hubi{\v c}ka, Karamanlis, Kompatscher, Paw\-liuk and the author~\cite{Aranda2017, Aranda2017a, Aranda2017c} studied Ramsey expansion of metric spaces from Cherlin's catalogue of metrically homogeneous graphs~\cite{Cherlin2011b,Cherlin2013}. These are metric spaces with distances from $\{1, \ldots, \delta\}$ with $\delta\geq 3$, such that one can further forbid several families of cycles (such as cycles of short odd perimeter or cycles of long perimeter). A key ingredient was finding an explicit procedure to complete $[\delta]$-edge-labelled graphs. This was done by introducing a binary commutative operation $\oplus$ on $\{1, \ldots, \delta\}$ and an order of $\{1, \ldots, \delta\}$ and then in stages according to the order complete \emph{forks} (that is, triples of vertices $u,v,w$ such that $uv$ and $vw$ are edges and $uw$ is not an edge) using the $\oplus$ operation. An exposition can be found in the bachelor thesis of the author~\cite{Konecny2018bc}, see also Section~\ref{sec:methom}.

All the mentioned results, some rather directly, some not at all, satisfy some form of triangle inequality and admit some form of the shortest path completion.

Let $\str G$ be an $\Semig$-edge-labelled graph, $u,v\in G$ vertices of $\str G$ and $\mathcal W$ be a finite family of walks in $\str G$ from $u$ to $v$. Then we define $\inf(\mathcal W) = \inf_{\str W\in \mathcal W}\|\str W\|$.
\begin{definition}[Shortest path completion]\label{defn:shortestpath}
Let $\ESemig$ be a \pocs{} and $\str{G}=(G,d)$ be a finite connected $\Semig$-edge-labelled graph. Denote by $\mathcal P(u,v)$ the (finite) family of all paths in $\str G$ from $u$ to $v$. Assume that $\inf(\mathcal P(u,v))$ is defined for every $u\neq v$.

Define $d'(u,v)=\inf(\mathcal P(u,v))$ for every $u\neq v$. Then we call the complete $\Semig$-metric graph $\str{A}=(G,d')$ the \emph{shortest path completion} of $\str{G}$.
\end{definition}

Note that the shortest path completion assumes the existence of all infima in encounters. The motivation for the following definition is to be able to ensure that this assumption is satisfied. We also take care of the distributivity of $\inf$ and $\oplus$ which will be needed later when proving the strong amalgamation property.

For sets of distances $A,B$, denote by $A\oplus B$ the set $\{a\oplus b:a\in A, b\in B\}$. If $\str A$ is an $\Semig$-edge-labelled graph, $u,v,w$ are its distinct vertices, $\mathcal W$ is a family of walks from $u$ to $v$ and $\mathcal W'$ is a family of walks from $v$ to $w$, then by $\mathcal W\oplus \mathcal W'$ we mean the family of all walks from $u$ to $w$ which one can get by concatenating a walk from $\mathcal W$ with a walk from $\mathcal W'$ by the vertex $v$.



\begin{definition}[Disobedient cycles]\label{defn:disobedient}
Let $\ESemig$ be a \pocs{} and $\mathcal F$ be a family of $\Semig$-edge-labelled cycles. We say that $\mathcal F$ \emph{contains all disobedient cycles} if the following holds for every $\Semig$-metric space $\str A\in \MF$:

Let $u,v,w\in A$ be vertices of $\str A$, let $\mathcal W$ be a finite family of walks from $u$ to $v$ in $\str A$ and let $\mathcal W'$ be a finite family of walks from $v$ to $w$ such that every walk in $\mathcal W$ and $\mathcal W'$ has at least two edges. Then it holds that
\begin{enumerate}
\item $\inf(\mathcal W)$ and $\inf(\mathcal W')$ are defined; and
\item $\inf(\mathcal W \oplus \mathcal W') = \inf(\mathcal W)\oplus \inf(\mathcal W')$.
\end{enumerate}
\end{definition}

We want to show that the class of all finite $\Semig$-metric spaces which omit homomorphic images of members of $\mathcal F$ has the strong amalgamation property. For it we need to ensure that the amalgamation process and the shortest path completion do not introduce any cycles from $\mathcal F$. We ensure it by putting some more conditions on $\mathcal F$.
\begin{definition}[Omissible family]\label{defn:omissible}
Let $\ESemig$ be a \pocs{} and $\mathcal F$ be a family of $\Semig$-edge-labelled cycles. We say that $\mathcal F$ is \emph{omissible} if for every cycle $(a_1, a_2, \ldots, a_k) \in \mathcal F$ the following conditions hold:
\begin{enumerate}
\item For every $1\leq i\leq k$ it holds that $a_i \mlt \bigoplus_{j\neq i} a_j$ (that is, $\mathcal F$ contains no \emph{geodesic} or non-$\Semig$-metric cycles). \label{defn:omissible:metric}
\item For every $1\leq i\leq k$ and every family of paths $\mathcal P$ such that $\inf(\mathcal P) = a_i$ and no two members of $\mathcal P$ together form a non-metric cycle or a cycle from $\mathcal F$ there is $\str P\in \mathcal P$ such that the cycle $\str C'$ is in $\mathcal F$, where $\str C'$ is obtained from $\str C$ by replacing $a_i$ by the path $\str P$. We say that $\mathcal F$ is \emph{closed under inverse steps of the shortest path completion} or \emph{upwards closed}. \label{defn:omissible:upwards}
\item $\mathcal F$ is \emph{downwards closed}, that is, the following two conditions are satisfied: \label{defn:omissible:downwards}
\begin{enumerate}
\item For every $1\leq i < j\leq k$ such that $2\leq j-i\leq k-3$ it holds that one of the cycles $(a_i,\ldots,a_j)$ and $(a_1,\ldots,a_{i-1},a_{j+1},\ldots,a_k)$ is in $\mathcal F$ or is non-$\Semig$-metric (two edges of different lengths are also a non-$\Semig$-metric cycle); and
\item for every $i,j$ such that $1\leq i < j\leq  k$, $2\leq j-i\leq k-3$ and an arbitrary distance $c\in \Semig$ it holds that one of the cycles $(c, a_i, a_{i+1}, \ldots, a_j)$ and $(a_1, \ldots, a_{i-1}, c, a_{j+1}, \ldots, a_k)$ is in $\mathcal F$ or is non-$\Semig$-metric (if $c$ is too small or large).
\end{enumerate}
\end{enumerate}
\end{definition}
\begin{example}\label{example:shortodd}
For $\Semig = (\mathbb Z^{\geq 1}, +, \leq)$ and an arbitrary positive integer $p$, every family $\mathcal F_p$ consisting of all metric cycles of odd perimeter smaller than $p$ is omissible. Indeed, the perimeter being odd ensures that $\mathcal F_p$ contains no geodesic cycles. Because $\Semig$ is linearly ordered, upwards closedness amounts to checking that if one replaces the edge of length $a$ by a path of length $a$, the cycle has still short odd perimeter and for downwards closedness it is enough to realise that if one cuts an odd cycle in two (by, possibly, an edge of length $0$ which is interpreted as gluing two vertices together) one of them will have odd perimeter and either one will be non-metric, or both will have perimeter less than $p$.
\end{example}
The notion of omissible family of cycles is a generalisation of Example~\ref{example:shortodd}. Cherlin and Shi~\cite{Cherlin1996} proved that metric spaces without short odd cycles have the strong amalgamation property; Ramsey property and EPPA follow in the same way as they do when one does not forbid any short odd cycles.

The main result of this section is the following theorem, but the auxiliary lemmas will be used later as well.
\begin{theorem}\label{thm:shortestpath}
Let $\ESemig$ be a partially ordered commutative semigroup and let $\mathcal F$ be an omissible family of $\Semig$-edge-labelled cycles containing all disobedient ones. Let $\str{G}$ be a finite \textbf{connected} $\Semig$-edge-labelled graph such that it contains no homomorphic image of any member of $\mathcal F$.
\begin{enumerate}
\item \label{thm:shortestpath:defined} Suppose that $\str G$ is metric and let $\str G' = (G, d')$ be its shortest path completion. Then $\str G'$ is well defined and it is a completion of $\str G$ in $\mathcal M_\Semig \cap \Forb(\mathcal F)$ (in the sense of Definition~\ref{defn:completion}).
\item \label{thm:shortestpath:aut} If $\str G$ is metric and the shortest path completion $\str G'=(G,d')$ is defined then $\Aut(\str G) = \Aut(\str G')$.
\item \label{thm:shortestpath:amalg} $\MF$ is a strong amalgamation class.
\end{enumerate} 
\end{theorem}
The first part says that the shortest path completion is ``complete'' (can complete everything which can be completed at all), the second part says that it is ``canonical'' (this will be useful in Section~\ref{sec:eppaproof}) and the third part suggests that we are interested in ``reasonable'' classes.

In order to prove Theorem~\ref{thm:shortestpath}, we first prove several auxiliary results. All the proofs are straightforward.
\begin{lemma}\label{lem:shortestpath:1}
Let $\Semig$, $\mathcal F$ and $\str G$ be as in Theorem~\ref{thm:shortestpath}.
\begin{enumerate}
\item \label{cor:suaer1:-1} Let $X\subseteq \Semig$ be a set of distances and $a\in \Semig$ be a distance such that $\inf X$ is defined and $\inf X \nmgeq a$. Then there is $x\in X$ with $x\nmgeq a$.

\item \label{cor:suaer1:0} If $\str{G}$ is $\Semig$-metric, then the shortest path completion is defined for $\str{G}$ (that is, for every pair of vertices $u,v\in G$ the infimum of the $\Semig$-lengths of any family of paths from $u$ to $v$ is defined).
\end{enumerate}
\end{lemma}
\begin{proof}\leavevmode
\paragraph{\ref{cor:suaer1:-1}.} Suppose that this is not true, that is, $x\mgeq a$ for every $x\in X$. Then $a\mleq \inf X$ which is a contradiction.
\medskip

\paragraph{\ref{cor:suaer1:0}.} Suppose that the shortest path completion of $\str{G}$ is not defined. That means that there is a pair of vertices $u,v\in G$ and a collection of paths $\mathcal P$ from $u$ to $v$ such that $\inf_{P\in \mathcal P} \|P\|$ is not defined. Take a minimal such $\mathcal P$. As $\str{G}$ omits all homomorphic images from $\mathcal F$ which contains all disobedient cycles, it means that one of the paths has length one, hence it is the edge $uv$. From minimality of $\mathcal P$ it follows that $i = \inf_{\str P\in \mathcal P \setminus \{uv\}} \|\str P\|$ is defined, but this in particular means that $i$ and $d(u,v)$ are incomparable. And from part~\ref{cor:suaer1:-1} we get a path $\str P\in \mathcal P$ such that $\|\str P\|\nmgeq d(u,v)$, a contradiction with $\str G$ being $\Semig$-metric.
\end{proof}

All parts of the following lemma will be used not only in the proof of Theorem~\ref{thm:shortestpath}, but also implicitly in the rest of the thesis.
\begin{lemma}\label{lem:shortestpath:2}
Let $\Semig$, $\mathcal F$ and $\str G$ be as in Theorem~\ref{thm:shortestpath}.
\begin{enumerate}
\item \label{cor:suaer1:1} If defined, the shortest path completion $\str{G}'=(G, d')$ of $\str G$ is an $\Semig$-metric space. 

\item \label{cor:suaer1:1.5} A cycle $\str C = (a_1,\ldots,a_k)$ such that $\str C\notin \mathcal F$ is non-$\Semig$-metric if and only if for some $i$ it holds that $a_i\nmleq \bigoplus_{j\neq i} a_j$.

\item \label{cor:suaer1:3} $\str{G}$ contains a homomorphic image of a non-$\Semig$-metric cycle if and only if it contains a non-$\Semig$-metric cycle as a (non-induced) subgraph.

\item \label{cor:suaer1:2} $\str{G}$ is $\Semig$-metric if and only if it contains no homomorphic image of a non-$\Semig$-metric cycle $\str{C}$.

\item \label{cor:suaer1:2.5} If $\str{G}$ is $\Semig$-metric and $\str{G}' = (G, d')$ is its shortest path completion, then $d'(u,v) = d(u,v)$ whenever $d(u,v)$ is defined (i.e. $\str{G}'$ is a completion of $\str G$ in the sense of Definition~\ref{defn:completion}).
\end{enumerate}
\end{lemma}
\begin{proof}\leavevmode
\paragraph{\ref{cor:suaer1:1}.}
We need to verify that $d'$ satisfies the triangle inequality. Take any three vertices $u,v,w \in \str{G}$. Combine the paths
from $\mathcal P(u,w)$ and $\mathcal P(w,v)$ in every possible way to get a family of walks $\mathcal W$ from $u$ to $v$ in $\str{G}$. But then $\inf(\mathcal W) = \inf(\mathcal P(u,w)\oplus \mathcal P(w,v)) = d'(u,w)\oplus d'(w,v)$, because $\mathcal F$ contains all disobedient cycles. Clearly $\inf(\mathcal P(u, v)) \mleq \inf(\mathcal W)$ as one can get a path from a walk by removing the ``loops'', thus we get $d'(u,v) = \inf(\mathcal P(u, v)) \mleq \inf(\mathcal W)$.

\medskip
\paragraph{\ref{cor:suaer1:1.5}.} Clearly such a cycle is non-$\Semig$-metric. On the other hand, if for every $i$ it holds that $a_i\mleq \bigoplus_{j\neq i}a_j$, then the shortest path completion (which is defined as $\str C\notin\mathcal F$) is a completion of $\str C$ in the sense of Definition~\ref{defn:completion}; it is an $\Semig$-metric space by part~\ref{cor:suaer1:1} and it preserves the edges of $\str C$ by the assumption on $\str C$.

\medskip
\paragraph{\ref{cor:suaer1:3}.} One implication is trivial. The other follows from $\oplus$ being monotone with respect to $\mleq$: Denote the cycle $\str C = (a_1, \ldots, a_n)$ with $a_1 \nmleq \bigoplus_{i=2}^n a_i$. Then it is enough to take the minimal subcycle of the homomorphic image of $\str C$ containing the edge $a_1$.

\medskip
\paragraph{\ref{cor:suaer1:2}.} If $\str G$ contains a homomorphic image of a non-$\Semig$-metric cycle then by part~\ref{cor:suaer1:3} it contains a non-$\Semig$-metric cycle as a non-induced substructure. But a completion of $\str G$ is in particular a completion of that cycle, which contradicts it being non-$\Semig$-metric.

We claim that if $\str G$ is non-$\Semig$-metric, then there are vertices $u,v\in G$ such that $d(u, v) \nmleq \inf(\mathcal P(u,v))$ where $\mathcal P(u,v)$ is the family of all paths from $u$ to $v$. By part~\ref{cor:suaer1:-1} of Lemma~\ref{lem:shortestpath:1} the statement then follows.

If $\str G$ is non-$\Semig$-metric and $\str G' = (G, d')$, its shortest path completion, is defined, from non-metricity we get vertices $u\neq v\in G$ with $d'(u,v) \mlt d(u,v)$. But that means
that the infimum of the $\Semig$-lengths of $\mathcal P(u,v)$ is not greater than or equal to $d(u,v)$. Otherwise $\str G'$ is undefined and we proceed as in the proof of part~\ref{cor:suaer1:0} of Lemma~\ref{lem:shortestpath:1}.

\medskip
\paragraph{\ref{cor:suaer1:2.5}.}
Clearly $d'(u,v)\mleq d(u,v)$, as the edge $u,v$ is a path between $u$ and $v$. So it suffices to show $d'(u,v)\mgeq d(u,v)$.
We show a stronger claim: if $\str{G}''=(G, d'')$ is a completion of $\str{G}$ then $d''(u,v)\mleq d'(u,v)$ for every $u\neq v\in G$.

Suppose, for a contradiction, that there are vertices $u\neq v\in G$ such that for some completion $\str{G}''$ of $\str{G}$ it holds that $d''(u,v)\nmleq d'(u,v)$. By definition of $d'$, there is a family of paths $\mathcal P(u,v)$ in $\str{G}$ with $d'(u,v)=\inf(\mathcal P(u,v))$. But these paths are also present in $\str{G}''$ and hence $\str{G}''$ contains a non-$\Semig$-metric cycle by part~\ref{cor:suaer1:-1} of Lemma~\ref{lem:shortestpath:1}, which is a contradiction.
\end{proof}

Now we are ready to prove Theorem~\ref{thm:shortestpath}.
\begin{proof}[Proof of Theorem~\ref{thm:shortestpath}]\leavevmode
\paragraph{\ref{thm:shortestpath:defined}.} The fact that $\str G'$ is well-defined follows from Lemma~\ref{lem:shortestpath:1}; it is a metric completion of $\str G$ by Lemma~\ref{lem:shortestpath:2}. It remains to check that $\str G'\in \Forb(\mathcal F)$.

Suppose for a contradiction that there is $\str C\in\mathcal F$ and a homomorphism $f\colon\str C\rightarrow \str G'$. For every edge $uv\in \str C$ there is a family of paths $\mathcal P_{uv}$ in $\str G$ such that $\inf(\mathcal P_{uv}) = d'(f(u),f(v)) = d_{\str C}(u,v)$. Because $\mathcal F$ is closed under inverse steps of the shortest path completion, there is $\str P$ in $\mathcal P_{uv}$ such that the cycle which one gets from $\str C$ by exchanging the edge $uv$ by the path $\str P$ is in $\mathcal F$. Doing this for every edge $uv\in \str C$, we get $\str C'\in \mathcal F$ with a homomorphism to $\str G$, which is a contradiction.

\medskip
\paragraph{\ref{thm:shortestpath:aut}.}
Let $\alpha$ be an automorphism of $\str G$. We need to prove that for every $u,v\in G$ it holds that $d'(u,v) = d'(\alpha(u), \alpha(v))$. By definition there is a family $\mathcal P(u,v)$ of paths from $u$ to $v$ such that $\inf(\mathcal P(u,v)) = d'(u,v)$ for every pair of vertices $u,v$. Because $\alpha$ is automorphism, there also has to be a family of paths from $\alpha(u)$ to $\alpha(v)$ with the same infimum and vice versa, hence indeed $d'(u,v) = d'(\alpha(u), \alpha(v))$.

\medskip
\paragraph{\ref{thm:shortestpath:amalg}.}
Now it is easy to see that part~\ref{cor:suaer1:2} of Lemma~\ref{lem:shortestpath:2} also holds for $\str G$ which is not connected, because it can be turned into a connected one by adding new edges connecting individual components without introducing any new cycles.

We will observe that for every $\str{A},\str{B}_1,\str{B}_2\in \mathcal M_\Semig\cap \Forb(\mathcal F)$
the free amalgamation of $\str{B}_1$ and $\str{B}_2$ over $\str{A}$ contains no
embedding of any non-$\Semig$-metric cycle and no embedding of any cycle from $\mathcal F$. Then we can take the amalgam to be the shortest path completion of the free amalgam (possibly after connecting the components if $\str{A} = \emptyset$).

First note that if we let $\mathcal F'$ be the union of $\mathcal F$ and the family of all non-$\Semig$-metric cycles, then $\mathcal F'$ is still downwards closed (cf. Definition~\ref{defn:omissible}). Suppose for a contradiction that there is a cycle $\str F\in \mathcal F'$ with a homomorphism $f$ to the free amalgam of $\str B_1$ and $\str B_2$ over $\str A$ and among all such situations take one where $\str F$ has the smallest number of vertices.

In the range of $f$ there must be vertices from both $B_1\setminus A$ and $B_2\setminus A$. And because cycles are 2-connected, there are at least two vertices $u\neq v\in F$ with $f(u), f(v)\in A$. Either $f(u)=f(v)$, but then from downwards closedness of $\mathcal F'$ we get a contradiction with minimality of $|F|$. Otherwise $f(u)\neq f(v)$ and because $\str A$ is complete, the distance $d_{\str A}(f(u), f(v))$ is defined. Denote by $\str F'$ and $\str F''$ the two cycles one gets from $\str F$ by adding the edge $uv$ of length $d_{\str A}(f(u), f(v))$. From downwards closedness of $\mathcal F'$ it follows that one of them is in $\mathcal F'$, so suppose without loss of generality that $\str F'\in \mathcal F'$. Clearly $f\restriction_{F'}$ is a homomorphism from $\str F'$ to the amalgam, which is a contradiction with minimality of $|F|$.
\end{proof}

\section{Henson constraints}\label{sec:henson}
In his catalogue of metrically homogeneous graphs~\cite{Cherlin2011b}, Cherlin allows to forbid arbitrarily large metric spaces containing only distances $1$ and $\delta$ (when certain conditions are satisfied) and he calls them \emph{Henson constraints}. It motivates the following paragraphs.

\begin{definition}
Let $\ESemig$ be a \pocs{}. We say that $a\in \Semig$ is \emph{reducible} if there are $b,c\in \Semig$ such that either $a = b\oplus c$, or $a=\inf(b,c)$ and $a\neq b,c$. Otherwise we call $a$ \emph{irreducible}.

Let further $\mathcal F$ be an $\Semig$-omissible family containing all $\Semig$-disobedient cycles. We say that $a\in \Semig$ is \emph{irreducible with respect to $\mathcal F$} if whenever $\str A$ is an $\Semig$-edge-labelled $\Semig$-metric graph from $\Forb(\mathcal F)$, $\str A'$ is its shortest path completion and $u,v\in A$ are vertices such that $d_{\str A'}(u,v) = a$ then $d_{\str A}(u, v) = a$. Otherwise $a$ is \emph{reducible with respect to $\mathcal F$}.
\end{definition}
In other words, $a\in\Semig$ is irreducible with respect to $\mathcal F$ if the shortest path completion of $\Semig$-metric graphs from $\Forb(\mathcal F)$ never introduces distance $a$.
\begin{example}
If $\Semig$ is a monoid with neutral element 0, then all elements are reducible (as $a=a\oplus 0$). For $\Semig = (\{a,b,c,M\},\oplus,\mleq)$ where $x\oplus y = M$ for every $x,y\in \Semig$ and $x\mleq M$ for every $x\in \Semig$ the irreducible elements are $a$, $b$ and $c$.

For a non-trivial example of elements irreducible with respect to $\mathcal F$ see Section~\ref{sec:methom} where in the magic semigroup $\magicsemig$, unless $M = \delta$ or $C=2\delta+2$, the irreducible elements with respect to the union of $K_1$-, $K_2$- and $\magicsemig$-metric $C$-cycles are precisely $1$ and $\delta$.
\end{example}

Clearly if an element is irreducible, then it is irreducible with respect to any family $\mathcal F$. Given a \pocs{} $\Semig$ and an $\Semig$-omissible family $\mathcal F$ containing all $\Semig$-disobedient cycles we say that $\mathcal H$ is a \emph{family of Henson constraints} if $\mathcal H$ consists of $\Semig$-metric spaces only using distances irreducible with respect to $\mathcal F$.

In proving Theorem~\ref{thm:shortestpath}, we in fact proved the following stronger result.
\begin{observation}\label{obs:hensoncompletion}
Let $\Semig$, $\mathcal F$ and $\str G$ be as in Theorem~\ref{thm:shortestpath} and let $\mathcal H$ be a family of Henson constraints. If $\str G$ does not contain any member of $\mathcal H$ then neither does its shortest path completion. In particular, $\MF\cap \Forb(\mathcal H)$ is a strong amalgamation class.
\end{observation}
Note that since $\mathcal H$ consists of irreducible structures, $\Forb(\mathcal H)$ is the class of all finite structures which do not embed any member of $\mathcal H$.

\begin{remark}
We think that adding Henson constraints explicitly into all the statements would cause unnecessary notational complications without much gain, but still we think that it is useful to remark that such a stronger statement follows from our proofs.
\end{remark}

\chapter{Locally finite description of the unordered classes}\label{ch:orderless}
For Theorem~\ref{thm:hn}, one needs to be able to describe structures with a completion to the given class by a bounded number of forbidden substructures. However, as we have observed in Section~\ref{sec:outline}, this is not always true for  $\mathcal M_\Semig\cap \Forb(\mathcal F)$.

First we make a small detour into a little more advanced model theory topic, the elimination of imaginaries. It very well describes the fact that some semigroups have arbitrarily large non-metric cycles and suggests a way of dealing with it (which we indeed take in Section~\ref{sec:lstar}). However, the notions introduced here will not be used in the proofs, only in their motivation. We only give the minimum necessary definitions, for more see~\cite{Hodges1993}.

\begin{definition}[Definable equivalences]
Let $\str A$ be an $L$-structure and $R\subseteq A^n$ be a relation. We say that $R$ is a \emph{definable} if there is a first-order formula $\varphi(\bar{x})$ (in language $L$) such that for every $\bar{x}\in A^n$ it holds that $\bar{x}\in R$ if and only if $\varphi(\bar{x})$ holds. A \emph{definable equivalence} is a definable relation which is an equivalence. A \emph{definable function} is a definable relation $f\subseteq A^{k+\ell}$ such that for every $\bar{x}\in A^k$ there is exactly one $\bar{y}\in A^\ell$ such that $\bar{x}^\frown\bar{y}\in f$ (where $\bar{x}^\frown\bar{y}$ is the concatenation of $\bar{x}$ and $\bar{y}$). 
\end{definition}

\begin{definition}[Elimination of imaginaries]
Let $\str A$ be an $L$-structure. Every equivalence class of every definable equivalence relation on $A^n$ is an \emph{imaginary element}. We say that $\str A$ \emph{eliminates imaginaries} if for every definable equivalence relation $E$ on $A^n$ there is a definable\footnote{Note that, in model theory, these definitions are applied on infinite structures. In order to avoid some pathologies, when talking about formulas and structures in the following definitions, we should actually talk about classes of structures and require the formulas to be fixed for the whole class.} function $f\colon A^n\to A^k$ such that for every $\bar{u},\bar{v}\in A^n$ it holds that
$$\bar{u}E\bar{v} \Leftrightarrow f(\bar{u}) = f(\bar{v}).$$
\end{definition}

\begin{example}\label{ex:imaginaries}
Consider an arbitrary expansion of the ordered natural numbers $\mathbb N$. It eliminates imaginaries: Indeed, let $E\subseteq \mathbb N^{2n}$ be an arbitrary definable equivalence. To each tuple $\bar x\in \mathbb N^n$ we can definably assign the lexicographically smallest tuple $\bar y\in \mathbb N^n$ such that $\bar x E \bar y$.

On the other hand, the $\Semig$-valued metric spaces for $\Semig = (\{1,2\}, \max, \leq)$, which we have seen as an example of arbitrary large non-metric cycles, do not eliminate imaginaries due to the ``being in distance 1'' definable equivalence. A possible way to fix this is to add a new vertex for each ``ball of diameter 1'' and link every ``original vertex'' to its ``ball vertex'' by an explicitly added unary function.
\end{example}

The goal of this chapter is to generalise the construction from Example~\ref{ex:imaginaries} and expand the classes $\mathcal M_\Semig\cap \Forb(\mathcal F)$ to obtain classes (isomorphic with $\mathcal M_\Semig\cap \Forb(\mathcal F)$ as categories) which have a locally finite description.

In order to do that, we first (in Sections~\ref{sec:blocks} and~\ref{sec:blockequivalences}) study the \emph{block structure} of partially ordered semigroups. Then (in Section~\ref{sec:lstar}) we define the $L_\Semig^\star$-expansion where we add new vertices as explicit representatives of each \emph{block equivalence class} and link the original vertices to them by unary functions, thereby eliminating imaginaries. In Section~\ref{sec:nonimportant} we introduce an approximation of maxima of blocks (because true maxima might not exist for infinite blocks) which then allow us to select from each non-$\Semig$-metric cycle a bounded number of \emph{important} edges and then in Section~\ref{sec:locallyfinitecompletion} we prove that indeed the $L_\Semig^\star$-expansions have a locally finite description.

In Chapter~\ref{ch:withorder} we further use the results of this chapter to obtain Ramsey expansions of $\mathcal M_\Semig\cap \Forb(\mathcal F)$. Besides the Ramsey property, the results of this chapter by themselves suffice to prove EPPA (Section~\ref{sec:eppaproof}).

\section{Blocks and block lattice}\label{sec:blocks}
Let $\ESemig$ be a \pocs. For a positive integer $n$ and $a\in \Semig$ we define $n\times a$ as 
$$n\times a= \underbrace{a\oplus a \oplus \cdots \oplus a}_\text{$n$ times}.$$
We say that $\Semig$ is \emph{archimedean} if for every $a,b\in \Semig$ there is a positive integer $n$ such that $n\times a \mgeq b$.

\begin{example}
There are many examples of both archimedean and non-\-archi\-medean partially ordered commutative semigroups which everyone encounters in their everyday life. All of $(\mathbb Z^{>0}, +, \leq)$, $(\mathbb R^{>0}, +, \leq)$ or $(\mathbb R^{>1}, \cdot, \leq)$ are archimedean. On the other hand the multiplicative semigroup $(\mathbb Z^{>0}, \cdot, \mid)$, where by $\mid$ we mean the ``is a divisor of'' relation, is non-archimedean because, for example, no power of $2$ is divisible by $3$.

Note that the monoid $(\mathbb N, +, \leq, 0)$, when treated as a \pocs, is also non-archimedean, because $n\times 0 = 0$ for every $n$.
\end{example}

As we will show later, the definable equivalences on $\Semig$-metric spaces are related to archimedean subsemigroups of $\Semig$.
The following is a generalisation of a definition by Sauer~\cite{Sauer2013}.
\begin{definition}\label{defn:block}
Let $\ESemig$ be a \pocs. A \emph{block} $\Block$ of $\Semig$ is either a subset of $\Semig$ such that it induces a maximal archimedean subsemigroup of $\Semig$ or a special block $\str 0$ (corresponding to the empty set).
\end{definition}
The introduction of the block $\str 0$ will be useful later. Note that $\str 0$ is added even if $\Semig$ contains a neutral element which one would intuitively expect to represent the identity (it does not). The basic properties of blocks can be summarized as follows.
\begin{lemma}
\label{lem:blocks}
Given a partially ordered commutative semigroup $\ESemig$ it holds that:
\begin{enumerate}
\item\label{lem:blocks:unique} For every $a\in \Semig$ there exists a unique block $\Block(a)$ containing $a$. 
\item\label{lem:blocks:sameblock} Let $a,b\in \Semig$. Then $a,b$ are in the same block if and only if there exist $m, n$ such that $m\times a \mgeq b$ and $n\times b \mgeq a$.
\end{enumerate}
\end{lemma}
\begin{proof}
Let $$\Block(a) = \left\{b \in \Semig \mid (\exists n)(n\times a \mgeq b) \land (\exists n)(n\times b \mgeq a)\right\}.$$ It is easy to check that $(\Block(a),\oplus,\mleq)$ is an archimedean subsemigroup of $\Semig$ containing $a$. Maximality and uniqueness follow from the fact that no $b\in \Semig\setminus \Block(a)$ can be in the same archimedean subsemigroup as $a$. Parts~\ref{lem:blocks:unique} and~\ref{lem:blocks:sameblock} now follow.
\end{proof}
Given a partially ordered commutative semigroup $\ESemig$ and $a\in \Semig$, we will always denote by $\Block(a)$ the unique block of $\Semig$ containing $a$ given by part~\ref{lem:blocks:unique} of Lemma~\ref{lem:blocks}.

\begin{lemma}\label{lem:blockorder}
Let $\ESemig$ be a partially ordered commutative semigroup and $\Block$ be a block of $\Semig$. Whenever $a,c\in \Block$ and $b\in \Semig$ such that $a\mlt b \mlt c$, then $b\in \Block$.
\end{lemma}
\begin{proof}
Take arbitrary $a,c\in\Block$ and $b\in \Semig$ with $a\mlt b \mlt c$. As $a,c\in\Block$, there is $n$ such that $n\times a\mgeq c$. But then also $n\times a\mgeq b$ and hence by part~\ref{lem:blocks:sameblock} of Lemma~\ref{lem:blocks} $a$ and $b$ are in the same block.
\end{proof}
If $\mleq$ is a linear order, Lemma~\ref{lem:blockorder} means that blocks form intervals. And this motivates the following definition:

\begin{definition}\label{defn:blockorder}
Let $\ESemig$ be a partially ordered commutative semigroup. By the same symbol $\mleq$ we denote the order on blocks of $\Semig$ putting $\Block \mleq \Block'$ if for every $a\in \Block$ there is $b\in\Block'$ such that $a\mleq b$.
\end{definition}
Note also that $\str 0\mleq \Block$ for every block $\Block$. By Lemma~\ref{lem:blockorder} $\mleq$ is a partial order of blocks of $\Semig$. Note that an equivalent definition could just ask for a pair $a\in \Block, b\in \Block'$ with $a\mleq b$.

Recall that a structure $(A, \vee, \wedge)$ is a \emph{lattice} if $\vee$ (\emph{join}) and $\wedge$ (\emph{meet}) are binary operations on $A$ satisfying the following equations.

\begin{align*}
a\wedge a &= a 									&	 a\vee a &= a\\
a\wedge b &= b\wedge a 							&	 a\vee b &= b\vee a\\
a\wedge (b\wedge c) &= (a\wedge b) \wedge c 	&	 a\vee (b\vee c) &= (a\vee b) \vee c\\
\\
a\wedge(a\vee b) &= a 							&	 a\vee(a\wedge b) &= a.
\end{align*}

Recall also that if $(B,\leq)$ is a partial order, then for $a,b\in B$ we call $c\in B$ their infimum and write $c=\inf(a,b)$ if $c\leq a, b$ and for every $x\in B$ such that $x\leq a,b$ it holds that $x\leq c$ and analogously we call $d\in B$ their supremum ($d=\sup(a,b)$) if $d\geq a,b$ and for every $x\geq a,b$ it holds that $x\geq d$.

It is a well-known fact that lattices and partial orders where all infima and suprema are defined are in 1-to-1 correspondence (in one direction just let $\vee$ be the supremum and $\wedge$ the infimum, in the other direction put $a\leq b$ if $a\wedge b=a$). Also from the existence of infima of pairs we get the existence of infima of any finite sets, the same for suprema.

In Definition~\ref{defn:blockorder} we introduced the order $\mleq$ of blocks, which is inherited from the order $\mleq$ of the elements of $\Semig$. It is natural to study the infima and suprema in the block order. For suprema it is quite straightforward. Let $\Block_1, \Block_2$ be blocks of $\Semig$ and take arbitrary $a\in \Block_1$ and $b\in \Block_2$. Then $\Block(a\oplus b)$ is the supremum of $\Block_1, \Block_2$: Indeed, clearly $\Block(a\oplus b)\mgeq \Block_1, \Block_2$. And if $\Block'\mgeq \Block_1, \Block_2$, then in particular there are $x_a, x_b\in \Block'$ with $x_a\mgeq a$ and $x_b\mgeq b$, hence $x_a\oplus x_b\mgeq a\oplus b$, thus $\Block'\mgeq \Block(a\oplus b)$ and we are done. We denote the join of blocks by $\vee$.

On the other hand, not all infima of blocks need to be defined. However, in our applications we will need the block order to be a lattice (in fact, a distributive lattice) and thus we choose to denote the block infima by the $\meet$ symbol nonetheless. From the definition of $\mleq$ it follows that a block $\Block\neq \str 0$ is the meet of blocks $\Block_1,\Block_2\neq \str 0$ if and only if $\Block\mleq \Block_1, \Block_2$ and for all $a\in \Block_1, b\in \Block_2$ and every $x\in \Semig$ such that $x\mleq a,b$ there is $c\in \Block$ such that $c\mgeq x$.

\begin{example}[Running example --- introduction]
In order to obtain a Ramsey expansion, we are going to need to work with several expansions of $\mathcal M_\Semig\cap \Forb(\mathcal F)$. While all of them are adaptations of rather standard concepts (eliminating imaginaries and convex ordering), we aim for this text to be accessible to broader audience. Therefore we develop these notions and classes without referring to the model-theoretical concepts and we also find it helpful to illustrate them on an example.

Let $\ExSemig = (\mathbb N^3, +, \leq)$ be the \pocs{} whose vertices are triples of natural numbers (that is, non-negative integers), the addition is coordinate-wise and $(a,b,c)\leq (u,v,w)$ if and only if $a\leq u$, $b\leq v$ and $c\leq w$. In other words, $\ExSemig$ is the third power of the standard monoid of natural numbers.

$\ExSemig$ has nine blocks: $\str 0$ and eight ``standard'' blocks, each corresponding to a subset $I\subseteq \{1,2,3\}$ by $$\Block_I = \{(x_1, x_2, x_3)\in \mathbb N^3; x_i\neq 0\Leftrightarrow i\in I\}.$$ The block lattice is distributive and is isomorphic to the subset lattice of a three-element set with a new identity added (to accommodate $\str 0$).

Note that in the $\ExSemig$-metric spaces there can be distinct vertices in distance $(0,0,0)$.
\end{example}

\begin{definition}\label{defn:meetreducible}
Let $\ESemig$ be a \pocs{}. We say that a block $\Block$ is \emph{meet-reducible} if there are blocks $\Block_1, \Block_2$ such that $\Block=\Block_1\meet\Block_2$ and $\Block\notin \{\Block_1, \Block_2\}$. Otherwise $\Block$ is \emph{meet-irreducible}.
\end{definition}

Let $R_\Semig$ be the set of all non-maximal meet-reducible blocks of $\Semig$ and $I_\Semig$ the set of all non-maximal meet-irreducible blocks of $\Semig$.

\begin{example}[Running example --- $I_\ExSemig$ and $R_\ExSemig$]
$I_\ExSemig = \{\str 0, \Block_{\{1,2\}}, \Block_{\{1,3\}}, \Block_{\{2,3\}}\}$, 
$R_\ExSemig = \{\Block_{\emptyset}, \Block_{\{1\}}, \Block_{\{2\}}, \Block_{\{3\}}\}$. $\Block_{\{1,2,3\}}$ is the maximum block of $\ExSemig$.
\end{example}

Later it will turn out to be useful to only represent meet-irreducible blocks. The following lemma shows that it is in some sense enough.
\begin{lemma}\label{lem:meetirreducible}
Let $\ESemig$ be a \pocs{} with finitely many blocks where the meet of every pair of blocks is defined. Let $\Block$ be a block of $\Semig$. Then there is a set $B\subseteq I_\Semig$ such that $\Block$ is the meet of $B$.
\end{lemma}
\begin{proof}
If $\Block\in I_\Semig$, then $B=\{\Block\}$ is a suitable choice, if $\Block$ is the maximum block, then take $B=\emptyset$. Enumerate all blocks as $\Block_1, \ldots, \Block_b$ such that there are no $i<j$ with $\Block_i\mleq \Block_j$ and let $i$ be the smallest integer such that the statement does not hold for $\Block_i$.

It follows that $\Block_i\in R_\Semig$, hence there are two blocks $\Block_j, \Block_k$ different from $\Block_i$ such that $\Block_j\meet \Block_k = \Block_i$. In particular, $\Block_i\mleq \Block_j, \Block_k$, hence $j,k<i$. Thus, by minimality of $i$, we have $B_j, B_k\subseteq I_\Semig$ with $\bigwedge B_j = \Block_j$ and $\bigwedge B_k = \Block_k$. It follows that $\Block = \bigwedge B$ for $B=B_j\cup B_k$, a contradiction.
\end{proof}

\section{Block equivalences and types}\label{sec:blockequivalences}
\begin{definition}
Let $\ESemig$ be a \pocs{}, $\str{A}$ be an $\Semig$-metric space and $\Block$ be a block of $\Semig$. 

\begin{enumerate}
\item A \emph{block equivalence $\sim_{\Block}$}
on vertices of $\str{A}$ is given by $u\sim_{\Block} v$ whenever there exists $a\in \Block$ such that $d(u,v)\mleq a$.

\item A {\em ball of diameter $\Block$} in $\str{A}$ is any equivalence class of $\sim_\Block$ in $\str{A}$.
\end{enumerate}
\end{definition}

To verify that for every block $\Block$ the relation $\sim_{\Block}$ is indeed an equivalence relation it suffices to check transitivity.  Given a triangle with distances $a,b,c$, if there exist $a'\in \Block$ and $b'\in \Block$ such that $a\preceq a'$ and $b\preceq b'$, it also holds that $c\preceq a\oplus b\preceq a'\oplus b'\in \Block$. Note that $u\sim_{\str 0} v$ if and only if $u=v$. It will be convenient later in Chapter~\ref{ch:withorder} to have the identity represented by a block.

For a finite set of blocks $B$ we denote the meet of $B$ as $\bigwedge B = B_1 \meet \cdots \meet B_k$ if it exists.

\begin{observation}\label{obs:ballintersection}
Let $\ESemig$ be a \pocs{} and $\str{A}$ be an $\Semig$-metric space. Then for every two blocks $\Block_1, \Block_2\neq \str 0$ such that their meet is defined and different from $\str 0$, every ball $B_1$ of diameter $\Block_1$ and every ball $B_2$ of diameter $\Block_2$ it holds that $B_1\cap B_2$ is either empty or it is a ball of diameter $\Block_1\meet\Block_2$.
\end{observation}
\begin{proof}
Denote $B = B_1\cap B_2$ and suppose that $B$ is nonempty. We need to prove that $B$ is a ball of diameter $\Block_1\meet\Block_2$.

Let $u,v\in A$ be vertices of $\str A$ such that $u\sim_{\Block_1} v$ and $u\sim_{\Block_2} v$. Denote $a=d_{\str A}(u,v)$. There are $b_1\in \Block_1$ and $b_2\in \Block_2$ such that $a\mleq b_1,b_2$, but then by the definition of meet there is $c\in \Block_1\meet\Block_2$ such that $c\mgeq a$ and hence $u\sim_{\Block_1\meet\Block_2} v$. This means that for every $u,v\in B$ it holds that $u\sim_{\Block_1\meet\Block_2} v$.

On the other hand if there are $u\in B$ and $v\in A$ such that $u\sim_{\Block_1\meet\Block_2} v$, then in particular $u\sim_{\Block_i} v$ for $i\in \{1,2\}$, so $v\in B_1\cap B_2$.
\end{proof}

\medskip
The ultimate goal of this chapter is to introduce an expansion of $\MF$ which satisfies the conditions of Theorems~\ref{thm:hn} and~\ref{thm:herwiglascar} and we want to do that by introducing new vertices for each ball of meet-irreducible diameter. The following example shows that there can be multiple types of pairs of balls of the same diameter and that therefore we will have to distinguish them explicitly in order for our expansion to have the amalgamation property:
\begin{example}[Types of pairs of blocks]\label{ex:blocktypes}
Let $S=\{1,3,4,6,7\}$. One can verify that $\oplus_S$ (see Example~\ref{example1}) is well-defined and associative. $\Semig = (S,\oplus_S,\leq)$ has three blocks: $\str 0$, $\{1\}$ and the rest. For every pair of distinct balls of diameter $1$ in an $\Semig$-metric space it holds that either all distances between them are from $\{3,4\}$ or they are from $\{6,7\}$.
\end{example}

Let $\Block$ be a non-maximal block of $\Semig$ and let $\ell\in \Semig$ be a distance. Generalising Example~\ref{ex:blocktypes}, we define $t(\Block,\ell)\subseteq \Semig$ by $a\in t(\Block, \ell)$ if and only if there is $\str A\in \MF$ and $u,u',v,v',w\in A$ such that the following conditions hold:
\begin{enumerate}
\item $u\sim_{\Block} u'$ and $v\sim_{\Block} v'$;
\item $d_\str{A}(u,v) = a$; and
\item $d_\str{A}(u',v') = \ell$.
\end{enumerate}
In other words, $t(\Block, \ell)$ is the set of all distances which can appear between balls of diameter $\Block$ if the distance $\ell$ appears between them.

\begin{example}[Running example --- $t(\Block, \Block',\ell)$]
$$t(\Block_{\{1\}},(5, 7, 0)) = \{((x,y,0) : x\in \mathbb N, y\geq 1\}.$$
Here, $y$ has to be at least 1, because otherwise we would obtain a non-metric cycle with longest edge $(5,7,0)$. On the other hand,
$$t(\Block_{\{1\}},(1, 1, 1)) = \{((x,y,z) : x\in \mathbb N, y\geq 1, z\geq 1\}.$$
\end{example}

The following easy observation says that the ``types'' $t(\Block, \ell)$ behave as expected
\begin{observation}\label{obs:blocktypes}
The following are true for every block $\Block$ and every pair of distances $\ell,\ell'\in \Semig$:
\begin{enumerate}
\item $\ell \in t(\Block, \ell)$;
\item either $t(\Block,\ell)\cap t(\Block, \ell') = \emptyset$, or $t(\Block, \ell)=t(\Block, \ell')$;
\item if $\Block'\mgeq \Block$ then $t(\Block',\ell)\supseteq t(\Block,\ell)$.
\end{enumerate}
\end{observation}

Denote by $T(\Semig)$ the set of all possible $t(\Block,\ell)$, that is,
$$T(\Semig) = \{t(\Block,\ell) : \ell\in \Semig, \Block\text{ is a block of }\Semig\}.$$

\section{Explicit representation of balls}\label{sec:lstar}
The existence of block equivalences is an obstacle for local finiteness because every cycle such that one edge says $u\nsim_\Block v$, but the rest of the cycle says $u\sim_\Block v$ is clearly non-$\Semig$-metric. In order to deal with this, we explicitly represent the balls of all meet-irreducible diameters by new vertices and link every \emph{original vertex} to its corresponding \emph{ball vertices} by unary functions, thereby eliminating the block imaginaries.

For the rest of the section, fix a \pocs{} $\ESemig$ with \textbf{finitely many blocks} where the meet of every pair of blocks is defined and different from $\str 0$ (unless, of course, one of them is $\str 0$). Note that this means that $\str 0\in I_\Semig$

Recall that we interpret an $\Semig$-metric space as a relational structure $\str{A}$ in the language $L_\Semig$ with (possibly infinitely many) binary relations $\rel{}{s}$, $s\in \Semig$. Now we are going to add explicit representatives for balls which will later make it possible to ensure that each non-$\Semig$-metric cycle has a ``non-$\Semig$-metric substructure'' of bounded size, which is necessary for Theorems~\ref{thm:hn} and~\ref{thm:herwiglascar}. This construction is by now standard in the structural Ramsey theory (cf.~\cite{Hubicka2016} or~\cite{Sam}), however, unlike Braunfeld~\cite{Sam}, for our expansion it is enough to use unary functions thanks to Lemma~\ref{lem:meetirreducible} and Observation~\ref{obs:ballintersection}.

\begin{definition}\label{defn:mstar}
Denote by $L^\star_\Semig$ the expansion of $L_\Semig$ adding
\begin{enumerate}
\item unary functions $\func{}{\Block}$ for every $\Block\in I_{\Semig}\setminus \{\str{0}\}$;
\item unary functions $\func{}{\Block,\Block'}$ for every pair of blocks $\Block, \Block'\in I_\Semig\setminus \{\str{0}\}$ such that $\Block\mlt \Block'$; and
\item $2i$-ary relations $\rel{}{i,t}$ for every $t\in T(\Semig)$ and for every $1\leq i\leq \vert I_\Semig\setminus\{\str 0\}\vert$.
\end{enumerate}
We have to add relations $\rel{}{i,t}$ of higher arities because balls of meet-reducible diameter will be represented by sets of balls of meet-irreducible diameters (cf. Lemma~\ref{lem:meetirreducible}).

For a given metric space $\str{A}\in \mathcal M_\Semig$, denote by $L^\star(\str{A})$ the $L^\star_\Semig$-structure $\str{A}^\star$ created by the following procedure:
\begin{enumerate}
\item Start with $\str A^\star$ being an exact copy of $\str A$.
\item For every $\Block\in I_\Semig\setminus \{\str{0}\}$ enumerate balls of diameter $\Block$ in $\str{A}$ as $E^1_\Block,  \ldots, E^{n_\Block}_\Block$.
\item For every $\Block\in I_\Semig\setminus \{\str{0}\}$ and $1\leq i\leq n_\Block$ add a new vertex $v^i_\Block$ to $\str{A}^\star$. We call these vertices \emph{ball vertices} (in contrast to the \emph{original vertices}).
\item For every $\Block\in I_\Semig\setminus \{\str{0}\}$,  $1\leq i\leq n_\Block$ and $v\in E^i_\Block$ put $\nbfunc{\str{A}^\star}{\Block}(v)=v^i_\Block$.
\item For every pair of blocks $\Block, \Block'\in I_\Semig\setminus \{\str{0}\}$ such that $\Block\mlt \Block'$ and every $1\leq i\leq n_\Block$ put $\nbfunc{\str{A}^\star}{\Block,\Block'}(v^i_\Block)=\nbfunc{\str{A}^\star}{\Block'}(v)$  where $v$ is some vertex of $E^i_\Block$.
\item For every non-maximal and non-$\str 0$ block $\Block$ of $\Semig$, every tuple $\Block_1,\ldots,\Block_k$ of distinct blocks from $I_\Semig\setminus \{\str{0}\}$ such that $\Block = \bigwedge_i \Block_i$ and for every pair $\bar x\neq \bar y$ of $k$-tuples of ball vertices such that the following is satisfied:
\begin{enumerate}
\item $x_i$ and $y_i$ are ball vertices for block $\Block_i$ for every $1\leq i\leq k$; and
\item there are original vertices $u,v$ such that for every $1\leq i\leq k$ it holds that $u$ lies in the ball corresponding to $x_i$ and $v$ lies in the ball corresponding to $y_i$,
\end{enumerate}
we put $(\bar x, \bar y)\in \rel{}{k,t(\Block,d_\str{A}(u,v))}$ (by Observation~\ref{obs:blocktypes} this does not depend on the particular choice of $u$ and $v$).
\end{enumerate}
Denote by $\mathcal M^\star_\Semig$ the class of all $L^\star(\str{A})$, $\str A\in \mathcal M_\Semig$ and by $\mathcal H^\star_\Semig$ the smallest hereditary superclass of $\mathcal M^\star_\Semig$ (see Remark~\ref{rem:hereditary}).
\end{definition}
Note that the assumption that $\Semig$ has finitely many blocks was necessary to ensure that $L^\star(\str A)$ has finitely many vertices if $\str A$ has. Also note that the original vertices could be understood as ball vertices for $\str 0$ (and indeed, Braunfeld does exactly this and it makes it possible to deal with lattices where $\str 0$ is meet-reducible~\cite{Sam}), but in our case they need to carry the distances and thus, for clarity, we treat them completely separately.

\begin{example}[Running example --- $L^\star_\ExSemig$]\label{ex:lstar}
\begin{figure}
\centering
\includegraphics{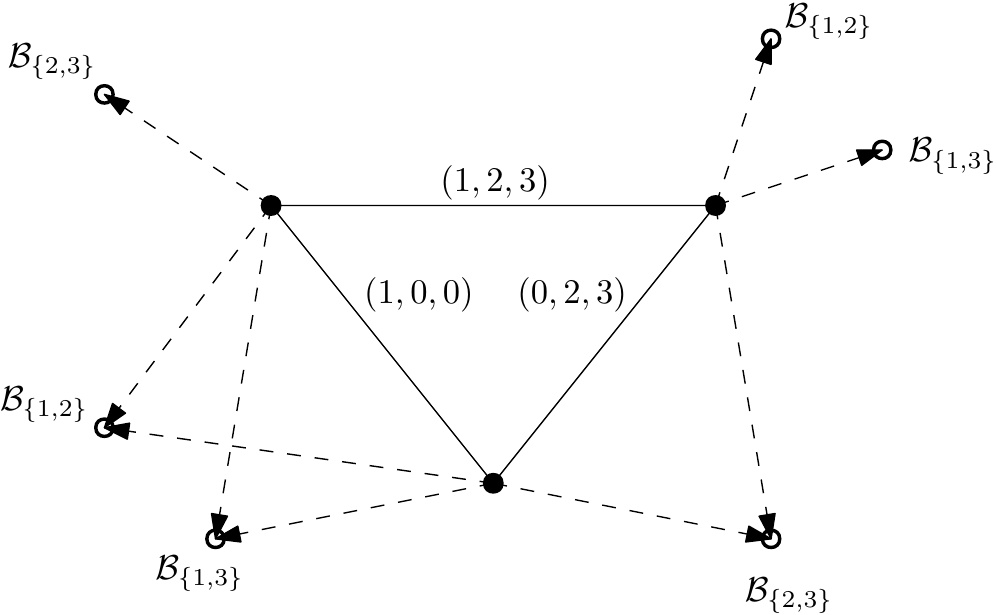}
\caption{The $L^\star_\ExSemig$-expansion, see Example~\ref{ex:lstar}}
\label{fig:ex:lstar}
\end{figure}
In Figure~\ref{fig:ex:lstar}, the $L^\star_\ExSemig$ expansion of the triangle with distances $(1,0,0)$, $(0,2,3)$ and $(1,2,3)$ is depicted. The $\func{}{\Block}$ functions are depicted by dashed arrows and their labels are moved to the ball vertices. There are no $\func{}{\Block, \Block'}$ functions in $L^\star_\ExSemig$. We did not draw the $\rel{}{i,t}$-relations.
\end{example}

Note that we will use $\Forb(\mathcal F)$ also for $L^\star_\Semig$-structures, the interpretation being that what is induced on the set of original vertices is from $\Forb(\mathcal F)$

\begin{prop}\label{prop:nonorderisomorphic}
The categories $\MF$ and $\Mstarf$ are isomorphic. In other words, $L^\star$ is a bijection between the structures from $\MF$ and $\Mstarf$ which preserves embeddings and their compositions.
\end{prop}
\begin{proof}
It is straightforward to check that $L^\star$ (with its inverse, which forgets the extra structure and the ball vertices) gives such an isomorphism.
\end{proof}

Proposition~\ref{prop:nonorderisomorphic} implies that $\classstarf$ has the amalgamation property. However, the strong amalgamation property is not a property expressible only in the categorical language. And, in fact, in order to get the strong amalgamation property for $\classstarf$, we need to ensure that the shortest path completion is consistent with the block structure (that is, the shortest path completion does not glue any balls which are not explicitly glued in the free amalgam).
\begin{definition}[Meet synchronization]\label{defn:meetsync}
Let $\ESemig$ be a \pocs{} and let $\mathcal F$ be a family of $\Semig$-edge-labelled cycles containing all disobedient ones. We say that $\mathcal F$ \emph{synchronizes meets} if for every $\str A\in\MF$, every $u,v,w,x\in A$, every finite family $\mathcal P$ of paths in $\str A$ from $u$ to $v$ and every finite family $\mathcal P'$ of paths in $\str A$ from $w$ to $x$ the following holds:

Suppose that there is a bijection $f\colon \mathcal P\rightarrow \mathcal P'$ such that $\Block(\|\str P\|) = \Block(\|f(\str P)\|)$ for every $\str P\in \mathcal P$ and every $\str P\in\mathcal P\cup \mathcal P'$ has at least two edges. Then
$$\Block(\inf(\mathcal P)) = \Block(\inf(\mathcal P')).$$
\end{definition}

If all meets of blocks are defined, meet synchronization means that whenever we encounter a set of paths in $\MF$, then the block where their infimum lies only depends on the blocks where the lengths of the paths lie, that is, $$\Block(\inf(\mathcal P)) = \bigwedge\limits_{\str P\in\mathcal P}\Block(\|\str P\|).$$

\begin{remark}
We are not aware of any situations when $\mathcal F$ containing all disobedient cycles and being confined would not also ensure that it synchronizes meets.
\end{remark}

\begin{corollary}\label{cor:lstaramalgamation}
If $\mathcal F$ is an omissible family of $\Semig$-edge-labelled cycles containing all disobedient ones which synchronizes meets, then $\mathcal M^\star_\Semig\cap \Forb(\mathcal F)$, the subclass of $\mathcal M^\star_\Semig$ which omits homomorphic images of members of $\mathcal F$, is a strong amalgamation class.
\end{corollary}
\begin{proof}
The $L^\star$ functor preserves distances and it is easy to check that since $\mathcal F$ synchronizes meets, the shortest path completion does not introduce any unnecessary $\sim_\Block$ relations if one chooses a distance from the largest block to connect the connected components of the amalgam (cf. proof of part~\ref{thm:shortestpath:amalg} of Theorem~\ref{thm:shortestpath}).
\end{proof}
\begin{remark}
The functions $\nbfunc{}{\Block,\Block'}$ were added to restrict what the substructures of a structure from $\mathcal M^\star_\Semig$ are. In particular, if a substructure contains a ball vertex for diameter $\Block$, it also needs to contain all ball vertices of larger diameters which represent the ``superballs''. This ensures that amalgamation behaves reasonably, saying that two vertices are in the same ball of diameter $\Block$, but in different balls of diameter $\Block'\mgt\Block$ has no reasonable interpretation in the metric space.
\end{remark}
\begin{remark}\label{rem:hereditary}
Unfortunately, Theorem~\ref{thm:hn} requires the class $\mathcal K$ to be hereditary. $\Hstarf$ differs from $\Mstarf$ by adding the structures which further contain some ball vertices to which no original vertex is linked. This hereditary class is in fact the reason for adding the $\rel{}{i,t}$ relations to ensure that it has the strong amalgamation property (for a \emph{confined} family $\mathcal F$), see Corollary~\ref{cor:hstarfamalg}.
\end{remark}

\subsection{Semigroups with infinitely many blocks}
We have already seen (and will see again) that one sometimes needs to assume that $\Semig$ has finitely many blocks. The following theorem says that one can, to some extent, assume it without loss of generality. In order to state it, we first need to give one more definition which will be useful later.

\begin{definition}[Confined families]\label{defn:slocallyfinite}
Let $\ESemig$ be a \pocs{}, let $\mathcal F$ be a family of $\Semig$-edge-labelled graphs. We say that $\mathcal F$ is \emph{confined} if for every finite $S\subseteq \Semig$ there are only finitely many $S$-edge-labelled graphs in $\mathcal F$.
\end{definition}

\begin{theorem}\label{thm:infinitelyblocks}
Let $\ESemig$ be a \pocs{} and $\mathcal F$ be a family of $\Semig$-edge-labelled cycles. Suppose that the following conditions hold:
\begin{enumerate}
\item $\mathcal F$ is $\Semig$-omissible;
\item $\mathcal F$ contains all $\Semig$-disobedient cycles;
\item $\mathcal F$ is confined;
\item all meets of blocks of $\Semig$ are defined and $\str 0$ is meet-irreducible; and
\item $\mathcal F$ synchronizes meets.
\end{enumerate}
Then for every \textbf{finite} $S\subseteq \Semig$ there is a countable \pocs{} $\Semig'\subseteq \Semig$ and a family of $\Semig'$-edge-labelled cycles $\mathcal F'$ such that $S\subseteq \Semig'$ and the following hold:
\begin{enumerate}
\item $\mathcal F'$ is $\Semig$-omissible;
\item $\mathcal F'$ contains all $\Semig'$-disobedient cycles;
\item $\mathcal F'$ is confined;
\item all meets of blocks of $\Semig'$ are defined and $\str 0$ is meet-irreducible;
\item $\mathcal F'$ synchronizes meets; and
\item $\Semig'$ has finitely many blocks.
\end{enumerate}
Furthermore it holds that $$\mathcal M_{\Semig}\cap\Forb(\mathcal F)\cap \mathcal G^S\subseteq\mathcal M_{\Semig'}\cap\Forb(\mathcal F')\subseteq \mathcal M_{\Semig}\cap\Forb(\mathcal F),$$
where $\mathcal G^S$ is the class of all finite $S$-edge-labelled graphs.
\end{theorem}
Informally, Theorem~\ref{thm:infinitelyblocks} says that ``in a problem about finitely many finite structures from $\mathcal M_{\Semig}\cap\Forb(\mathcal F)$ one can assume that $\Semig$ has finitely many blocks''. From $\Semig'$ being countable it furthermore follows that $\mathcal M_{\Semig'}\cap\Forb(\mathcal F')$ is a \Fraisse{} class.
\begin{proof}
Let $\mathcal A$ be the class of all $\Semig$-metric $S$-edge-labelled graphs from $\Forb(\mathcal F)$ and define $\mathcal A^+$ to be the class containing the $\Semig$-shortest path completions of graphs from $\mathcal A$ (they exist by Theorem~\ref{thm:shortestpath}). Finally let $S^+$ be the set of all distances appearing in $\mathcal A^+$.

Clearly $S^+$ is closed on $\oplus$: If $a,b\in S^+$ then there are $S$-edge-labelled graphs $\str A,\str B\in \mathcal A$ and vertices $u,v\in A$, $x,y\in B$ such that $uv$ gets completed to $a$ and $xy$ gets completed to $b$. Then the free amalgamation of $\str A$ and $\str B$ identifying only $v$ and $x$ is also in $\mathcal A$ and in its shortest path completion $uy$ gets completed to $a\oplus b$, hence $a\oplus b\in S^+$.

Denote by $\Semig'$ the \pocs{} induced by $\Semig$ on $S^+$.

Let $a,b\in \Semig'$ and suppose that $\inf(a,b)$ does not exist in $\Semig'$ or is different in $\Semig'$ than in $\Semig$. If it does not exist even in $\Semig$, then $\mathcal F$ takes care of such a pair. Suppose that it exists in $\Semig$. By definition $a$ is in $S^+$ because there is a family $\mathcal P_a$ of $S$-edge-labelled paths such that $a=\inf(\mathcal P_a)$. By the same argument we get $\mathcal P_b$ for $b$. Because $\inf(a,b)$ does not exist in $\Semig'$ or is different in $\Semig'$ than in $\Semig$ we know that for every such family $\mathcal P_a$ and every such family $\mathcal P_b$ it holds that $\mathcal P_a\cup\mathcal P_b$ (where we assume that all the paths in $\mathcal P_a\cup \mathcal P_b$ share the same two endpoints) contains either a path of length one (hence a non-$\Semig$-metric cycle) or a cycle from $\mathcal F$.

Let $\mathcal F'$ be the subset of $\mathcal F$ consisting of all $S^+$-edge-labelled cycles from $\mathcal F$. Clearly $\mathcal F'$ is $\Semig'$-omissible and it is confined. It also holds that $\mathcal F'$ contains all disobedient cycles: We checked it for undefined infima, for non-distributivity it follows from $\mathcal F$ containing all disobedient cycles and the fact that when infimum in $\Semig'$ is different than in $\Semig$, it is also $\mathcal F'$-forbidden. The same argument can also be done for meet synchronization.

This implies that $$\mathcal M_{\Semig}\cap\Forb(\mathcal F)\cap \mathcal G^S\subseteq\mathcal M_{\Semig'}\cap\Forb(\mathcal F')\subseteq \mathcal M_{\Semig}\cap\Forb(\mathcal F).$$

Finally note that $\Semig'$ is countable (each distance from $\Semig'$ corresponds to a finite $S$-edge-labelled graph) and has only finitely many blocks: They correspond to the sublattice of the block lattice of $\Semig$ generated by blocks containing elements from $S$, because $\mathcal F$ synchronizes meets.
\end{proof}

\section {Important and unimportant summands}\label{sec:nonimportant}
This part is key for obtaining a locally finite
description of $\mathcal{M}_\Semig$ needed for Theorems~\ref{thm:hn} and~\ref{thm:herwiglascar}. We show that from every non-$\Semig$-metric cycle one can select a bounded number of important edges (bounded by a function of the set of distances used in the cycle) such that the cycle stays non-$\Semig$-metric after replacing the unimportant edges by arbitrary distances from the same blocks. The $L_\Semig^\star$-expansion is useful precisely because it can express ``$d(u,v)\in\Block'$ for some $\Block'\mleq \Block$''.

Given a \pocs{} $\ESemig$ and $S\subseteq \Semig$, we denote by $S^\oplus$ the set of all elements of $\Semig$ which can be
obtained as nonempty sums of values from $S$, i.e. the subsemigroup of $\Semig$ generated by $S$.

Blocks of a semigroup may be infinite and may not contain a maximal element which
would be useful in our arguments (these maximal elements are referred to as jump numbers in~\cite{Hubicka2016}).  For a fixed finite $S\subseteq \Semig$ we seek
a sufficient approximation $\mus(\Block,S)$ of $\max(\Block)$ given by the
following lemma. Note that if $\Semig$ is finite, this section could basically consist of the statement ``Let $\mus(\Block, S) = \max(\Block)$.'' The name $\mus$ means \emph{maximum useful} distance (in $\Block$ with
respect to $S$).
\begin{lemma}\label{lem:obstaclemus}
Let $\ESemig$ be a \pocs{} and let $S\subseteq \Semig$ be a finite subset of $\Semig$. 
 Then for every block $\Block$ of $\Semig$ there is a distance $\mus(\Block, S) \in \Block$ such that for every $\ell\in S$ and $e\in S^\oplus$ one of the following holds:
\begin{enumerate}
\item $e\oplus \mus(\Block, S) \mgeq \ell$, or
\item $e\oplus b\nmgeq \ell$ for every $b \in \Block$ (and thus also for every $b\in\Block'$, where $\Block'\mleq \Block$).
\end{enumerate}
Furthermore for every block $\Block$ such that $\Block\cap S^\oplus\neq \emptyset$ we can pick $\mus(\Block,S)\in S^\oplus$.
\end{lemma}
\begin{example}[Running example --- $\mus$]
Consider the semigroup $\ExSemig$ and let $S=\{(5,5,5), (7,1,0)\}$. Then a possible choice is $\mus(\Block_{\{1,2,3\}}, S) = (7,5,5)$, $\mus(\Block_{\{1,2\}}, S) = (7,5,0)$ and $\mus(\Block_{\{3\}},S) = (0,0,5)$.
\end{example}

We will prove Lemma~\ref{lem:obstaclemus} using the following claim:
\begin{claim}\label{claim:obstaclemus:d}
Let $\ESemig$ be a partially ordered commutative semigroup, $S\subseteq \Semig$ a finite subset of $\Semig$, $\Block$ an arbitrary block of $\Semig$ and $\ell\in S$ an arbitrary distance from $S$.
There is a distance $d(\Block, S, \ell) \in \Block$ such that for every $e\in S^\oplus$ one of the following holds:
\begin{enumerate}
\item $e\oplus d(\Block, S, \ell) \mgeq \ell$, or
\item $e\oplus b\nmgeq \ell$ for every $b \in \Block$.
\end{enumerate}
\end{claim}

\begin{proof}[Proof of Lemma~\ref{lem:obstaclemus} using Claim~\ref{claim:obstaclemus:d}]
We can set $\mus(\Block, S)$ to be an element of $\Block$ greater than or equal to $d(\Block, S, \ell)$ for every $\ell\in S$. As $S$ is finite, there is such an element (for example $\bigoplus_{\ell\in S}d(\Block, S, \ell)$ for nonempty $S$). To satisfy the furthermore part it is enough to use archimedeanity.
\end{proof}

\begin{proof}[Proof of Claim~\ref{claim:obstaclemus:d}]
Suppose that the statement is not true. Then for every $d\in \Block$ there is some $e\in S^\oplus$ and $b\in \Block$ such that $e\oplus d \nmgeq \ell$, but $e\oplus b \mgeq \ell$.

By induction we now define sequences $(d_i)_{i=0}^\infty$ and $(e_i)_{i=1}^\infty$. Let $d_0\in \Block$ be an arbitrary distance from $\Block$. 

Assume that both sequences are defined up to $i-1$. Take $e\in S^\oplus$ and $b\in \Block$ such that $e\oplus d_{i-1} \nmgeq \ell$, but $e\oplus b \mgeq \ell$ and set $d_i = d_{i-1}\oplus b$ and $e_i = e$. Note that $e_i\oplus d_i\mgeq \ell$ and $d_j \mleq d_i$ for every $j\leq i$. We shall prove by contradiction that there are no indices $i < j$ with $e_i\mleq e_j$.

Suppose that there are such indices $i<j$ with $e_i\mleq e_j$. By definition we have $e_i\oplus d_i \mgeq \ell$, but $d_{j-1}\mgeq d_i$, so $e_j\oplus d_{j-1}\mgeq e_i\oplus d_i \mgeq \ell$, which is a contradiction with the definition of $e_j$.

Now observe that every distance $a\in S^\oplus$ corresponds to a function $f_a\colon S \rightarrow \mathbb N$ and vice versa by $$a = \bigfoplus_{s\in S} f_a(s)\times s.$$ Say that $f_a \trianglelefteq f_b$ if and only if $f_a(s) \leq f_b(s)$ for every $s\in S$ (i.e. $\trianglelefteq$ is the component-wise order of the vectors $\mathbb N^{|S|}$). Then clearly by monotonicity of $\oplus$ whenever $f_a\trianglelefteq f_b$, then also $a\mleq b$.

Take the sequence $(f_{e_i})_{i=1}^\infty$. By Dickson's lemma~\cite{dickson} $\trianglelefteq$ is a well-quasi-order, hence there are indices $i<j$ with $f_{e_i} \trianglelefteq f_{e_j}$. But that implies $e_i\mleq e_j$, which is a contradiction and hence finishes the proof.
\end{proof}

\medskip

The following proposition (an easy consequence of Lemma~\ref{lem:obstaclemus}) is the main result of this section which will be used in proving local finiteness.
\begin{prop}
\label{prop:Sequivalence}
Let $\ESemig$ be a \pocs{} and let $S\subseteq \Semig$ be a finite subset of $\Semig$. There exists $n=n(S)$ such that for every $\ell\in S$ and every sequence $e_1, e_2, \ldots, e_k \in S$ with $\ell \nmleq e_1\oplus e_2 \oplus \cdots \oplus e_k$ there is a sequence $f_1, f_2, \ldots, f_m\in S$ satisfying the following properties:
\begin{enumerate}
\item $(f_i)$ is a subsequence of $(e_i)$;
\item $m \leq n$; and
\item if $(f_i) \subsetneq (e_i)$, $\Block$ is the join $\bigvee\{\Block(a);a\in(e_i)\setminus (f_i)\}$ and $b\in\Block$ an arbitrary distance, then $\ell\nmleq b\oplus f_1\oplus f_2 \oplus \cdots \oplus f_m$.
\end{enumerate}
\end{prop}
We will call the distances $(f_i)$ \emph{important}.
\begin{proof}
Let $\Block_1, \Block_2, \ldots, \Block_p$ be the blocks of $\Semig$ represented in $S$ by some distance (that is, for each $\Block_i$ there is $a_i\in S$ with $a_i\in \Block_i$). As $S$ is finite, there are only finitely many such blocks. For each $s\in S$ define $n_s$ to be the smallest integer such that $n_s\times s \mgeq \mus(\Block(s), S)$. Put $$n = n(S) = \sum_{s\in S} n_s.$$

To simplify the argument, add a new element $0$ to $\Semig$, which is neutral with respect to addition.

Let $\ell, e_1, e_2, \ldots, e_k \in S$ be given with $\ell \nmleq e_1\oplus\cdots\oplus e_k$. Now we shall construct the sequence $(f_i)$ satisfying the properties from the statement. For each block $\Block_i$ create a variable $c_i$ which is initially set to $0$. Now go through all $e_i$ one by one and do the following:
\begin{enumerate}
\item Let $\Block_j$ be the block containing $e_i$.
\item If $c_j \mgeq \mus(\Block_j, S)$, go to the next $e_i$.
\item Otherwise put $e_i$ into the $(f_i)$ sequence and increment $c_j \leftarrow c_j\oplus e_i$.
\end{enumerate}
Now we check that $(f_i)$ satisfies all properties from the statement. The first two are trivial. It is enough to check the third for $b\in S^\oplus$ thanks to archimedeanity. Let $I$ be a set such that $i\in I$ if and only if there is $a\in (e_j)\setminus (f_j)$ such that $a\in \Block_i$. Clearly one can write $b$ as $b=\bigoplus_{i=1}^p b_i$ such that $b_i=0$ if $i\notin I$ and $b_i\in \Block_i\cap S^\oplus$ otherwise. Define 
\begin{align*}
b_i'&=\bigoplus_{j<i}b_j,\\
c_i &= \bigoplus_{f_j\notin \Block_i}f_j,\\
c_i' &= \bigoplus_{f_j\in \Block_i}f_j.
\end{align*}
In particular, $b_1'=0$.

By induction we will prove that $\ell\nmleq c_i\oplus c_i'\oplus b_i'$. This holds for $i=1$. Suppose now that $\ell\nmleq (c_i\oplus b_i')\oplus c_i'$. We know that $c_i\oplus b_i'\in S^\oplus$. Either $i\notin I$ and we are done, or $i\in I$, but then $c_i'\mgeq \mus(\Block_i, S)$ and thus $\ell\nmleq (c_i\oplus b_i')\oplus (c_i'\oplus b_i)$ which is what we wanted.
\end{proof}


\section{Completing structures with ball vertices}\label{sec:locallyfinitecompletion}
As we have already mentioned, the class $\mathcal M^\star_\Semig$ is not hereditary which is a problem for Theorems~\ref{thm:hn} and~\ref{thm:herwiglascar}. We defined $\mathcal H^\star_\Semig$ as the smallest hereditary superclass of $\mathcal M^\star_\Semig$ which corresponds to adding structures which contain some ball vertices to which no original vertex is linked. For our applications, we need the class $\mathcal H^\star_\Semig$ and we need it to have the strong amalgamation property. (The reader is encouraged to check Appendix~\ref{ap:classes} should they need a list of all the classes we are working with.)

One possible way to show the strong amalgamation property is to prove that free amalgams have strong completions. Coincidentally, a very similar statement needs to be proved for Theorems~\ref{thm:hn} and~\ref{thm:herwiglascar}. In order to avoid doing the same thing three times, we describe a more abstract class $\GstarS$ of (unordered) $L^\star_\Semig$-structures and prove that it has completions in $\classstarf$.

Since in free amalgams of structures from $\Hstarf$ as well as in structures which Theorems~\ref{thm:hn} and~\ref{thm:herwiglascar} ask us to complete it holds that every irreducible substructure is in $\Hstarf$, we start with a straightforward observation about such structures.

\begin{observation}\label{obs:gstar:irreducible}
Let $\str A$ be an $L^\star_\Semig$-structure such that every irreducible substructure of $\str A$ is in $\Hstarf$. This in particular means that the following hold in $\str A$:
\begin{enumerate}
\item All the distance relations $\rel{A}{a},a\in\Semig$ are symmetric and irreflexive and every pair of vertices is in at most one distance relation;
\item each vertex is either \emph{original} (that is, has all the functions $\func{A}{\Block}$ defined, none of the functions $\func{A}{\Block,\Block'}$ defined and is not in the range of any function) or it is a \emph{ball vertex} (that is, it is not in any distance relation, the functions in whose domain the vertex is are precisely all the $\func{A}{\Block, \Block'}$ for some block $\Block$ and it is only in the range of functions $\func{A}{\Block}$ and $\func{A}{\Block'', \Block}$);
\item the functions $\func{A}{\Block}$ and $\func{A}{\Block, \Block'}$ are \emph{consistent}, in particular, for every original vertex $v\in A$, and every pair of blocks $\Block,\Block'\in I_\Semig\setminus\{\str 0\}$ with $\Block\mlt \Block'$ it holds that $\func{A}{\Block'}(v) = \func{A}{\Block, \Block'}(\func{A}{\Block}(v))$;
\item\label{gstars:cond:closure} if $u,v\in A$ have distance $a$ then $\func{A}{\Block}(u) = \func{A}{\Block}(u)$ if and only if $\Block\mgeq \Block(a)$;
\item the relations $\rel{A}{i,t}$ ($t\in T(\Semig)$) are defined only on correct tuples of ball vertices, and they are consistent with each other, with the distances of the corresponding original vertices and with other ball vertices.
\end{enumerate}
\end{observation}

The relations $\rel{}{i,t}$ are difficult to visualize. When trying to complete structures satisfying conditions of Observation~\ref{obs:gstar:irreducible}, a good intuition is to imagine that for some pairs of original vertices we have a prescribed set of possible distances:

\begin{definition}[Block-distance]
Let $\str A$ be an $L^\star_\Semig$-structure such that every irreducible substructure of $\str A$ is in $\Hstarf$ and let $u\neq v$ be vertices of $\str A$. Let $\Block$ be the smallest block of $\Semig$ such that there is a set $B\subseteq I_\Semig\setminus\{\str 0\}$ with $\Block = \bigwedge B$ and for every $\Block'\in B$ it holds that $\func{A}{\Block'}(u) = \func{A}{\Block'}(v)$. ($\Block$ is well-defined by Lemma~\ref{lem:meetirreducible}; note that $B$ might be empty in which case $\Block$ is the largest block of $\Semig$.)

If there is $t\in T(\Semig)$ such that there is a tuple $(\Block_1,\ldots,\Block_k)\subseteq I_\Semig\setminus\{\str 0\}$ and two tuples of ball vertices $\bar x, \bar y$ such that the following holds:
\begin{enumerate}
\item $x_i$ and $y_i$ are both ball vertices for $\Block_i$ for every $1\leq i\leq k$;
\item $\func{\str{A}}{\Block_i}(u)=x_i$ and $\func{\str{A}}{\Block_i}(v)=y_i$ for every $1\leq i\leq k$; and
\item $(\bar x,\bar y)\in \rel{}{k,t}$,
\end{enumerate}
then we put $t'$ to be such $t$ which is inclusion-minimal (it exists by Observations~\ref{obs:blocktypes} and~\ref{obs:gstar:irreducible}), otherwise we put $t'=\Semig$.

We define the \emph{block-distance} of $u$ and $v$, denoted as $t(u,v)$ to be $t'\cap \Block$. By Observation~\ref{obs:gstar:irreducible} it holds that $t(u,v)$ is nonempty.
\end{definition}
Note that $t(u,v)$ is precisely the set of possible distances for $u$ and $v$ determined by the ball vertices.

\medskip

The whole point of the $L^\star_\Semig$ expansion is to find bounded-size witnesses for non-metric cycles. The following definition describes what non-metric cycles look like in $L^\star_\Semig$-structures.

\begin{definition}[(Non-$\Semig$-metric) $\star$-cycles]\label{defn:starcycle}
Let $\ESemig$ be a \pocs{} with finitely many blocks where all meets of non-$\str 0$ blocks are defined and non-$\str 0$ and let $\mathcal F$ be an $\Semig$-omissible family of $\Semig$-edge-labelled cycles containing all disobedient ones which synchronizes meets.

Let $\str K$ be an $L^\star_\Semig$-structure such that $\str K\in \Forb(\mathcal F)$, every irreducible substructure of $\str K$ is in $\Hstarf$ and for every ball vertex there is an original vertex pointing at it by a function. We say that $\str K$ is a \emph{strong $\star$-cycle} if its original vertices can be enumerated as $v_1,\ldots,v_k$ such that the only defined distances are between some of the pairs $v_i,v_{i+1}$ (where we identify $v_1=v_{k+1}$).

We say that an $L^\star_\Semig$-structure whose every irreducible substructure is in $\Hstarf$ is a (non-strong) $\star$-cycle if it can be obtained from a strong $\star$-cycle by removing some vertices, but only those which are in no distance and $\rel{}{i,t}$ relations.

A strong $\star$-cycle is \emph{non-$\Semig$-metric} if it is not possible to choose the missing distances $d(v_i,v_{i+1})$ from $t(v_i,v_{i+1})$ such that the resulting cycle on the original vertices is $\Semig$-metric. A $\star$-cycle is non-$\Semig$-metric if it comes from a strong non-$\Semig$-metric $\star$-cycle.
\end{definition}

Let $\ESemig$ and $\mathcal F$ be as in Definition~\ref{defn:starcycle} and let $\str A$ be an $L^\star_\Semig$-structure such that every irreducible substructure of $\str A$ is in $\Hstarf$. We say that a ball vertex of $\str A$ is an \emph{orphan} if there is no original vertex pointing to it.

\begin{lemma}\label{lem:noorphans}
Let $\ESemig$ and $\mathcal F$ be as in Definition~\ref{defn:starcycle} and let $\str A$ be an $L^\star_\Semig$-structure such that every irreducible substructure of $\str A$ is from $\Hstarf$. There is an $L^\star_\Semig$-structure $\str B$ satisfying the following:
\begin{enumerate}
\item $\str A\subseteq \str B$ and the inclusion is automorphism-preserving.
\item Every irreducible substructure of $\str B$ is from $\Hstarf$.
\item $\str B$ contains no orphans.
\item $\str A$ contains a homomorphic image of a non-$\Semig$-metric $\star$-cycle if and only if $\str B$ contains a homomorphic image of a strong non-$\Semig$-metric $\star$-cycle.
\item $\str B$ has a completion in $\Hstarf$ if and only if $\str A$ does.
\item If $\str K\subseteq \str B$ has no completion in $\Hstarf$ then the substructure induced by $\str A$ on $A\cap K$ has no completion in $\Hstarf$ either.
\end{enumerate}
\end{lemma}
\begin{proof}
Start with $\str B = \str A$ and then for each orphan which represents a ball of diameter $\Block$ extend $\str B$ by an original vertex $o$ and by ball vertices $b_{\Block'}$ for every $\Block'\in I_\Semig$ such that $\Block'\nmgeq\Block$. Put $\func{B}{\Block}(o)=b$, $\func{B}{\Block'}(o)=b_{\Block'}$ for smaller blocks and analogously define all the other functions for $o$ and $b_{\Block'}$, there is a unique way of doing that while ensuring that the closure of $o$ is in $\Hstarf$.

Clearly $\str A\subseteq \str B$ and the inclusion is automorphism-preserving, $\str B$ also contains no orphans and every irreducible substructure of $\str B$ is from $\Hstarf$. If $\str B$ contains a (homomorphic image of a) strong non-$\Semig$-metric $\star$-cycle then forgetting the newly added vertices gives a (homomorphic image of a) non-$\Semig$-metric $\star$-cycle. If $\str B$ has a completion in $\Hstarf$ then it is also a completion of $\str A$.

Next we prove that if $\str A$ contains a homomorphic image of a non-$\Semig$-metric $\star$-cycle then $\str B$ contains a homomorphic image of a strong non-$\Semig$-metric $\star$-cycle. It suffices to prove this for the case when $\str A$ is a non-$\Semig$-metric $\star$-cycle. Then we know that there is a strong non-$\Semig$-metric $\star$-cycle $\str A'$ such that $\str A\subseteq \str A'$ and no verties of $A'\setminus A$ are in any distance or $\rel{}{i,t}$ relations. Clearly, $\str B$ is a strong $\star$-cycle. For a contradiction suppose that $\str B$ has a completion in $\Hstarf$. But then by considering the $\Semig$-metric space $\str H$ on the original vertices of the completion, one can use the strong amalgamation property of $\MF$, Theorem~\ref{thm:shortestpath} and the fact that $\mathcal F$ synchronizes meets to add new vertices to $\str H$ and obtain $\str H'$ such that every ball of every diameter of $\str H'$ has many sub-balls of every smaller diameter (including $\str 0$). It then follows that $L^\star(\str H')$ is a completion of $\str A'$, a contradiction.

Now suppose that $\str A$ has a completion $\str A'$ in $\Hstarf$. By the definition of $\Hstarf$ there is $\str H \in \MF$ such that $\str A' \subseteq L^\star(\str H)$. As in the previous paragraph, using the strong amalgamation property of $\MF$, Theorem~\ref{thm:shortestpath} and the fact that $\mathcal F$ synchronizes meets, we can add vertices to $\str H$ and obtain $\str H'$ such that $\str H\subseteq \str H'$ and every ball of every diameter of $\str H'$ has many sub-balls of every smaller diameter. It follows that $L^\star(\str H')$ is a completion of $\str B$.

To prove the last claim, note that we created $\str B$ from $\str A$ by adding new vertices in order to remove orphans. Thus, in particular, all these decisions were purely local. Therefore, if we plug into this lemma the structure induced by $\str A$ on $A\cap K$ then we get precisely $\str K$ (or its super-structure) and thus the last claim follows from the previous points.
\end{proof}

Lemma~\ref{lem:noorphans} will be useful several times, because it is convenient not to have to worry about orphans. We start with a simple corollary of Lemma~\ref{lem:noorphans} which justifies calling the $\star$-cycles from Definition~\ref{defn:starcycle} (non-)$\Semig$-metric.

\begin{corollary}
Non-$\Semig$-metric $\star$-cycles have no completion in $\Hstarf$.
\end{corollary}
\begin{proof}
By definition this holds for strong non-$\Semig$-metric $\star$-cycles, the rest follows by Lemma~\ref{lem:noorphans}, because it produces strong $\star$-cycles from non-strong ones.
\end{proof}


Next we define the class $\GstarS$ of structures which we will want to complete to $\classstarf$.
\begin{definition}[The class $\GstarS$]\label{defn:gstars}
Let $\ESemig$ be a \pocs{} with finitely many blocks where all meets of non-$\str 0$ blocks are defined and non-$\str 0$ and let $\mathcal F$ be a confined $\Semig$-omissible family of $\Semig$-edge-labelled cycles containing all disobedient ones which synchronizes meets.

$\GstarS$ is defined as the subclass of all finite $L^\star_\Semig$-structures such that every $\str A\in \GstarS$ satisfies the following:
\begin{enumerate}
\item Every irreducible substructure of $\str A$ is in $\Hstarf$,
\item\label{gstars:cond:prelast} $\str A\in \Forb(\mathcal F)$; and
\item\label{gstars:cond:nonmetric} $\str A$ contains no homomorphic images of non-$\Semig$-metric $\star$-cycles.
\end{enumerate}
\end{definition}
Note that such $\str A$ contains a homomorphic image of a non-$\Semig$-metric $\star$-cycle if an only if it contains a monomorphic image of a non-$\Semig$-metric $\star$-cycle.


Now we are ready to justify the $\rel{}{i,t}$ relations in the expansion.
\begin{example}[Running example --- necessity of $\rel{}{i,t}$ relations]\label{ex:TODOLABEL}
Recall the \pocs{} $\ExSemig$ and assume that we have an analogue of $\mathcal H^\star_\ExSemig$ without the $\rel{}{i,t}$ relations. Consider the free amalgam of the edges of lengths $(1,2,3)$ and $(1,1,1)$ over a vertex such that we also glue some ball vertices so that the two non-glued vertices need to be $\sim_{\Block_{\{1\}}}$-equivalent. Such a free amalgam has no metric completion.
\end{example}

However, with the $\rel{}{i,t}$ relations such a thing never happens:
\begin{lemma}\label{lem:freeingstars}
Free amalgams of structures from $\Hstarf$ are in $\GstarS$.
\end{lemma}
\begin{proof}
Suppose that we have a free amalgam of $\str B_1$ and $\str B_2$ over $\str A$. By definition of $\Hstarf$ there are $\str G_1,\str G_2\in \MF$ such that $\str B_1\subseteq L^\star(\str G_1)$ and $\str B_2\subseteq L^\star(\str G_2)$. If we show that the free amalgam of $L^\star(\str G_1)$ and $L^\star(\str G_2)$ over $\str A$ is in $\GstarS$ then also the original amalgam is in $\GstarS$, because $\GstarS$ is hereditary. Thus we can, in particular, assume that $\str B_1$ and $\str B_2$ contain no orphans.

Clearly, free amalgams of structures from $\Hstarf$ are from $\Forb(\mathcal F)$ and all their irreducible substructures are from $\Hstarf$. It remains to prove that they contain no non-$\Semig$-metric $\star$-cycles, which amounts to technical checking that our expansion indeed does what it should. By the assumption that $\str B_1$ and $\str B_2$ contain no orphans, we can get a strong non-$\Semig$-metric $\star$-cycle from every non-$\Semig$-metric $\star$-cycle which the free amalgam potentially contains (by Lemma~\ref{lem:noorphans}). Then we can assume, for a contradiction, that this lemma is not true and take the smallest counterexample. It follows that the strong non-$\Semig$-metric $\star$-cycle has at most 4 original vertices and it is straightforward to check all the cases and arrive at a contradiction.
\end{proof}

Besides $\Hstarf$ having the strong amalgamation property, we also want the obstacles for completions to $\Hstarf$ to be of bounded size. Towards proving this (and also towards completing structures from $\GstarS$) we will use the following technical lemma:
\begin{lemma}\label{lem:addpaths}
Let $\Semig$ and $\mathcal F$ be as in Definition~\ref{defn:gstars}. Let $S$ be a finite subset of $\Semig$ and let $T$ be a finite subset of $T(\Semig)$. Put
$$T' = \{t\cap \Block : t\in T,\Block\text{ is a block of }\Semig\}\setminus \{\emptyset\}.$$
Assume that for every $\Block$ of $\Semig$ and every $t\in T'$ it holds that $S\cap \Block\neq \emptyset$ and $S\cap t\neq \emptyset$. Then for every $t\in T'$ such that $t=t(\Block,\ell)\cap \Block'$ there are distances $a(t)\in \Block\cap S^\oplus$ and $b(t)\in t\cap S^\oplus$ such that the following holds:

Let $\str A$ be a strong $\star$-cycle from $\Forb(\mathcal F)$ such that it contains only distances from $S$ and relations $\rel{}{i,t}$ for $t\in T$ and enumerate the original vertices of $\str A$ as $v_1,\ldots, v_k$ as in Definition~\ref{defn:starcycle}. Let $\str B$ be an $\Semig$-edge-labelled cycle created by starting with what $\str A$ induces on its original vertices and then connecting each non-edge $v_i,v_{i+1}$ by a path with distances $a(t(v_i,v_{i+1})), b(t(v_i,v_{i+1})), a(t(v_i,v_{i+1}))$. Then $\str B\in \Forb(\mathcal F)$ and $\str B$ is non-$\Semig$-metric if and only if $\str A$ is.
\end{lemma}

In other words, Lemma~\ref{lem:addpaths} says that given $\str A$, one can glue on every non-edge $uv$ the path with distances $a(t(u,v)), b(t(u,v)), a(t(u,v))$ and in this way get a direct witness of $t(u,v)$ in the graph (which ensures that in every possible completion, $d(u,v)$ will be from $t(u,v)$). We will use this later in this section but first we show how it implies the following observation which is in fact the reason why we introduced the $L^\star_\Semig$-expansions and the $\mus$'s:

\begin{observation}\label{obs:indeedlocallyfinite}
Let $\Semig$ and $\mathcal F$ be as in Definition~\ref{defn:gstars}, let $S$ be a finite subset of $\Semig$ and let $T$ be a finite subset of $T(\Semig)$. There exists $n=n(S,T)$ such that the following holds: Let $\str A$ be an $L^\star_\Semig$-structure such that $\str A$ only uses distances from $S$ and only uses relations $\rel{}{i,t}$ for $t\in T$, $\str A\in \Forb(\mathcal F)$ and every irreducible substructure of $\str A$ is in $\Hstarf$. If $\str A$ contains a non-$\Semig$-metric $\star$-cycle then it contains a non-$\Semig$-metric $\star$-cycle with at most $n$ vertices.
\end{observation}
\begin{proof}
Thanks to Lemma~\ref{lem:noorphans} it is enough to prove this for strong $\star$-cycles.

Extend $S$ by (finitely many) distances so that $S$ and $T$ satisfy the condition of Lemma~\ref{lem:addpaths}. Suppose that there is a strong non-$\Semig$-metric $\star$-cycle $\str K$ in $\str A$. By Lemma~\ref{lem:addpaths} we get a non-$\Semig$-metric cycle $\str B$ on a superset of its original vertices which uses only distances from $S$ and the (finitely many) distances $a(t)$ and $b(t)$. Hence, by Proposition~\ref{prop:Sequivalence}, we get a subset of the distances of the cycle of bounded size (by a function of $S$ and $T$) which corresponds to a substructure of $\str K$ in $\str A$ of bounded size which has no completion in $\Hstarf$.
\end{proof}

We now prove Lemma~\ref{lem:addpaths}.
\begin{proof}[Proof of Lemma~\ref{lem:addpaths}]
For $t = t(\Block, \ell)\cap \Block' \in T'$ define $b(t)$ to be either some element of $t\cap S^\oplus$ larger than or equal to $\mus(\Block', S)$ or the largest element of $t$ (which is necessarily in $S^\oplus$). Put $S' = S\cup \{b(t) : t\in T'\}$ and let $q$ be the maximum number of vertices of any $S'$-edge-labelled member of $\mathcal F$ (this is finite as $\mathcal F$ is confined) or $1$ if this is zero. Now we can define $a(t) = q\times\mus(\Block, S')$. Let $\str B$ be as in the statement.

It is easy to see that $\str B\in \Forb(\mathcal F)$, because $\str A$ was, and by closedness on inverse steps of shortest path completion we can exchange any edge $a(t)$ by a path of $q$ edges of length $\mus(\Block, S')$, thereby contradicting the choice of $q$ if $\str B$ contained a member of $\mathcal F$.

Assume first that $\str B$ is $\Semig$-metric. Let $\str B'$ be its $\Semig$-shortest path completion. By the choice of $a(t)$ and $b(t)$ it follows that $L^\star(\str B')$ is a completion of $\str A$.

If $\str B$ is non-$\Semig$-metric then one of two possibilities can happen. Either the longest edge is an original one. Then we can ignore the $a(t)$ distances and by the definition of $b(t)$ and by Lemma~\ref{lem:obstaclemus} we get that $\str A$ is non-$\Semig$-metric as well. Otherwise the longest edge is an added one. It cannot be any of the $a(t)$'s, because we always add two of them. So it is some $b(t)$. We show that from our choice of $a(t)$'s and $b(t)$'s it follows that $\str A$ has no completion in $\Hstarf$, hence is non-$\Semig$-metric.

Assume for a contradiction that there is a completion $\str A'$ of $\str A$ in $\Hstarf$. Let $e_1,\ldots,e_{m}$ be the distances to which $\str A'$ completed the non-edges $v_iv_{i+1}$ of $\str A$ and let $t_1,\ldots,t_{m+1}$ be the block-distances of pairs which $\str A'$ completed to $e_1,\ldots,e_{m}$ in this order.

Observe that for every $1\leq i\leq m$ such that $t_i = t(\Block, \ell)\cap \Block'$ there is a distance $a_i'\in\Block$ such that the triangle $e_i,b(t_i),a_i'$ is metric (indeed, this follows from the definition of $t(\Block, \ell)$ and the fact that $e_i,b(t_i)\in t_i$). This means that the $\Semig$-edge-labelled graph which we obtain from what $\str A'$ induces on its original vertices by adding a path with distances $a_i',b(t_i),a_i'$ on every edge $e_i$ is also $\Semig$-metric (and from $\Forb(\mathcal F)$). However, by the definition of $a(t)$ and by Lemma~\ref{lem:obstaclemus}, it follows that $\str B$ is also $\Semig$-metric which is a contradiction.
\end{proof}

We now show how to complete structures from $\GstarS$ to $\Mstarf$. For the rest of this section fix $\Semig$ and $\mathcal F$ as in Definition~\ref{defn:gstars}.

\begin{prop}
\label{prop:starfcompletion}
Let $S$ and $T$ be as in Lemma~\ref{lem:addpaths}. Let $\str A$ be a structure from $\GstarS$ using only distances from $S$ and only relations $\rel{}{i,t}$ for $t\in T$. Then there is $\str A'\in\classstarf$ which is an automorphism-preserving completion of $\str A$.
\end{prop}
\begin{proof}
By Lemma~\ref{lem:noorphans} we can assume that $\str A$ has no orphans. We will create an $\Semig$-metric space $\str G'$ such that we can then put $\str A' = L^\star(\str G')$.

Start with $\str G$ being the $\Semig$-edge-labelled graph induced by $\str A$ on the set of its original vertices. Then, for every non-edge $u,v$ of $\str A$, connect $u$ and $v$ by a path $a(t(u,v)), b(t(u,v)), a(t(u,v))$ as in Lemma~\ref{lem:addpaths}. This does not create any non-$\Semig$-metric cycles in $\str G$ nor any cycles from $\mathcal F$ and clearly preserves automorphisms. Moreover, it now holds that all information from the ball vertices of $\str A$ is now present in $\str G$. Namely, whenever $t(u,v)=t(\Block, \ell)\cap\Block'$ then there are vertices $u',v'\in G$ such that $u\sim_\Block u'$, $v\sim_\Block v'$ and $d_\str{G}(u',v') \in t(u,v)$, which implies that in every completion of $\str G$ the distance of $u$ and $v$ will be from $t(u,v)$. Note that $\str G$ is connected, because for vertices $u,v$ from different connected components of $\str A$, $t(u,v)$ is the largest block of $\Semig$.

Now $\str{G}'$ can be constructed as the shortest path completion of $\str{G}$. By part~\ref{thm:shortestpath:aut} of Theorem~\ref{thm:shortestpath} it follows that $\Aut(\str G) = \Aut(\str G')$.

Finally we put $\str A' = L^\star(\str G')$. Clearly $\str A'\in \classstarf$. We need to prove that $\str A'$ is a completion of $\str A$ and that it is automorphism-preserving. We know that $\str G'$ is precisely the $\Semig$-metric space induced by $\str A'$ on the original vertices and that it is an automorphism-preserving completion of $\str G$. Thus we only need to prove that the ball vertices and block-distances in $\str A'$ and $\str A$ correspond to each other, which follows from the fact that we added witnessing paths for all of those.
\end{proof}
\begin{corollary}\label{cor:hstarfamalg}
The class $\Hstarf$ has the strong amalgamation property (see Definition~\ref{defn:mstar}).
\end{corollary}
\begin{proof}
By Lemma~\ref{lem:freeingstars}, free amalgams of structures from $\Hstarf$ are in $\GstarS$, strong amalgamation now follows by Proposition~\ref{prop:starfcompletion}.
\end{proof}

\begin{corollary}\label{cor:locallyfinite}
The class $\Hstarf$ is a locally finite automorphism-preserving subclass of the class of all finite $L^\star_\Semig$-structures.
\end{corollary}
\begin{proof}
We verify the axioms (see Definition~\ref{def:locallyfinite}). Let $S$ be the set of all distances occurring in $\str C_0$ and let $T$ be a finite subset of $T(\Semig)$ such that the all nonempty $\rel{C_0}{i,t}$ have $t\in T$. If necessary, expand $S$ by finitely many distances to satisfy the conditions of Lemma~\ref{lem:addpaths}. Let $q$ be the maximum number of vertices of an $S$-edge-labelled graph from $\mathcal F$ and put $n=\max(q,n(S,T))$, where $n(S,T)$ is given by Observation~\ref{obs:indeedlocallyfinite}. When given such $\str C$ as in Definition~\ref{def:locallyfinite} it follows that it belongs to $\GstarS$ and has an automorphism-preserving completion in $\Hstarf$ (in fact, even in $\classstarf$).
\end{proof}

\subsection{Henson constraints}
Note that every nonempty $t(\Block,\ell)\cap\Block'$ contains a distance which is reducible with respect to $\mathcal F$ (see Section~\ref{sec:henson}), hence we can pick all $a(t)$'s and $b(t)$'s in Lemma~\ref{lem:addpaths} to be reducible with respect to $\mathcal F$. This means that in this whole section, we can further restrict all considered classes to omit a family $\mathcal H$ of Henson constraints (cf. Observation~\ref{obs:hensoncompletion}). In particular, we get the following corollary.
\begin{corollary}\label{cor:hensonstar}
The class $\Hstarf\cap\Forb(\mathcal H)$ is a locally finite automor\-phism-preserving subclass of the class of all finite $L^\star_\Semig$-structures and has the strong amalgamation property. Moreover, every member of $\Hstarf\cap\Forb(\mathcal H)$ has an automorphism-preserving completion in $\classstarf\cap\Forb(\mathcal H)$.
\end{corollary}

\section{EPPA}\label{sec:eppaproof}

Now we are ready to prove Theorem~\ref{thm:eppa}.
\begin{proof}[Proof of Theorem~\ref{thm:eppa}]
Using Theorem~\ref{thm:infinitelyblocks} we can assume that $\Semig$ has finitely many blocks. We want to apply Theorem~\ref{thm:herwiglascar} for $\mathcal E$ being the class of all $L^\star_\Semig$-structures (which has EPPA by Theorem~\ref{thm:eppanr}) and for $\mathcal K = \Hstarf$. However, $\Hstarf$ does not consist of irreducible structures only. To fix this, we can assume that $\Hstarf$ contains a binary relation $E$ such that $E$ is a complete graph on the set of all vertices of every structure. It is easy to verify that all the results of the previous section still hold with this $E$ relation. Now we get, using Corollary~\ref{cor:locallyfinite}, that $\Hstarf$ has EPPA. Moreover, since $\classstarf\subseteq \Hstarf$ and every structure in $\Hstarf$ has an automorphism-preserving completion in $\classstarf$ (by Proposition~\ref{prop:starfcompletion}), we get that $\classstarf$ has EPPA. Given $\str A\in \MF$, we know that there is $L^\star(\str B)\in \classstarf$ which is an EPPA-witness for $L^\star(\str A)$. It follows that all the more so (by looking at what is induced on original vertices) $\str B\in \MF$ is an EPPA-witness for $\str A$ which is what we wanted.
\end{proof}

\begin{remark}
Theorem~\ref{thm:herwiglascar} in fact promises \emph{coherent EPPA}, which is a strengthening of EPPA by Siniora and Solecki~\cite{Siniora, Siniora2} with stronger group-theo\-retical consequences. We did not want to define it in Section~\ref{sec:eppa}, but it follows that in Theorem~\ref{thm:eppa} we could in fact promise coherent EPPA.
\end{remark}
\begin{remark}
Using Corollary~\ref{cor:hensonstar} we can prove EPPA also for the class $\MF\cap\Forb(\mathcal H)$ where $\mathcal H$ is a family of Henson constraints.
\end{remark}
\chapter{Convex order and the Ramsey property}\label{ch:withorder}
In this chapter we prove Theorem~\ref{thm:main}. We adapt Braunfeld's definition of convex ordering of $\Lambda$-ultrametric spaces for all semigroup-valued metric spaces. Then we show how to transfer this order between the $L_\Semig$-structures and $L^\star_\Semig$-structures, which enables us to use the machinery of Chapter~\ref{ch:orderless}.

Unless stated otherwise, in the whole chapter we fix a \pocs{} $\ESemig$ with \textbf{finitely many blocks} where the meet of every non-$\str 0$ pair of blocks is defined and non-$\str 0$.

\section{Convex ordering}\label{sec:convord}
To obtain a Ramsey class we need to define a notion of ordering for classes $\mathcal M_\Semig$. The convex ordering of a metric space was first used by Nguyen Van Th\'e~\cite{NVT2009,The2010}. In the case of ultrametric spaces it is possible to order vertices in such a way that every ball is a linear interval. Braunfeld~\cite{Sam} generalised the concept of convex ordering to $\Lambda$-ultrametric spaces where this is no longer possible. We proceed analogously.

\begin{definition}
\label{defn:Js}
Let $L^+_\Semig$ be the expansion of the language $L_\Semig$ which adds a binary relation $\leq^\Block$ for each $\Block \in I_\Semig$.

Given an $\Semig$-metric space $\str{A}=(A,(\rel{A}{s})_{s\in \Semig})$, its {\em convexly ordered expansion} is an \mbox{$L^+_\Semig$} expansion of $\str{A}$ such that for every $\Block\in I_\Semig$ the relation $\leq^\Block$ is a partial order satisfying the following
\begin{enumerate}
\item One of $u \leq^\Block v$ and $v \leq^\Block u$ is defined if and only if $u\not\sim_\Block v$ and for every $\Block'\mgt \Block$ it holds that $u\sim_{\Block'} v$; and
\item for every $w \in A$ such that $u\sim_\Block w$ we have $w\leq^\Block v$ if and only if $u\leq^\Block v$.
\end{enumerate}
We will denote by $\overrightarrow{\mathcal M}_\Semig$ the class of all convexly ordered $\Semig$-metric spaces.
\end{definition}

As we shall see, the orders $\leq^\Block$ are a concise description of orders of balls of diameter $\Block$. In particular, if $\Semig$ is archimedean then we only added $\leq^{\str 0}$ --- a linear order of vertices of $\Semig$.

\begin{example}[Running example --- $\overrightarrow{\mathcal M}_\ExSemig$]
In $\ExSemig$ for $\Block_{\{1,2\}}$, there is only one larger block, namely the largest one. Hence $u\leq^{\Block_{\{1,2\}}} v$ or $v\leq^{\Block_{\{1,2\}}} u$ is defined if and only if $u\nsim_{\Block_{\{1,2\}}} v$, similarly for $\Block_{\{1,3\}}$ and $\Block_{\{2,3\}}$. On the other hand, $\Block_\emptyset$ is the smallest block, hence $u\leq^{\Block_\emptyset} v$ or $v\leq^{\Block_\emptyset} u$ if and only $d(u,v) = (0,0,0)$.
\end{example}

In order to prove the strong amalgamation property for the convexly ordered metric spaces, we need to strengthen our assumptions of $\Semig$ and $\mathcal F$. In particular, we need the block lattice to be distributive and we need for every pair $\Block$, $\Block'$ of blocks of $\Semig$ and every $a\in \Block$, $b\in \Block'$ which ``we encounter in our metric spaces'' that $\Block(\inf(a,b)) = \Block\meet\Block'$ (see Definition~\ref{defn:meetsync}).

A lattice is \emph{distributive} if it satisfies the following two equations
\begin{align*}
a\wedge(b\vee c) &= (a\wedge b)\vee(a\wedge c)\\
a\vee(b\wedge c) &= (a\vee b)\wedge(a\vee c).
\end{align*}
Note that it is enough to check whether a lattice satisfies one of the equations, the other can then be derived syntactically.

Braunfeld~\cite{Sam} proved that, for a distributive lattice $\Lambda=(\Lambda, \vee, \wedge, 0)$ where $0$ is meet-irreducible, the class of convexly ordered $\Lambda$-ultrametric spaces has the strong amalgamation property (see Remark~\ref{rem:sam}). For completeness, we include the proof adapted to the setting of semigroup-valued metric spaces. Note that each block of $\Lambda$ other than $\str 0$ contains precisely one distance from $\Lambda$ and the block lattice without $\str 0$ is isomorphic to $\Lambda$.
\begin{theorem}[Braunfeld~\cite{Sam}]\label{thm:samstrongamalg}
Let $\Lambda$ be a distributive lattice. The class $\overrightarrow{\mathcal M}_\Lambda$ has the strong amalgamation property.
\end{theorem}
\begin{proof}
By Theorem~\ref{thm:shortestpath} we know that the unordered reduct $\mathcal M_\Lambda$ has the strong amalgamation property. Hence it suffices to be able to complete the partial orders. The key step of the proof is noticing that if $\str C_0$ is the free amalgam of $\str B_1$ and $\str B_2$ over $\str A$ then for a fixed block $\Block\in I_\Lambda$ the relation $\leq^\Block_{\str C_0}$ can be extended to a linear order, that is, there are no vertices $u\in B_1\setminus A$, $v\in B_2\setminus A$ and $w\in A$ with, say, $u\leq^\Block_{\str C_0} w \leq^\Block_{\str C_0} v$ such that in the $\Lambda$-shortest path completion of $\str C_0$ we have $u\sim_\Block v$. For a contradiction, assume the existence of such $u,v,w$.

$\Block$ being meet-irreducible and $u\sim_\Block v$ in the shortest path completion of $\str C_0$ means that there are vertices $w_1,\ldots,\allowbreak w_k\in A$ such that $$\Block\mgeq\bigwedge\{\Block(d_{\str C_0}(u, w_i))\vee \Block(d_{\str C_0}(w_i, v));1\leq i\leq k\}.$$

Because $\Lambda$ (and thus also the block lattice which is isomorphic to $\Lambda$) is distributive, there is $i$ such that $$\Block\mgeq \Block(d_{\str C_0}(u, w_i))\vee \Block(d_{\str C_0}(w_i, v))$$ (indeed, if $\Xi$ is a distributive lattice, $a\in \Xi$ is meet-irreducible and there are $b_1,\ldots,b_k\in \Xi$ such that $a\geq \bigwedge b_i$, this means that $a = a\vee\bigwedge b_i$, or $a = \bigwedge (a\vee b_i)$ and from meet-irreducibility it follows that there is $i$ with $a\geq a\vee b_i$) and this means $u\sim_\Block w_i$ in $\str B_1$ and $v\sim_\Block w_i$ in $\str B_2$. But then $u\leq_{\str B_1}^\Block w$ if and only if $w_i\leq_{\str B_1}^\Block w$ and $v\leq_{\str B_2}^\Block w$ if and only if $w_i\leq_{\str B_2}^\Block w$. Both $w_i$ and $w$ are in $A$, thus $u\leq_{\str C_0}^\Block w$ if and only if $v\leq_{\str C_0}^\Block w$, which is a contradiction.
\end{proof}

We intend to reduce the strong amalgamation property for general convexly ordered $\Semig$-metric spaces (satisfying some conditions) to the strong amalgamation property for convexly ordered $\Lambda$-ultrametric spaces. Again, we know that by Theorem~\ref{thm:shortestpath} the classes $\mathcal M_\Semig\cap \Forb(\mathcal F)$ have the strong amalgamation property and thus it is sufficient to be able to complete the orders. And they only depend on the block structure, hence it is enough to look at the \emph{reducts} of the structures where distances are replaced by their blocks. Note that these are not reducts in the sense of Definition~\ref{defn:expansion}, they however correspond to the more general model-theoretical definition of a reduct.
\begin{corollary}\label{cor:convamalgamation}
Assume that the block order of $\Semig$ is a distributive lattice and let $\mathcal F$ be an omissible family of $\Semig$-edge labelled cycles containing all disobedient ones which \textbf{synchronizes meets}. Then $\overrightarrow{\mathcal M}_\Semig\cap \Forb(\mathcal F)$, the subclass of $\overrightarrow{\mathcal M}_\Semig$ omitting homomorphic images of members of $\mathcal F$, is a strong amalgamation class.
\end{corollary}
\begin{proof}
Let $\str C_0$ be the free amalgam of $\str B_1$ and $\str B_2$ over $\str A$ where $\str A, \str B_1, \str B_2\in \overrightarrow{\mathcal M}_\Semig\cap \Forb(\mathcal F)$. We will show that $\str C_0$ has a strong completion $\str C'\in\overrightarrow{\mathcal M}_\Semig\cap \Forb(\mathcal F)$ which is then the strong amalgam of $\str B_1$ and $\str B_2$ over $\str A$.

By Theorem~\ref{thm:shortestpath} we know that if we forget the orders, $\str C_0$ has a strong $\Semig$-metric completion. Thus it remains to show that one can complete the orders.

Consider the reduct $\str C_0^-$ of $\str C_0$ where we replace every distance $a$ with $\Block(a)$. From the definition of the block order and from the assumption that $\mathcal F$ synchronizes meets we get that $\str C_0^-$ is a free amalgam of convexly ordered $\Lambda$-ultrametric spaces, that is, semigroup-valued metric spaces where the semigroup is $\Lambda = (\{\Block(a); a\in \Semig\}, \vee, \wedge)$. Furthermore the shortest path completion of the distances in $\str C_0^-$ is the same as the reduct of the shortest path completion of the distances in $\str C_0$.

By Theorem~\ref{thm:samstrongamalg} we know that the class of convexly ordered $\Lambda$-ultra\-metric spaces has the strong amalgamation property and this gives us the completion of the partial orders which we needed.
\end{proof}
\begin{remark}
We essentially just used the fact that the \emph{shortest path completion functor} and the \emph{reduct to $\Lambda$-ultrametric spaces functor} commute. In order for them to commute, one needs $\mathcal F$ to synchronize meets.
\end{remark}

\section{Ordering the ball vertices}
If $(X, \leq_X)$ and $(Y, \leq_Y)$ are two ordered sets then the \emph{lexicographic order} $\leq_\mathrm{lex}$ on $X\times Y$ is given by $(x,y)\leq_\mathrm{lex} (x',y')$ if and only if either $x<_X x'$ or $x=x'$ and $y\leq_Yy'$. This naturally generalizes to the product of several orders.

For every $\Block \in I_\Semig$, define $\Block^+$ to be the unique block such that for every $\Block'\mgeq \Block$ it holds that $\Block\mlt \Block^+ \mleq \Block'$ (its existence and uniqueness follow from the fact that there are only finitely many blocks and that all meets are defined).

We are going to show that from the partial orders $\leq^\Block$, one can define linear orders of balls of every diameter. There are many possible ways to do it and we need to pick one. There is nothing special about the particular choices in the following definition.
\begin{definition}\label{defn:u}
Assume an arbitrary but fixed choice of a linear order $\trianglelefteq$ of all blocks of $\Semig$ such that $\Block_i\trianglelefteq \Block_j$ whenever $\Block_i\mgeq \Block_j$. Then define the function $U\colon R_\Semig\cup I_\Semig\to (R_\Semig\cup I_\Semig)^2$ as follows:
\begin{enumerate}
\item If $\Block \in R_\Semig$ then $U(\Block) = (\Block_1, \Block_2)$, where $(\Block_1, \Block_2)$ is the $\trianglelefteq$-lexicographically smallest pair of blocks such that $\Block_1, \Block_2\neq \Block$ and $\Block = \Block_1\meet \Block_2$.
\item If $\Block \in I_\Semig$ then $U(\Block) = (\Block,\Block)$.
\end{enumerate}
\end{definition}

\begin{definition}\label{defn:definable}
Let $U$ be the function from Definition~\ref{defn:u}. Given $\str{A}\in \overrightarrow{\mathcal M}_\Semig$, define inductively for every block $\Block$ of $\Semig$ the relation $\ll^\Block$ on balls of diameter $\Block$ in $\str A$.

If $\Block$ is the largest block of $\Semig$, then there is only one ball of diameter $\Block$ and $\ll^\Block$ is trivial. Otherwise $B_1\ll^\Block B_2$ if and only if one of the following holds:
\begin{enumerate}
 \item\label{lem:definable:reducible} $\Block \in R_\Semig$, $U(\Block)=(\Block_1,\Block_2)$ and the unique pair $B^1_1,B^2_1$ of blocks of diameter $\Block_1$ and $\Block_2$ respectively containing $B_1$
is lexicographically (in the orders $\ll^{\Block_1}$ and $\ll^{\Block_2}$) smaller than the unique pair $B^1_2,B^2_2$ of blocks of diameter $\Block_1$ and $\Block_2$ respectively containing $B_2$.
 \item\label{lem:definable:inball} $\Block \in I_\Semig$ and there exist $u\in B_1$, $v\in B_2$ such that $u\leq^\Block v$.
 \item\label{lem:definable:outball} $\Block \in I_\Semig$, $B'_1$ and $B'_2$ are the unique balls of diameter $\Block^+$ containing $B_1$ and $B_2$ respectively, and $B'_1\ll^{\Block^+} B'_2$.
\end{enumerate}
\end{definition}

\begin{lemma}\label{lem:definable}
All relations $\ll^{\Block}$ given by Definition~\ref{defn:definable} are linear orders.

\end{lemma}
\begin{proof}
Enumerate blocks of $\Semig$ as $\Block_1, \Block_2,\ldots, \Block_n$ in non-increasing order of $\mleq$ (i.e. there is no $\Block_i\mlt \Block_j$ for $1\leq i\leq j\leq n$). Given $\str{A}\in \overrightarrow{\mathcal M}_\Semig$, we verify by induction that for every block $\Block_i$ the relation $\ll^{\Block_i}$ is indeed a linear order of balls of diameter $\Block_i$. As $\Block_1$ is the largest block, the statement is trivial for $\Block_1$.

Now the induction step. If $\Block_k\in R_\Semig$, $U(\Block)=(\Block(a),\Block(b))$ then by the induction hypothesis we know that $\ll^{\Block(a)}$ and $\ll^{\Block(b)}$ are linear orders of balls of diameter $\Block(a)$ and $\Block(b)$. Because every ball of diameter $\Block_k$ lies in a unique intersection of $\Block(a)$ and $\Block(b)$, these two orders uniquely define lexicographically the linear order $\ll^{\Block_k}$.

Otherwise $\Block_k\in I_\Semig$. Now for every pair $B_1$, $B_2$ of distinct balls of diameter $\Block_k$ we have that $\leq^{\Block_k}$ is defined 
for an arbitrary choice of $u\in B_1$, $v\in B_2$ if and only if both blocks $B_1$ and $B_2$ belong to the same ball of diameter $\Block^+_k$. In this case we use rule~\ref{lem:definable:inball}, otherwise the order is defined by rule~\ref{lem:definable:outball}.
\end{proof}

Note that, as a special case of Lemma~\ref{lem:definable}, we get that $\ll^\str{0}$ is a definable linear order on vertices of $\str{A}$.

Now we need to transfer the convex order to the $L^\star$-expansions, which will make it possible to apply Theorem~\ref{thm:hn}.
\begin{definition}\label{defn:mstarleq}
Denote by $L^{\star,\leq}_\Semig$ the expansion of $L^\star_\Semig$ adding the order $\leq$.

For a given convexly ordered metric space $\str{A}\in \overrightarrow{\mathcal M}_\Semig$, denote by $L^{\star,\leq}(\str{A})$ the $L^{\star,\leq}_\Semig$-structure $\str{A}^{\star,\leq}$ created by the following procedure:
\begin{enumerate}
\item Start with $\str{A}^{\star,\leq}=L^\star(\str A^-)$ given by Definition~\ref{defn:mstar}, where $\str A^-$ is the $L_\Semig$-reduct of $\str A$ (that is, we forget the orders).
\item Define the linear order $\leq_{\str{A}^{\star,\leq}}$  as follows:
\begin{enumerate}
 \item Order vertices of $\str{A}$ according to $\ll^{\str{0}}$ and let them form an initial segment of $\leq_{\str{A}^{\star,\leq}}$.
 \item For every pair of ball vertices $v^i_\Block$ and $v^j_{\Block'}$ (see Definition~\ref{defn:mstar}) put $v^i_\Block\leq_{\str{A}^{\star,\leq}} v^j_{\Block'}$ if and only if $\Block\triangleleft \Block'$ or $\Block = \Block'$ and $E^i_{\Block}\ll^{\Block} E^j_{\Block}$ (see Definitions~\ref{defn:u} and~\ref{defn:definable}).
\end{enumerate}
\end{enumerate}
Denote by $\Mstarleq$ the class of all $L^{\star,\leq}(\str{A})$ for $\str A\in \overrightarrow{\mathcal M}_\Semig$ and by $\Hstarleq$ the smallest hereditary superclass of $\Mstarleq$.
\end{definition}

\begin{lemma}\label{lem:bidefinable}
The categories $\classstarleqf$ and $\overrightarrow{\mathcal M}_\Semig\cap \Forb(\mathcal F)$ are isomorphic. In other words, there is a bijection between $\classstarleqf$ and $\overrightarrow{\mathcal M}_\Semig\cap \Forb(\mathcal F)$ which preserves embeddings and their compositions.
\end{lemma}
\begin{proof}
Given $\str{A}^{\star,\leq}\in \classstarleqf$ we want to construct $\str{A}\in \overrightarrow{\mathcal M}_\Semig\cap \Forb(\mathcal F)$ such that $L^{\star,\leq}(\str A) = \str{A}^{\star,\leq}$. For this it is enough to reconstruct the partial orders $\leq^\Block$. This can be done by putting $u\leq^\Block v$ if and only if $u\nsim_\Block v$, $\func{A}{\Block}(u)\leq_\str{A} \func{A}{\Block}(v)$ and $u\sim_{\Block'} v$ for every $\Block'\mgeq \Block$.

It is then straightforward to check that indeed this is the inverse of $L^{\star,\leq}$ and that they give an isomorphism of categories.
\end{proof}
\begin{corollary}\label{cor:starleqamalgamation}
Let $\ESemig$ be a \pocs{} with finitely many blocks such that the block order is a distributive lattice and let $\mathcal F$ be an omissible family of $\Semig$-edge-labelled cycles containing all disobedient ones which synchronizes meets. Then $\classstarleqf$ is a strong amalgamation class.
\end{corollary}
\begin{proof}
By the category isomorphism we immediately get the amalgamation property. Strong amalgamation property follows from meet synchronization as in Corollary~\ref{cor:lstaramalgamation}.
\end{proof}
\begin{remark}
Again, the class $\classstarleqf$ is not hereditary. One gets $\Hstarleqf$ by adding structures containing some ball vertices which are not linked to any original vertices. Later we shall observe that this still is a strong amalgamation class. Note that, thanks to having a linear order, $\Hstarleqf$ consists of irreducible structures, which is another requirement of Theorem~\ref{thm:hn}.
\end{remark}

\section{Completing the order}\label{sec:completingorder}
For the Ramsey property we need order. This is why we introduced the $\classstarleqf$ and $\Hstarleqf$ classes (see Definition~\ref{defn:mstarleq}). In this section we show that the structures from $\GstarS$ with an additional relation $\leq$ which is ``tame enough'' have a completion in $\classstarleqf$. This will imply both the strong amalgamation property and the Ramsey property of $\Hstarleqf$ and also the Ramsey property of $\classstarleqf$, which is what we aim for. But first we observe that if $\mathcal F$ is confined while containing all disobedient cycles (and synchronizing meets), it has some consequences for the block lattice.

\begin{lemma}\label{lem:meets}
Let $\ESemig$ be a \pocs{} and $\mathcal F$ be an $\Semig$-omissible, \textbf{confined} family of $\Semig$-labelled cycles containing all disobedient ones.
\begin{enumerate}
\item\label{lem:meets:nonsync} Suppose that $\Semig$ has finitely many blocks. Then for every two blocks $\Block_1$ and $\Block_2$ of $\Semig$ their meet $\Block_1\meet \Block_2$ is defined, that is, the order of blocks is a lattice. Furthermore $\str 0$ is meet-irreducible.

\item\label{lem:meets:sync} Suppose that $\mathcal F$ synchronizes meets. Then the order of blocks is a distributive lattice. Furthermore $\str 0$ is meet-irreducible.
\end{enumerate}
\end{lemma}
On the first sight, it might seem surprising that having a family $\mathcal F$ can have some implications for the block structure of $\Semig$. But on the second sight, the existence of a \textbf{confined} family $\mathcal F$ with the desired properties is quite a strong condition.
\begin{proof}
If $\Block_1 \mleq \Block_2$ or vice-versa, the smaller of them is their meet. So now suppose that $\Block_1$ and $\Block_2$ are not comparable and thus different from $\str 0$.

Take arbitrary $a\in \Block_1$ and $b\in \Block_2$ and define the sequence $(c_n)_{n=1}^\infty$ by putting $c_n = \inf(n\times a, n\times b)$.

Observe that there are only finitely many $c_n$'s undefined: If $\inf(n\times a, n\times b)$ is undefined for $n\geq 2$, then the cycle $\str C_n = (a, \ldots, a, b, \ldots, b)$ with $n$ edges of length $a$ and $n$ edges of length $b$ is disobedient and hence $\str C_n\in \mathcal F$. But as $\mathcal F$ is confined, there can only be finitely many such cycles in $\mathcal F$.

Now note that whenever $m\leq n$ and both $c_m$ and $c_n$ are defined, then $c_m\mleq c_n$ and thus $\Block(c_m) \mleq \Block(c_n)$.
 
Put
$$\Block = \lim_{n\rightarrow \infty} \Block(c_n),$$
where we ignore all the $n$'s with $c_n$ undefined. Note that if this limit is defined, then it certainly is different from $\str 0$. Thus if we prove that $\Block$ is the meet of $\Block_1$ and $\Block_2$, we in particular get that $\str 0$ is not the meet of any two non-$\str 0$ blocks and hence is meet-irreducible.

\medskip
For part~\ref{lem:meets:nonsync} note that the sequence of blocks $(\Block_{c_n})$ is non-decreasing and $\Semig$ has only finitely many blocks. Thus the limit is well-defined.

For part~\ref{lem:meets:sync} we may have infinitely many blocks, but we now have the assumption that $\mathcal F$ synchronizes meets. This means that whenever $\Block(c_n) \neq \Block(c_m)$, then one of the cycles $\str C_n$, $\str C_m$ needs to be in $\mathcal F$ (indeed, if this wasn't so, then their disjoint union would have a completion in $\mathcal M_\Semig\cap \Forb(\mathcal F)$ by Theorem~\ref{thm:shortestpath} and this would contradict Definition~\ref{defn:meetsync}). As $\mathcal F$ is confined, there can only be finitely many such forbidden $\str C_n$'s, hence the limit is again well-defined.
\medskip

Blocks are archimedean, hence for every $a'\in \Block_1$ there are $m$ and $n$ with $m\times a'\mgeq a$ and $n\times a\mgeq a'$, the same holds also for $\Block_2$. This implies that $\Block$ doesn't depend on the choice of $a$ and $b$.

We shall prove that $\Block$ is the meet of $\Block_1$ and $\Block_2$. The fact that $\Block\mleq \Block_1, \Block_2$ follows straight from the definition of $\Block$. Take some $x\in \Semig$ such that there are $a\in \Block_1$ and $b\in \Block_2$ such that $a,b\mgeq x$. Then also $n\times a, n\times b\mgeq x$ for every $n$. Let $n$ be an arbitrary integer such that $\inf(n\times a, n\times b)$ is defined (by confinedness of $\mathcal F$ there is such $n$). Clearly $\inf(n\times a, n\times b)\mgeq x$, but also there is $c\in \Block$ such that $c\mgeq \inf(n\times a, n\times b)$, hence $c\mgeq x$.
\medskip

To prove distributivity for part~\ref{lem:meets:sync} it suffices to check $$\Block_1\vee(\Block_2\meet\Block_3) = (\Block_1\vee\Block_2)\meet(\Block_1\vee\Block_3).$$ Take arbitrary $a\in \Block_1, b\in \Block_2, c\in \Block_3$ and let $m$ be larger than the number of vertices of any cycle in $\mathcal F$ which contains only distances from $\{a,b,c\}$ ($\mathcal F$ is confined and thus there is such a finite $m$). From archimedeanity we get $a'=m\times a\in \Block_1$, $b'=m\times b\in \Block_2$ and $c'=m\times c\in \Block_3$. First note that $$a'\oplus\inf(b', c') = \inf(a'\oplus b', a'\oplus c'),$$ because $\mathcal F$ contains all disobedient cycles. Clearly $$\Block\left(a'\oplus b'\right) = \Block_1\vee\Block_2$$ and $$\Block\left(a'\oplus c'\right) = \Block_1\vee \Block_3.$$ But $\mathcal F$ also synchronizes meets, this means that $$\Block\left(\inf(b', c')\right) = \Block_2\meet\Block_3$$ and $$\Block\left(\inf(a'\oplus b', a'\oplus c')\right) = (\Block_1\vee\Block_2)\meet (\Block_1\vee \Block_3).$$ Finally $$\Block\left(a'\oplus\inf(b', c')\right) = \Block_1\vee(\Block_2\meet\Block_3),$$ hence indeed $$\Block_1\vee(\Block_2\meet\Block_3) = (\Block_1\vee\Block_2)\meet(\Block_1\vee\Block_3).$$
\end{proof}

We want to prove the following proposition.
\begin{prop}
\label{prop:finite4value2}
Let $\ESemig$ be a \pocs{} and $\mathcal F$ be a family of $\Semig$-labelled cycles. Suppose that the following conditions hold:

\begin{enumerate}
\item $\mathcal F$ is $\Semig$-omissible;
\item $\mathcal F$ contains all $\Semig$-disobedient cycles;
\item $\mathcal F$ is confined;
\item $\mathcal F$ synchronizes meets; and
\item $\Semig$ has finitely many blocks.
\end{enumerate}

Then the class $\classstarleqf$ of all finite $L^{\star,\leq}_\Semig$-expansions of $\Semig$-metric spaces omitting $\mathcal F$ has the Ramsey property.
\end{prop}
In order to do that, we first need some auxiliary lemmas. For the rest of this section fix $\Semig$ and $\mathcal F$ as in Proposition~\ref{prop:finite4value2}.

The following lemma is, similarly to Definition~\ref{defn:gstars} and the subsequent lemmas, stated in a more abstract setting in order to avoid having to repeat the proof twice.
\begin{lemma}\label{lem:ordercompletion}
Let $\str C_0$ be an $L^{\star,\leq}_\Semig$-structure such that every irreducible substructure of $\str C_0$ is from $\Hstarleqf$, its $L^\star_\Semig$ reduct is from $\GstarS$ and there is a linear extension $\leq_0$ of $\leq_{\str C_0}$. Then $\str C_0$ has a completion in $\classstarleqf$.
\end{lemma}
\begin{proof}
By Lemma~\ref{lem:meets} we know that meets of all blocks are defined in $\Semig$-metric spaces which do not contains cycles from $\mathcal F$, that the block lattice is distributive and that $\str 0$ is meet-irreducible.

By Proposition~\ref{prop:starfcompletion} we know that $\str C_0$ without the order has a completion in $\classstarf$. Suppose that $\str C$ is this completed structure (that is, all functions and distances and relations $\rel{}{i,t}$ are defined and consistent and each ball vertex is pointed to by an original vertex). It remains to complete the order. We can assume that $\leq_0$ is a linear order of vertices of $\str C$ (because the vertices added to $\str C$ are not in any $\leq$ relation, hence one can add them to the order arbitrarily).

Let $\trianglelefteq$ be the order of blocks of $\Semig$ from Definition~\ref{defn:u} and enumerate $I_\Semig\setminus\{\str 0\}$ as $\Block_1\trianglerighteq \Block_2\trianglerighteq \cdots\trianglerighteq \Block_p$ (recall that this is a $\mleq$-non-decreasing order of blocks).
Now define $\leq_{\str C}$ on vertices of $\str C$ as follows:
\begin{enumerate}
 \item For every pair $u,v$ of original vertices of $\str{C}$ put $u\leq_{\str C} v$ if the sequence of vertices $(\func{C}{\Block_i}(u))_i$ is in the order $\leq_0$ lexicographically before $(\func{C}{\Block_i}(v))_i$ or they are equal and $u\leq_0 v$.
 \item For every pair $u,v$ of ball vertices such that $u$ corresponds to block $\Block_i$ and $v$ to block $\Block_j$ put $u\leq_{\str{C}} v$ if one of the following holds:
\begin{enumerate}
\item $i<j$,
\item $i=j$ and the sequence $(\func{C}{\Block_i,\Block_{i'}}(u))_{1\leq i'\leq i, \Block_i\mlt\Block_{i'}}$ is in the order $\leq_0$ lexicographically before $(\func{C}{\Block_i,\Block_{i'}}(v))_{1\leq i'\leq i, \Block_i\mlt\Block_{i'}}$,
\item $i=j$, $(\func{C}{\Block_i,\Block_{i'}}(u))_{1\leq i'\leq i, \Block_i\mlt\Block_{i'}} = (\func{C}{\Block_i,\Block_{i'}}(v))_{1\leq i'\leq i, \Block_i\mlt\Block_{i'}}$ and $u\leq_0 v$.
\end{enumerate}
 \item Finally put $u\leq_{\str C} v$ if $u$ is an original vertex and $v$ is a ball vertex.
\end{enumerate}
$\leq_{\str C}$ is clearly a linear order and by comparing the construction with Definition~\ref{defn:mstar} it can be verified that $\str C$ (with the order) is in $\classstarleqf$ and that $\leq_{\str{C}_0}\subseteq\leq_{\str C}$ (because the convex ordering coincides with the lexicographic ordering according of balls). Hence $\str C$ is a completion of $\str C_0$ in $\classstarleqf$.
\end{proof}

\begin{corollary}\label{cor:hstarleqfamalg}
The class $\Hstarleqf$ has the strong amalgamation property.
\end{corollary}
\begin{proof}
Cf. Corollary~\ref{cor:hstarfamalg}; it is enough to check that every free amalgam of structures in $\Hstarleqf$ satisfies the conditions of Lemma~\ref{lem:ordercompletion}.
\end{proof}

Now we are ready to prove Proposition~\ref{prop:finite4value2}:
\begin{proof}[Proof of Proposition~\ref{prop:finite4value2}]
Using Lemma~\ref{lem:ordercompletion} we know that $\Hstarleqf$ is a locally finite subclass of $\mathcal R_\Semig$, the class of all $L^{\star,\leq}_\Semig$-structures where $\leq$ is a linear order, which is Ramsey by Theorem~\ref{thm:NRclosures}. We also know that $\Hstarleqf$ is a strong amalgamation class (Corollary~\ref{cor:hstarleqfamalg}) and it is hereditary, therefore by Theorem~\ref{thm:hn} we know that $\Hstarleqf$ is Ramsey.

We, however, need to show that $\classstarleqf$ is Ramsey. Because $\classstarleqf\subseteq \Hstarleqf$, it follows that for every $\str A, \str B\in \classstarleqf$ there is $\str C\in \Hstarleqf$ such that $\str C\longrightarrow (\str B)^{\str A}_2$. By Lemma~\ref{lem:ordercompletion} there is $\str C'\in \classstarleqf$ with $\str C\subseteq \str C'$. And thus $\str C'\longrightarrow (\str B)^{\str A}_2$ hence $\classstarleqf$ indeed is Ramsey.
\end{proof}

\section{Proof of Theorem~\ref{thm:main} and the expansion property}
Now we are ready to prove the main theorem of this thesis.
\begin{proof}[Proof of Theorem~\ref{thm:main}]
Theorem~\ref{thm:shortestpath} gives us the strong amalgamation property. Proposition~\ref{prop:finite4value2} shows that $\classstarleqf$ is Ramsey and thus by Lemma~\ref{lem:bidefinable} we then get that $\overrightarrow{\mathcal M}_\Semig\cap \Forb(\mathcal F)$ is Ramsey when $\Semig$ has finitely many blocks.

Suppose now that $\Semig$ has infinitely many blocks. Take $\str A\subseteq\str B\in \overrightarrow{\mathcal M}_\Semig\cap \Forb(\mathcal F)$. We want to find $\str C\in\overrightarrow{\mathcal M}_\Semig \cap \Forb\left(\mathcal F\right)$ such that $\str C\longrightarrow (\str B)^{\str A}_2$. Let $S$ be the (finite) set of distances in $\str B$. By Theorem~\ref{thm:infinitelyblocks} we get $\Semig'$ and $\mathcal F'$ such that $\mathcal M_{\Semig'}\cap \Forb(\mathcal F')\subseteq \mathcal M_{\Semig}\cap \Forb(\mathcal F)$. Therefore $\overrightarrow{\mathcal M}_{\Semig'}\cap \Forb(\mathcal F')$ has the Ramsey property by the previous paragraphs.

Clearly, in $\str A$ and $\str B$ the only relations $\leq_\Block$ which are not empty are those where $\Block$ is also a block of $\Semig'$. We can thus consider $\str A$ and $\str B$ to be from $\overrightarrow{\mathcal M}_{\Semig'}\cap \Forb(\mathcal F')$ and from the Ramsey property we get $\str C\in\overrightarrow{\mathcal M}_{\Semig'}\cap \Forb(\mathcal F')$ which is the Ramsey witness for $\str A$ and $\str B$. Then it is enough to again add empty orders to $\str C$ to get a structure from $\overrightarrow{\mathcal M}_\Semig\cap \Forb(\mathcal F)$.
\end{proof}

\begin{theorem}\label{thm:expansionprop}
Let $\Semig$ and $\mathcal F$ be as in Theorem~\ref{thm:main}. Then $\overrightarrow{\mathcal M}_{\Semig}\cap \Forb(\mathcal F)$ has the expansion property with respect to $\mathcal M_{\Semig}\cap \Forb(\mathcal F)$.
\end{theorem}

The proof is a generalization of Braunfeld's proof~\cite[Lemma 7.10]{Sam} which is in turn an adaptation of the standard argument. However, it is quite technically challenging and thus, as an introduction, we first sketch the proof of the case when $\Semig$ is archimedean. Then there are no ball vertices, no unary functions and the expansion is just all linear orders. Note that $\mus(\Semig, S)$ is a well-defined distance for a finite set $S\subseteq \Semig$.

Given $\str A\in \mathcal M_\Semig \cap \Forb(\mathcal F)$ we need to find $\str B\in \mathcal M_\Semig \cap \Forb(\mathcal F)$ such that every ordering of $\str B$ contains every ordering of $\str A$. Denote by $S$ the set of distances in $\str A$. Let $\overrightarrow{\str A}\in \overrightarrow{\mathcal M}_{\Semig}\cap \Forb(\mathcal F)$ be a fixed ordering of $\str A$. Enumerate all pairs of vertices $u<v\in A$ as $(u_1,v_1),\ldots,(u_k,v_k)$ and define a sequence $\overrightarrow{\str A} = \overrightarrow{\str A}^0\subseteq\overrightarrow{\str A}^1\subseteq\ldots\subseteq\overrightarrow{\str A}^k$ of structures from $\overrightarrow{\mathcal M}_{\Semig}\cap \Forb(\mathcal F)$ by induction such that $\overrightarrow{\str A}^{i+1}$ is the strong amalgam of $\overrightarrow{\str A}^i$ and the ordered triangle $\str T_i$ with vertices $u_i,v_i,x_i$ where $d_\str{T}(u_i,v_i)=d_{\overrightarrow{\str{A}}}(u_i,v_i)$, $d_\str{T}(u_i,x_i)=d_\str{T}(v_i,x_i)=\mus(\Semig, S)$ and $u_i\leq_\str{T} x_i\leq_\str{T} v_i$ over the edge $u_i,v_i$.

Now let $\overrightarrow{\str D}$ be the joint embedding of the $\overrightarrow{\str A}^k$'s for all linear orderings $\overrightarrow{\str A}$ of $\str A$. By the Ramsey property there exists $\overrightarrow{\str B}\longrightarrow (\overrightarrow{\str D})^{\overrightarrow{\str H}}_2$, where $\overrightarrow{\str H}$ is the ordered pair of vertices in distance $\mus(\Semig, S)$. Let $\leq_0$ be the linear order on vertices of $\overrightarrow{\str B}$ and finally let $\str B$ be the reduct of $\overrightarrow{\str B}$ by forgetting the order.

We claim that $\str B$ is the desired expansion property witness for $\str A$. Indeed, let $\leq_1$ be an arbitrary linear ordering of $\str B$. Define the colouring $c\colon{\overrightarrow{\str B}\choose \overrightarrow{\str H}}\rightarrow \{0,1\}$ by putting $c(\alpha) = 0$ if the orders $\leq_0$ and $\leq_1$ agree on $\alpha(\overrightarrow{\str H})$ and $1$ otherwise.

By the Ramsey property we get a copy $\overrightarrow{\str D}\subseteq \overrightarrow{\str B}$ such that $\leq_0$ and $\leq_1$ are on the copy either the same or opposite. In both cases $\overrightarrow{\str D}$ contains all possible orderings of $\str A$.

\begin{proof}[Proof of Theorem~\ref{thm:expansionprop}]
The complication now is that there are multiple (partial) orders which depend on the distances. There is no hope to find a universal distance (like $\mus(\Semig, S)$ in the archimedean case) which would allow us to place auxiliary vertices between every pair of vertices where we want to fix the order. Instead, we are going to do that separately for every meet-irreducible block (and thus every $\leq^\Block$). This will hence need an iterated use of the Ramsey property.

Let $\str A\in \mathcal M_\Semig \cap \Forb(\mathcal F)$ be given and let $S$ be the set of distances in $\str A$. Enumerate the non-maximal meet-irreducible blocks of $\Semig$ (including $\str 0$) which ``non-trivially appear in $\str A$'' as $\Block_1,\ldots,\Block_b$, that is, $\Block \in I_\Semig$ is in the sequence if and only if $\leq^\Block_{\overrightarrow{\str A}}$ is nonempty in some convex ordering $\overrightarrow{\str A}$ of $\str A$. Also define a sequence of distances $m_1,\ldots,m_b$ where $m_i = \mus(\Block_i, S\cup\{m_j;j< i\})$. By induction we define $\str A = \str B_0\subseteq \str B_1\subseteq\cdots\subseteq \str B_b=\str B$ as follows.

Let $\str B_i$ be defined. We now do a very similar construction as in the archime\-de\-an case with $\str B_i$ playing the role of $\str A$ and $\str B_{i+1}$ being the expansion property witness. Fix an expansion $\overrightarrow{\str B}_i$ of $\str B_i$ in $\overrightarrow{\mathcal M}_{\Semig}\cap \Forb(\mathcal F)$. Enumerate the balls of diameter $\Block_{i+1}$ in $\overrightarrow{\str B}_i$ as $E_1, \ldots, E_\ell$ and pick a representative $w_j\in E_j$ for each ball $E_j$ arbitrarily. Finally enumerate by $(u_1, v_1), \ldots, (u_k, v_k)$ all pairs of the representatives such that $u_j <^{\Block_{i+1}}_{\overrightarrow{\str B}_i} v_j$.

Now define a sequence $\overrightarrow{\str B}_i = \overrightarrow{\str B}_i^0\subseteq\overrightarrow{\str B}_i^1\subseteq\ldots\subseteq\overrightarrow{\str B}_i^k$ by induction such that $\overrightarrow{\str B}_i^{j+1}$ is the strong amalgam of $\overrightarrow{\str B}_i^{j}$ and the ordered triangle $\str T$ over the edge $u_j,v_j$, where $\str T$ is the triangle with vertices $u_j,v_j,x_j$ with $u_j<^{\Block_{i+1}}_{\str T}x_j<^{\Block_{i+1}}_{\str T}v_j$, $d_\str{T}(u_j,v_j)=d_{\overrightarrow{\str{B}}_i}(u_j,v_j)$ and $d_\str{T}(u_j,x_j)=d_\str{T}(v_j,x_j)=m_{i+1}$.

Put $\overrightarrow{\str D}_{i+1}$ be the joint embedding of the $\overrightarrow{\str B}_i^k$'s for all expansions $\overrightarrow{\str B}_i$ of $\str B_i$ in $\overrightarrow{\mathcal M}_{\Semig}\cap \Forb(\mathcal F)$. By the Ramsey property there exists $\overrightarrow{\str B}_{i+1}\longrightarrow (\overrightarrow{\str D}_{i+1})^{\overrightarrow{\str H}_{i+1}}_2$, where $\overrightarrow{\str H}_{i+1}$ is the ordered pair of vertices in distance $m_{i+1}$. Let $\leq^{\Block_{i+1}}_0$ be the partial order for $\Block_{i+1}$ of $\overrightarrow{\str B}_{i+1}$ and finally let $\str B_{i+1}$ be the reduct of $\overrightarrow{\str B}_{i+1}$ forgetting all the orders.

By an analogous argument as in the archimedean case (after remembering that if $u\sim^{\Block} v$, then $u\leq^{\Block}w$ if and only if $v\leq^{\Block}w$) we get that in every expansion of $\str B_{i+1}$ in $\overrightarrow{\mathcal M}_{\Semig}\cap \Forb(\mathcal F)$ there are all possible orders $\leq^{\Block_{i+1}}$ of $\str B_i$.

It follows that in every expansion of $\str B=\str B_b$ there are all possible expansions of $\str A$ as desired.
\end{proof}

\begin{remark}
Similarly as for EPPA and the strong amalgamation property, we can further restrict our classes to omit a family of Henson constraints and have the same theorems (in particular the same Ramsey expansions) for them.
\end{remark}

\chapter{Applications}\label{ch:applications}
It is trivial to check that the (partially) ordered commutative semigroups connected to the $S$-metric spaces with $\oplus_S$ defined or to $\Lambda$-ultrametric spaces satisfy the axioms of Theorem~\ref{thm:main}. Similarly, we also get Ramsey expansions and EPPA for the multiplicative metric spaces (which has not been done before). Another original corollary of our results is EPPA for Sauer's $S$-metric spaces, Braunfeld's $\Lambda$-ultrametric spaces (our methods would work even for the ones with meet-reducible identity) and non-semi-archimedean Conant's generalised metric spaces.

Corollary-wise, it remains to show how metrically homogeneous graphs fit into our framework. This is more complicated, in fact, they were in fact the motivation for introducing the general conditions on the family $\mathcal F$.

\section{Primitive metrically homogeneous graphs}\label{sec:methom}
A metrically homogeneous graph is a countable connected graph which gives rise to a homogeneous metric space when one computes the distances between all vertices. Cherlin~\cite{Cherlin2011b,Cherlin2013} recently gave a catalogue of the known homogeneous metric spaces of this type. A key role in the catalogue is played by the so-called \emph{primitive 3-constrained cases}. They can be described as $\{1,\ldots,\delta\}$-edge-labelled graphs with some triangles forbidden. And these triangles can be described in terms of five parameters $(\delta, K_1, K_2, C_0, C_1)$.

Ramsey expansions and other combinatorial properties of the classes from Cherlin's catalogue have been found by Aranda, Bradley-Williams, Hubi{\v c}ka, Karamanlis, Kompatscher, Pawliuk and the author~\cite{Aranda2017, Aranda2017a, Aranda2017c} and were also studied by others~\cite{Coulson, Sokic2017}. In~\cite{Aranda2017} the key ingredient is an explicit completion procedure quite similar to the shortest path completion.

To interpret every primitive 3-constrained case as a semigroup-valued metric space would take a lot of mechanical work and inequality checking. In this thesis, we only sketch the proof for part of the primitive 3-constrained cases (namely we omit the cases when $|C_0-C_1|>1$) and therefore only give the minimum necessary introduction and definitions. For a broader overview see~\cite{Aranda2017} or the author's Bachelor thesis~\cite{Konecny2018bc}.

\begin{definition}
We say that a sequence of integers $(\delta, K_1, K_2, C)$ is \emph{relevant} if the following conditions hold:
\begin{itemize}
\item $3\leq \delta < \infty$;
\item $1\leq K_1\leq K_2\leq \delta$;
\item $2\delta+2\leq C\leq 3\delta+2$;
\item one of the following holds:
\begin{enumerate}[label=(\Roman*)]
  \setcounter{enumi}{1}

\setlength\itemsep{0em}
\item\label{II} $C\leq 2\delta+K_1$, and:
\begin{itemize}
 \setlength\itemsep{0em}
 \item $C=2K_1+2K_2+1$;
 \item $K_1+K_2\geq \delta$;
 \item $K_1+2K_2\leq 2\delta-1$,
\end{itemize}
\item\label{III} $C\geq 2\delta+K_1+1$, and:
\begin{itemize}
 \setlength\itemsep{0em}
 \item $K_1+2K_2\geq 2\delta-1$ and $3K_2\geq 2\delta$;
 \item If $K_1+2K_2=2\delta-1$ then $C\geq 2\delta+K_1+2$.
\end{itemize}
\end{enumerate}
\end{itemize}
\end{definition}
This definition is a subset of the union of two of Cherlin's definitions (\emph{acceptability} and \emph{admissibility}) and is tailored so that it only contains the classes interesting for our purposes, which is also why Case~I is missing.

Hubi\v cka, Kompatscher and the author~\cite{Hubickacycles2018} characterized all the $\{1,\ldots,\allowbreak\delta\}$-edge-labelled graphs which have a completion into one of Cherlin's classes with relevant parameters as $\Forb(\Fclass)$ for $\Fclass$ being a family of the following cycles, where we say that a cycle has distances $a_1,\ldots a_k$ if it has $k$ edges labelled by $a_1,\ldots,a_k$ in some order. A perimeter of a cycle is the sum of its distances.
\begin{description}
\item[$C$-cycles:] Cycles with distances $d_0, d_1, \ldots, d_{2n}, x_1, \ldots, x_k$ for some $n\geq 0$ such that $$\sum_{i=0}^{2n} d_i > n(C - 1) + \sum_{i=1}^k x_i.$$
\item[$K_1$-cycles:] Non-$C$-cycles of odd perimeter with distances $x_1, \ldots, x_k$ such that $$2K_1 > \sum_{i=1}^k x_i.$$
\item[$K_2$-cycles:] Non-$C$-cycles of odd perimeter with distances $d_1,\ldots,d_{2n+2},\allowbreak x_1, \ldots,\allowbreak x_k$ such that $$\sum_{i=1}^{2n+2}d_i > 2K_2 + n(C-1) + \sum_{i=1}^k x_i.$$
\end{description}
Notice that a non-metric cycle is a $C$-cycle with $n=0$. Cherlin's class $\Aclass$ can then be defined as the subclass of $\Forb(\Fclass)$ containing only complete graphs. It is then already determined by only forbidding the triangles from $\Fclass$. Cherlin proved that all the $\Aclass$'s are strong amalgamation classes.

In~\cite{Aranda2017} we proved that the classes $\Aclass$ are locally finite subclasses of the class of all $\{1,\ldots,\delta\}$-edge-labelled complete graphs using a special \emph{magic completion algorithm} for which the following definition was key (it was stated in a different way).
\begin{definition}[Magic semigroup]\label{defn:magicsemigroup}
Let $3\leq \delta<\infty$, $2\delta+2\leq C\leq 3\delta+1$ and $\lceil\frac{\delta}{2}\rceil\leq M \leq \frac{C-\delta-1}{2}$ be integers. Then the operation $\oplus\colon\{1,\ldots,\delta\}^2\rightarrow \{1,\ldots,\delta\}$ is defined as follows
$$x \oplus y =
  \begin{cases}
    |x-y| & \text{if } |x-y| > M \\
    \min\left(x+y, C-1-x-y\right) & \text{if } \min\left(\ldots\right) < M \\
    M & \text{otherwise}.\end{cases}$$
\end{definition}
The magic completion is a refinement of the shortest path completion algorithm which step-by-step adds edges according to the $\oplus$ operation. Observe that in particular $M\oplus x = M$ for every $x$.

It is straightforward to check that $\oplus$ is associative and commutative, hence $(\{1,\ldots,\delta\},\oplus)$ is indeed a commutative semigroup for every valid choice of $\delta$, $C$ and $M$.

Fix $\delta$, $C$ and $M$ from Definition~\ref{defn:magicsemigroup}. The \emph{natural partial order} $\mleq$ on $\{1,\ldots,\delta\}$ is defined by $a\mleq b$ if and only if $a=b$ or there is $1\leq c\leq \delta$ such that $b=a\oplus c$. It is again straightforward to check that $\mleq$ is a partial order and that the \emph{(ordered) magic semigroup} $\magicsemig=(\{1,\ldots,\delta\},\oplus, \mleq)$ is a \pocs{}. Note that if $C= 3\delta+1$ and $M=\delta$ then $\magicsemig$ is just $\{1,\ldots,\delta\}$ with the standard order and addition capped by $\delta$.

\begin{figure}
\centering
\includegraphics{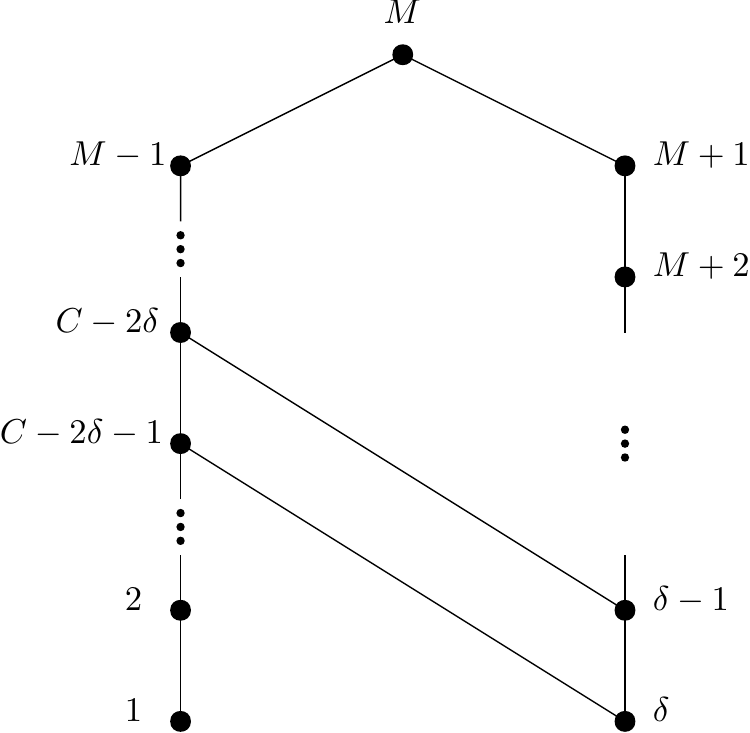}
\caption{The Hasse diagram of the magic order $\mleq$.}
\label{fig:magicorder}
\end{figure}

\begin{observation}\label{obs:incomparable}
In $\magicsemig$ we have $a\mleq b$ if and only if one of the following holds (see Figure~\ref{fig:magicorder}):
\begin{enumerate}
\item $a\leq b\leq M$;
\item $a\geq b\geq M$; or
\item $a\geq M$ and $C-1-\delta-a\leq b\leq M$.
\end{enumerate}
In other words, $a$ is $\mleq$-incomparable with $b$ (without loss of generality we can assume $a\geq b$) if and only if $a>M$, $b<M$ and $C-1-\delta-a>b$.
\end{observation}

Clearly $\magicsemig$ is archimedean. In order to plug it into the machinery developed here, we need to understand the $\mleq$-infima, their distributivity with $\oplus$ and what the non-$\magicsemig$-metric cycles are.

\begin{prop}\label{prop:magiccanonicalform}
Let $x_1,\ldots,x_k$ and $d_1,\ldots,d_n$ be two sequences of integers such that $1\leq x_i < M$ and $M<d_i\leq \delta$ for every $i$. Denote $S=x_1\oplus\cdots\oplus x_k\oplus d_1\oplus\cdots\oplus d_n$. Then one of the following happens:
\begin{enumerate}
\item $S=M$;
\item $S<M$, $n$ is even and $$S=\frac{n}{2}(C-1) + \sum_{i=1}^k x_i - \sum_{i=1}^n d_n;$$
\item $S>M$, $n$ is odd and $$S=\sum_{i=1}^n d_n-\frac{n-1}{2}(C-1) - \sum_{i=1}^k x_i.$$
\end{enumerate}
\end{prop}
\begin{proof}
By induction. Clearly holds for $k+n\leq 2$, and the induction step is straightforward.
\end{proof}

\begin{corollary}
Let $1\leq a,b\leq \delta$ be $\mleq$-incomparable and let $\str K$ be a cycle with distances $a_1,\ldots,a_k,b_1,\ldots,b_\ell$ such that $k+\ell\geq 3$, $a=a_1\oplus\cdots\oplus a_k$ and $b=b_1\oplus\cdots\oplus b_\ell$. Then $\str K$ is a $C$-cycle.
\end{corollary}
\begin{proof}
We can without loss of generality (cf. Observation~\ref{obs:incomparable}) assume $a<M$ and $b>M$.

By Proposition~\ref{prop:magiccanonicalform} we get $a=\frac{n_a}{2}(C-1)+\sum_{i=1}^{m_a} x^a_i-\sum_{i=1}^{n_a} d^a_i$ and $b=\sum_{i=1}^{n_b} d^b_i-\frac{n_b-1}{2}(C-1) - \sum_{i=1}^{m_b} x^b_i$ for $\{d^a_i;1\leq i\leq n_a\}\cup \{x^a_i;1\leq i\leq m_a\} = \{a_i;1\leq i\leq k\}$ with the union being disjoint, analogously for $b$. As $b>M>a$, we get 
$$\sum_{i=1}^{n_b} d^b_i-\frac{n_b-1}{2}(C-1) - \sum_{i=1}^{m_b} x^b_i>\frac{n_a}{2}(C-1)+\sum_{i=1}^{m_a} x^a_i-\sum_{i=1}^{n_a} d^a_i,$$
or
$$\sum_{i=1}^{n_b} d^b_i + \sum_{i=1}^{n_a} d^a_i > \frac{n_a+n_b-1}{2}(C-1) + \sum_{i=1}^{m_b} x^b_i + \sum_{i=1}^{m_a} x^a_i,$$
which is just the $C$-inequality.
\end{proof}
This corollary implies that the family of $C$-cycles contains all disobedient ones: Indeed, whenever there are two paths between two vertices in a $\magicsemig$-metric space omitting $C$-cycles then their $\magicsemig$-lengths have to be comparable, therefore there are no non-trivial infima. And trivial infima (that is, minima) distribute with $\oplus$.

\begin{lemma}\label{lem:magicfminus}
The family $\mathcal F^-$ of all $\magicsemig$-metric $C$-cycles is $\magicsemig$-omissible.
\end{lemma}
\begin{proof}
We need to check that $\mathcal F^-$ contains no geodesic cycles, that it is downwards closed and that it is closed under inverse steps of the shortest path completion. Closedness under the inverse steps is straightforward: There are no non-trivial infima involved, so any edge gets replaced by a path whose $\magicsemig$-length is the length of the edge and it is easy to check that this preserves the $C$-inequality and $\magicsemig$-metricity.

From Proposition~\ref{prop:magiccanonicalform} it also follows that $\mathcal F^-$ contains no geodesic cycles (if it did then we use the correspondence from Proposition~\ref{prop:magiccanonicalform} to get a contradiction with the $C$-inequality being strict).

It remains to check downwards closedness. Take any contiguous segment of edges of a cycle $\str K\in \mathcal F^-$ and enumerate the lengths of its edges as $e_1,\ldots,e_n,\allowbreak y_1,\ldots,y_k$ where $e_i>M$ and $y_i<M$ (there is no edge of length $M$ in $\str K$ because otherwise it would be non-$\magicsemig$-metric or geodesic). The parity of $n$ determines whether $\bigoplus e_i\oplus\bigoplus y_i$ is greater than or smaller than $M$ and based on this and Proposition~\ref{prop:magiccanonicalform} we get the desired result.
\end{proof}

Again by an analysis of different reasons for $a\mleq b$ it follows that every non-$\magicsemig$-metric cycle is a $C$-cycle. This in particular implies that all $K_1$- and $K_2$-cycles are $\magicsemig$-metric. It then again takes some checking (in a similar fashion as before) that if further $K_1\leq M\leq K_2$, the $K_1$- and $K_2$-cycles are not geodesic and that their union is closed both downwards and on the inverse steps of the shortest path completion. Here one needs to use the fact that the parameters are relevant. But in the end we get the following:

\begin{theorem}\label{thm:magicissemig}
Let $(\delta, K_1, K_2, C)$ be relevant parameters and define $\mathcal F$ as the union of $\mathcal F^-$ from Lemma~\ref{lem:magicfminus} and all $K_1$- and $K_2$-cycles. Then $\mathcal F$ is $\magicsemig$-omissible, contains all $\magicsemig$-disobedient cycles and is finite. Therefore $\magicsemig$ and $\mathcal F$ satisfy the assumptions of Theorem~\ref{thm:main}.

Furthermore, the class $\magicsemig\cap\Forb(\mathcal F)$ is precisely $\Aclass$.
\end{theorem}

In order to get an analogue of Theorem~\ref{thm:magicissemig} for all primitive 3-constrained metrically homogeneous graphs one only needs to do more rather mechanical inequality checking (cf.~\cite{Hubickacycles2018}).

It is worth stressing out that Theorem~\ref{thm:magicissemig} depends heavily on previous results~\cite{Aranda2017, Hubickacycles2018} and the author does not see a way to bypass this which would be simpler then redoing all the proofs in a different language.

\section{Stationary independence relations}\label{sec:sir}
In 2012 Tent and Ziegler~\cite{Tent2013} proved that the automorphism group of the Urysohn space is simple modulo the \emph{bounded automorphisms} (that is, automorphisms $\alpha$ such that there is a distance $a=a(\alpha)$ such that for every vertex $x$ it holds that $d(x, \alpha(x))\leq a$). In the paper, they defined the \emph{(local) stationary independence relation} (or \emph{SIR}), which is a ternary relation on finite subsets of a homogeneous structure satisfying some axioms.

Let $\mathbb F$ be a homogeneous structure in a relational language and let $A,B\subseteq F$ be its finite subsets.
We will identify them with the substructures induced by $\mathbb F$ on $A$ and $B$ respectively and by $AB$ we will denote the union $A\cup B$ (and hence also the substructure induced by $\mathbb F$ on $AB$). If the set $A=\{a\}$ is singleton, we may write $a$ instead of $\{a\}$.

\begin{definition}[Stationary Independence Relation]
\label{defn:sir}
Let $\mathbb{F}$ be a homogeneous structure in a relational language.  A ternary relation $\ind$ on finite subsets of $\mathbb{F}$ is called a \emph{stationary independence
relation} (\emph{SIR}) if the following conditions are satisfied:
\begin{enumerate}[label=SIR\arabic*]
 \item\label{invariance} \emph{(Invariance)}. The independence of finite subsets of $\mathbb{F}$ only depends on their type. In particular, for 
any
automorphism $f$ of $\mathbb{F}$, we have $A\ind_{C}B$ if and only if
$f(A)\ind_{f(C)}f(B)$.
 \item\label{symmetry}\emph{(Symmetry)}. If $A\ind_C B$ then $B\ind_C A$.
\item\label{monotonicity} \emph{(Monotonicity)}. If $A\ind_{C}BD$ then
$A\ind_{C}B$ and $A\ind_{BC}D$.
\item\emph{(Existence)}. \label{existence} For any $A,B$ and $C$ in $\mathbb{F}$, there
is some
$A'\models \tp(A/C)$ with $A'\ind_{C}B$. 
\item\emph{(Stationarity)}. \label{stationarity} If $A$ and $A'$ have the same
type over $C$ and are both independent over $C$ from some set $B$ then they
also have the same type over $BC$. 
\end{enumerate}
If the relation $A\ind_{C}B$ is only defined for nonempty $\str{C}$, we
call $\ind$ a \emph{local} stationary independence relation.
\end{definition}
Here, $tp(A/C)$ is the type of $A$ over $C$, which for relational homogeneous structures (with trivial algebraic closures) simply amounts to the isomorphism type of $AC$ with a given enumeration of vertices. $A'\models\tp(A/C)$ thus means that $AC\simeq A'C$ with the given enumerations of vertices.

Stationary independence relations correspond to ``canonical amalgamations'' by putting $A\ind_C B$ if and only if the canonical amalgamation of $AC$ and $BC$ over $C$ is isomorphic to $ABC$. The notion of canonical amalgamations can be formalised, see~\cite{Aranda2017}.

The following theorem is a direct consequence of Theorem~\ref{thm:shortestpath}
\begin{theorem}\label{thm:sir}
Let $\ESemig$ be a \pocs{} and let $\mathcal F$ be an $\Semig$-omissible family of $\Semig$-edge-labelled cycles containing all disobedient ones. Assume that $\MF$ is a \Fraisse{} class and let $\mathbb M$ be its \Fraisse{} limit. Then the following ternary relation $\ind$ on finite subsets of $\mathbb M$ is a local stationary independence relation, where $A\ind_C B$ if and only if $C$ is nonempty and for every $a\in A$ and every $b\in B$ it holds that $d_\mathbb M(a,b) = \inf\{d(a,c)\oplus d(b,c) : c\in C\}$.

Moreover, if $\Semig$ contains a maximum element then this definition makes sense also for empty $C$ and hence $\ind$ is a (non-local) stationary independence relation.
\end{theorem}
\begin{proof}
It is easy to check that this relation indeed satisfies Definition~\ref{defn:sir}.
\end{proof}

Very recently, Evans, Hubi\v cka, Li and the author~\cite{Evanssimplicity} proved that for every non-trivial finite archimedean \pocs{} $\Semig$ the \Fraisse{} limits of all $\MF$ classes where the shortest path completion uses no non-trivial infima have simple automorphism groups. In particular, this implies that the automorphism groups of all finite diameter primitive metrically homogeneous graphs (see Section~\ref{sec:methom}) are simple. This should be taken as further evidence that the concept of semigroup-valued metric spaces is a useful one.

\medskip

Given a homogeneous $L$-edge-labelled graph with a (local) SIR, one can recover a commutative semigroup on $L$ from the SIR by looking at pairs of edges independent over a vertex. Since Theorem~\ref{thm:sir} states that every semigroup-valued metric space studied in this thesis has a SIR, a natural question is whether every homogeneous $L$-edge-labelled graph is a semigroup-valued metric space which satisfies the conditions of Theorem~\ref{thm:shortestpath}. However, the answer to this question is negative:
\begin{example}\label{ex:sharpurysohn}
Let $\mathbb B$ be the countable homogeneous structure with one equivalence relation with two equivalence classes (understood as a complete $\{N,E\}$-edge-labelled graph where $E$ means ``equivalent''). This structure has a local SIR, but is not a semigroup-valued metric space.

Let $\mathcal M^<$ be the class of all finite complete $\mathbb Q^{>0}$-valued metric spaces which moreover contain no geodesic triangles (that is, for every triangle with distances $a,b,c$ it holds that $a>b+c$). This is a \Fraisse{} class (and even has EPPA and is Ramsey with a linear order), but its \Fraisse{} limit has no (local) SIR.
\end{example}
\begin{remark}
On the other hand, the class of all finite complete $\mathbb Z^{>0}$-valued metric spaces which contain no geodesic triangles is in fact a class of semigroup-valued metric spaces where the operation is $a\oplus b = a+b-1$ and the order is the standard order.
\end{remark}

This gives rise to the following question:
\begin{question}
Under what circumstances does a strong amalgamation class of edge-labelled graphs whose \Fraisse{} limit has a local SIR admit an interpretation as a semigroup-valued metric space? Is it possible to explicitly define the order using the SIR? How to get the corresponding omissible family $\mathcal F$?
\end{question}
Conjecture~\ref{conj:universal} says that finite language and primitivity are sufficient.

\chapter{Conclusion and open problems}\label{ch:conclusion}
We found a common generalisation of results of Hubi\v cka and Ne\v set\v ril~\cite{Hubicka2016} on Ramsey expansions of $S$-metric spaces, results of Hubi\v cka, Ne\v set\v ril and the author~\cite{Hubicka2017sauerconnant} on Ramsey expansions of metric spaces with values from a linearly ordered monoid, results of Aranda, Bradley-Williams, Hubi{\v c}ka, Karamanlis, Kompatscher, Pawliuk and the author~\cite{Aranda2017, Aranda2017a, Aranda2017c} on Ramsey expansions and EPPA for metrically homogeneous graphs and Braunfeld's~\cite{Sam} results on Ramsey expansions of $\Lambda$-ultrametric spaces with strong amalgamation. These examples include also classes studied by Solecki~\cite{solecki2005}, Ne\v set\v ril~\cite{Nevsetvril2007}, Vershik~\cite{vershik2008}, Nguyen Van Th\'e~\cite{The2010}, Conant~\cite{Conant2015}, Ma\v sulovi\' c~\cite{Masulovic2017metric}, Soki\'c~\cite{Sokic2017}, Coulson~\cite{Coulson} and others. As a corollary, we also solved several open problems (e.g. EPPA for $S$-metric spaces, $\Lambda$-ultrametric spaces or Conant's generalized non-semi-archimedean metric spaces).

The unifying property of all the aforementioned classes of structures is that they admit some form of the shortest path completion. In this paper we tried to extract some properties which ensure that the shortest path completion works. We conclude with some more open questions and conjectures.

\medskip
While we do have some partial examples, it is not clear to which extent the conditions on the semigroup $\Semig$ and family $\mathcal F$, which we presented here, are necessary.
\begin{question}
Let $\Semig$ be a \pocs{} and $\mathcal F$ a family of $\Semig$-edge-labelled cycles. What are the necessary and sufficient conditions for $\mathcal M_\Semig \cap \Forb(\mathcal F)$ admitting a shortest path completion as in this paper? What are the conditions for it having a (precompact) Ramsey expansion? What do these expansions look like?
\end{question}

\paragraph{Bipartiteness.} There are two extremal variants of the 3-constrained metrically homogeneous graphs which our machinery does not cover, namely the bipartite ones and the antipodal ones. It is reasonable to ask whether one can get similar (or even more general) properties also for the semigroup-valued metric spaces, for example by considering an infinite semigroup and then identifying some distances after doing the shortest path completion.

The question on bipartiteness can be generalised. The author has a computer program which can find all amalgamation classes of complete $\{A,B,C,D,E\}$-edge-labelled graphs which are given by forbidding triangles (that is, a generalisation of~\cite[Appendix]{Cherlin1998}). If we only consider the strong amalgamation ones, it seems that the structure of definable equivalences is not too wild. In particular, it gives rise to the following conjecture and question.
\begin{conjecture}\label{con:stablestructure}
Let $L$ be a finite set and let $\mathcal C$ be a strong amalgamation class of $L$-edge-labelled graphs such that $\mathcal C = \Forb(\mathcal T)$ for $\mathcal T$ a family of $L$-edge-labelled triangles. (In other words, $\mathcal C$ is a triangle constrained strong amalgamation class with finitely many 2-types.) Then $\mathcal C$, when expanded by quotients of $0$-definable equivalence relations on vertices, has \emph{weak elimination of imaginaries} (its \Fraisse{} limit does).
\end{conjecture}
Informally, Conjecture~\ref{con:stablestructure} says that for finite symmetric binary relational languages, the ``only definable equivalences one needs to understand'' are $0$-definable equivalences on vertices. There are, of course, other definable equivalences such as the ones with finite equivalence classes or, say, definable equivalences on tuples which are products of equivalences on vertices, or also restrictions of $0$-definable equivalences to some $A$-types. These do not seem to be important for the applications which we have in mind.

It is easy to see that the zero-definable equivalences form a lattice (meets correspond to conjunctions, joins are their transitive closure (here we are using the fact that there are only finitely many 2-types). Braunfeld proved~\cite[Lemma~4.5]{braunfeld2016lattice} that if the equivalences satisfy the \emph{infinite index property} (that is, if $E_1$ and $E_2$ are definable equivalences such that $E_1\subseteq E_2$ then there is no bound finite on the number of different equivalence classes of $E_1$ within an equivalence class of $E_2$) then the lattice is distributive, and found an example of a non-distributive lattice of definable equivalences~\cite[Example~7]{SamPhD}. It is still worth asking if one can say anything more about the lattice in the setting of Conjecture~\ref{con:stablestructure}.

Conjecture~\ref{conj:universal} from Chapter~\ref{ch:preface} states that every strong amalgamation class in a finite language with a primitive \Fraisse{} limit fits into our framework. We have already seen examples that the finiteness of the language is necessary (such as the class of sharp metric spaces from Example~\ref{ex:sharpurysohn}). In the same example we also discuss the equivalence with two equivalence classes which shows that one cannot drop the primitivity assumption completely. To generalise Conjecture~\ref{conj:universal}, one could ask the same question with primitivity exchanged for the infinite index property. We do not have any counterexamples for this generalization.


 \bibliographystyle{alpha}

\renewcommand{\bibname}{Bibliography}


\bibliography{./bibliography.bib}


%
%

\appendix
\chapter{Vocabulary}\label{ap:vocabulary}

\begin{description}[style=nextline]
\item[$L$-edge-labelled graph] A graph such that each edge has a label from the set $L$. Can also be understood as the pair $(V,\ell)$, where $V$ is the vertex set and $\ell\colon{V\choose 2}\to L$ is a partial function. See Definition~\ref{defn:graphs}.

\item[Shortest path completion] Given an $\Semig$-edge-labelled graph $\str G$ for a \pocs{} $\Semig$, the shortest path completion sets each non-edge to be the infimum of the $\Semig$-lengths of all paths connecting the two vertices. A family of $\Semig$-edge-labelled cycles $\mathcal F$ is then introduced to ensure that this completion exists and that it has nice properties. See Definition~\ref{defn:shortestpath}.

\item[Block] A maximal archimedean subsemigroup of $\Semig$, or $\str 0$. Corresponds to definable equivalences on the $\Semig$-valued metric spaces. See Definition~\ref{defn:block}.

\item[Omissible family] A family of $\Semig$-edge-labelled cycles $\mathcal F$ which behaves nicely with respect to the shortest path completion (that is, the shortest path completion of a graph $\str G$ is in $\Forb(\mathcal F)$ if and only if $\str G$ is, furthermore $\mathcal F$ contains only non-geodesic $\Semig$-metric cycles). See Definition~\ref{defn:omissible}.

\item[Disobedient cycles] A family $\mathcal F$ contains all disobedient cycles if, in the graphs from $\Forb(\mathcal F)$, all infima for the shortest path completion are defined and furthermore distribute with $\oplus$. See Definition~\ref{defn:disobedient}.

\item[Confined family] A family $\mathcal F$ is confined if it is $S$-locally finite for every finite $S\subseteq \Semig$, which means that there are only finitely many $S$-edge-labelled cycles in $\mathcal F$. See Definition~\ref{defn:slocallyfinite}.

\item[Meet synchronization] A family $\mathcal F$ synchronizes meets if, in the situations one encounters in graphs from $\Forb(\mathcal F)$, the block in which the infimum of some set of distances lies is the meet of the blocks in which the distances lie. See Definition~\ref{defn:meetsync}.

\item[mus] An approximation of the (possibly nonexistent) maximum of a block relative to a finite set of distances $S$. Implies that if we have a non-$\Semig$-metric cycle with distances from $S$ then only a bounded number of distances from each block are important, the rest of them only represent the block. See Section~\ref{sec:nonimportant}.

\item[Original and ball vertices] If $\str A$ is an $L^\star_\Semig$ expansion of an $\Semig$-valued metric space, it contains two kinds of vertices: The original vertices which are in the distance relations and newly added ball vertices which represent the balls of irreducible diameters. See Definition~\ref{defn:mstar}.
\end{description}
\chapter{List of classes}\label{ap:classes}
\begin{description}[style=nextline]
\item[$\bm{\Forb(\mathcal F)}$] If $\mathcal F$ is a family of finite $L$-structures, then $\Forb(\mathcal F)$ is the class of all finite $L$-structures (or, if it is clear from the context, $L^+\supset L$-structures) $\str A$ such that there is no $\str F\in \mathcal F$ with a homomorphism $\str F\rightarrow \str A$. See the beginning of Chapter~\ref{ch:background}.

\item[$\bm{\mathcal M_\Semig}$] The class of all finite $\Semig$-valued metric spaces. See Definition~\ref{def:Mmetric}.

\item[$\bm{\mathcal M_\Semig\cap \Forb(\mathcal F)}$] Expectably, the class of all finite $\Semig$-valued metric spaces containing no homomorphic images of members of $\mathcal F$.

\item[$\bm{\overrightarrow{\mathcal M}_\Semig\cap \Forb(\mathcal F)}$] The class of all convexly ordered $\Semig$-valued metric spaces omitting homomorphic images from $\mathcal F$, see Definition~\ref{defn:Js}. They differ from $\mathcal M_\Semig$ by adding a partial orders $\leq^\Block$ for each non-maximal meet-irreducible block of $\Semig$, from these it is possible to define linear orders of balls of every diameter.

\item[$\bm{\classstarf}$] See Definition~\ref{defn:mstar}. This is a class of all images of members of $\mathcal M_\Semig\cap \Forb(\mathcal F)$ under the $L^\star$ functor. This means that we add ball vertices for each block in $I_\Semig$ and link the original and the ball vertices by several different unary functions, in particular to each original vertex we link the ball vertices representing the balls it lies in. This makes it possible to describe the non-metric cycles using structures of bounded size. Remember that $\mathcal M_\Semig\cap \Forb(\mathcal F)$ and $\classstarf$ are isomorphic as categories (Proposition~\ref{prop:nonorderisomorphic}).

\item[$\bm{\Hstarf}$] The ``hereditary closure'' of $\classstarf$ (see Definition~\ref{defn:mstar}). Namely, we allow for the structures in $\Hstarf$ to contain ball vertices to which no original vertex is linked. We need this class because Theorem~\ref{thm:hn} requires a hereditary class, but clearly each member of $\Hstarf$ can be completed to $\classstarf$.

\item[$\bm{\GstarS}$] A class of \textbf{incomplete} $L^\star_\Semig$-structures. All members of $\GstarS$ have a completion in $\classstarf$, on the other hand $\GstarS$ contains all free amalgams of structures from $\Hstarf$ (with distances from $S$) as well as structures which we need to complete for the Ramsey property and EPPA. See Definition~\ref{defn:gstars}.

\item[$\bm{\classstarleqf}$] This class is to $\overrightarrow{\mathcal M}_\Semig\cap \Forb(\mathcal F)$ what $\classstarf$ is to $\mathcal M_\Semig\cap \Forb(\mathcal F)$. We added the ball vertices and also a linear order which corresponds to the partial orders $\leq^\Block$ in $\overrightarrow{\mathcal M}_\Semig\cap \Forb(\mathcal F)$. Again, $\overrightarrow{\mathcal M}_\Semig\cap \Forb(\mathcal F)$ and $\classstarleqf$ are isomorphic as categories. See Definition~\ref{defn:mstarleq}.

\item[$\bm{\Hstarleqf}$] The ``hereditary closure'' of $\classstarleqf$, see Definition~\ref{defn:mstarleq}.

\end{description}

\openright
\end{document}